\documentclass[10pt]{amsart}

\usepackage{amsmath,amsfonts,amssymb,dsfont}
\usepackage{color, ulem}
\usepackage{graphicx, epsfig}
\usepackage{xspace}

\newcommand{\dir}{}

\newcommand\N{\mathbb N}
\newcommand\R{\mathbb R}
\newcommand{\bE}{{\boldsymbol E}}
\newcommand{\bX}{{\boldsymbol X}}
\newcommand{\bY}{{\boldsymbol Y}}
\newcommand{\bH}{{\boldsymbol H}}
\newcommand{\bu}{\boldsymbol u}
\newcommand{\bv}{\boldsymbol v}

\newcommand{\bF}{\boldsymbol F}
\newcommand{\bG}{\boldsymbol G}

\newcommand{\opmax}{\mathcal{M}}

\newcommand\dist{\operatorname{dist}}
\newcommand\spec{\sigma}
\renewcommand\div{\operatorname{div}}

\newcommand\x{\times}

\newcommand\llll{|\kern-1pt|\kern-1pt|}

\renewcommand{\L}{\mathcal{L}}
\newcommand{\E}{\mathcal{E}}

\newcommand{\dom}{\operatorname{D}}

\newcommand{\btau}{\mathbf{t}}

\newcommand{\sqr}{\mathrm{sqr}}
\newcommand{\lsp}{\mathrm{L}}
\newcommand{\lstd}{\mathrm{F}}
\newcommand{\cbe}{\mathrm{cbe}}
\newcommand{\sla}{\mathrm{sla}}

\newcommand\1{{\ensuremath {\mathds 1} }}

\theoremstyle{remark}
\newtheorem{remark}{Remark}

\theoremstyle{plain}
\newtheorem{algorithm}{Procedure}

\newtheorem{theorem}{Theorem}

\newtheorem{lemma}[theorem]{Lemma}

\DeclareMathOperator{\tr}{Tr}

\newcommand\curl{\operatorname{curl}}
\newcommand\nn{\mathbf{n}}

\newcommand\RR{\mathbb R}

\newcommand{\sobol}{\mathcal{H}}

\newcommand{\up}{\mathrm{up}}
\newcommand{\low}{\mathrm{low}}

\newcommand{\bomega}{\partial \Omega}       

\begin{document}

\title[Maxwell eigenvalue enclosures]{Finite element eigenvalue enclosures for the Maxwell operator}

\author[G.R. Barrenechea]{Gabriel R. Barrenechea}
\address{Department of Mathematics and Statistics, University of Strathclyde, 26 Richmond
Street, Glasgow G1 1XH, Scotland} 
\email{gabriel.barrenechea@strath.ac.uk}

\author[L. Boulton]{Lyonell Boulton}
\address{Department of Mathematics and Maxwell Institute for Mathematical
Sciences, Heriot-Watt University, Edinburgh, EH14 4AS, UK}
\email{L.Boulton@hw.ac.uk}

\author[N. Boussaid]{Nabile Boussa{\"\i}d}
\address{D\'epartement de Math\'ematiques, Universit\'e de Franche-Comt\'e,
Besan\c con, France}
\email{nboussai@univ-fcomte.fr}


\begin{abstract}
 We propose employing the extension of the Lehmann-Maehly-Goerisch method developed by Zimmermann and Mertins, as a highly effective tool for the pollution-free finite element computation of the eigenfrequencies of the resonant cavity problem on a bounded region. This method gives complementary bounds for the eigenfrequencies which are adjacent to a given parameter $t\in \mathbb{R}$. We present a concrete numerical scheme which provides certified enclosures in a suitable asymptotic regime. We illustrate the applicability of this scheme by means of some numerical experiments on benchmark data using Lagrange elements and unstructured meshes. 
\end{abstract}

\date{14th January 2014}
\maketitle
\tableofcontents


\section{Introduction}

The framework developed by Zimmermann and Mertins \cite{ZM95} which generalizes the Lehmann-Maehly-Goerisch method \cite{1985Goerisch2,1980Goerisch,1949Lehmann,1950Lehmann,1952Maehly} (also \cite[Chapter~4.11]{1974Weinberger}), is a reliable tool for the numerical computation of bounds for the eigenvalues of linear operators in the spectral pollution regime \cite{Behnke:2009p3097,BouStr:2011man,theoretical}. In its most basic formulation \cite{davies-plum,theoretical,BoHo2013}, this framework relies on fixing a parameter $t\in \mathbb{R}$ and then characterizing the spectrum which is adjacent to $t$ by means of a combination of the Variational Principle with the Spectral Mapping Theorem. In the present paper we show that
this formulation can be effectively implemented for computing sharp estimates for the angular frequencies and electromagnetic field phasors of the resonant cavity problem by means of the finite element method.

Let $\Omega\subset \RR^3$  be a polyhedron. Denote by $\bomega$ the boundary of this region and by $\nn$ its outer normal vector.  Consider the anisotropic Maxwell eigenvalue problem: find $\omega\in \R$ and $(\bE,\bH) \not=0$ such that
\begin{equation} \label{maxwell}
\left\{
\begin{aligned}
 & 
 \begin{aligned} &\curl \bE = i\omega\mu \bH  \\
 & \curl \bH = -i\omega\epsilon \bE 
 \end{aligned}  & \text{in }\Omega  \\
 & \bE\times\nn =0& \text{on } \bomega.
 \end{aligned} \right. 
\end{equation}
The physical phenomenon of electromagnetic oscillations in a resonator is described by  \eqref{maxwell}, assuming that the  field phasor satisfies Gauss's law
\begin{equation}   \label{ansatz_div}
\div (\epsilon \bE)=0=\div (\mu \bH)  \quad \text{in }\Omega. 
\end{equation}
Here $\epsilon$ and $\mu$, respectively, are the given electric permittivity and magnetic permeability at each point of the resonator. 

The orthogonal complement in a suitable inner product \cite{1990Birman} of the solenoidal space \eqref{ansatz_div}  is the gradient space. This gradient space has infinite dimension and is part of the kernel of the densely defined linear self-adjoint operator 
\[
\opmax:\dom (\opmax)\longrightarrow L^2(\Omega)^6
\] 
associated to \eqref{maxwell}. In turns, this means that \eqref{maxwell}-\eqref{ansatz_div} and the unrestricted problem \eqref{maxwell}, have exactly the same non-zero spectrum and exactly the same eigenvectors orthogonal to the kernel. For general data, the numerical computation of $\omega$ by means of the finite element method is extremely challenging, due to a combination of variational collapse ($\opmax$ is strongly indefinite) and the fact that finite element bases seldom satisfy the ansatz \eqref{ansatz_div}. 

Several ingenious methods for the finite element treatment of the eigenproblem 
\eqref{maxwell}-\eqref{ansatz_div} have been developed in the recent past. 
Perhaps the most effective among these methods \cite{BFGP1999, Boffi-Act-Num} 
consists in re-writing the spectral problem associated to $\opmax^2$ in a mixed 
form and employing edge elements. This turns out to be linked to deep 
mathematical ideas on the rigorous treatment of finite elements 
\cite{Arnold:2010p3067} and it is at the core of an elegant geometrical 
framework. Other approaches include, \cite{Bramble05} combining nodal elements 
with a least squares formulation of  \eqref{maxwell}-\eqref{ansatz_div} 
re-written in weak form, \cite{BCJ09} employing continuous finite element spaces 
of Taylor-Hood-type  by coupling  \eqref{maxwell} with \eqref{ansatz_div} via a 
Lagrange multiplier, and \cite{BG11} enhancing the divergence of the electric 
field in a fractional order negative Sobolev norm.

In spite of the fact that some of these techniques are convergent, unfortunately, none of them provides  \textit{a priori}  guaranteed one-sided bounds for the exact eigenfrequencies. In turns, detecting the presence of a spectral cluster (or even detecting multiplicities) is extremely difficult. Below we argue that the most basic formulation of the pollution-free technique described in \cite{ZM95} can be successfully implemented for determining certified upper and lower bounds for the eigenfrequencies and 
corresponding approximated field phasors of \eqref{maxwell}. Remarkably the classical family of nodal finite elements renders sharp numerical approximations.

In Section~\ref{settings} we fix the rigorous setting of the self-adjoint 
operator $\opmax$ and set 
our concrete assumptions on the data of the problem. For these concrete 
assumptions we consider both a region $\Omega$ with and without cylindrical 
symmetries, generally non-convex and not even Lipschitz. In 
Section~\ref{feceb} we describe the finite element realization of the 
computation of complementary eigenvalue bounds. Based on this realization, in 
Section~\ref{numstratnut} an algorithm providing certified
eigenvalue enclosures in a given interval is presented and analyzed. This 
algorithm is then implemented and its results are reported in 
Sections~\ref{convex-domains}-\ref{transmission}.


\section{Abstract setting of the Maxwell eigenvalue problem}
\label{settings}
\subsection{Concrete assumptions on the data}  
\label{Assumptions}
The concrete assumptions on the data of  equation \eqref{maxwell} made below are as follows.
The polyhedron $\Omega\subset \mathbb{R}^3$ will always be open, bounded and simply connected.
The permittivities will always be such that
 \begin{equation} \label{bdd_away_from_0}
 \epsilon,\,\frac 1\epsilon,\,\mu,\,\frac 1\mu \in L^\infty(\Omega).
\end{equation}

Without further mention, the non-zero spectrum of  $\opmax$ will be assumed to be purely discrete and it does not accumulate at $\omega=0$. This hypothesis is verified, for example, whenever $\Omega$ is a polyhedron with  a Lipschitz boundary, \cite[Corollary 3.49]{Monk2003} and \cite[Lemma 
1.3]{1990Birman}. A more systematic analysis of the spectral properties 
of $\opmax$ on more general regions $\Omega$ is being carried out elsewhere 
 \cite{BBBP}.

\subsection{The self-adjoint Maxwell operator}
\label{3dmaxwell}
We follow closely \cite{1990Birman}. Let
\[
\begin{aligned}
   \sobol(\curl;\Omega)&=\left\{\bu \in L^2(\Omega)^3 : \curl \bu \in
L^2(\Omega)^3 \right\} \\
\sobol_0(\curl;\Omega) 
& = \{ \bu \in \sobol(\curl;\Omega) :\int_\Omega \curl \bu \cdot \bv =
\int_\Omega \bu \cdot \curl \bv
 \quad \forall \bv \in \sobol(\curl;\Omega)\}.
\end{aligned}
\]
The linear space $\sobol(\curl;\Omega)$ becomes a Hilbert space for the norm
\begin{equation*} 
    \| \bu \|_{\curl,\Omega}=\left(\|\bu\|_{0,\Omega}^2+\|\curl\bu\|_{0,\Omega}^2\right)^{1/2},
\end{equation*}
where 
\[
    \|\bv\|_{0,\Omega}=\left( \int_\Omega |  \bv|^2 \right)^{1/2}
\]
is the corresponding norm of $L^2(\Omega)^3$.
By virtue of Green's identity for the rotational \cite[Theorem~I.2.11]{1986Giraultetal}, if $\Omega$ is a Lipschitz
domain \cite[Notation~2.1]{ABDG98}, then $\bu\in
\sobol_0(\curl;\Omega)$ if and only if $\bu\in \sobol(\curl;\Omega)$ and
$\bu\times\nn ={\mathbf 0} \;\mathrm{on}\;\bomega$. 
Moreover 
\begin{equation}    \label{closure}
\sobol_0(\curl;\Omega)^3=\overline{C^\infty_0(\Omega)^3},
\end{equation}
where the closure is in the norm $\|\cdot\|_{\curl,\Omega}$.

A domain of self-adjointness of the operator associated to \eqref{maxwell} for
$\epsilon=\mu=1$ is \[\mathcal{D}_1=\sobol_0(\curl;\Omega)\times \sobol(\curl;\Omega)\subset L^2(\Omega)^6\]
and its action is given by 
\[
\opmax_1 =\begin{bmatrix} 0 & i \curl \\ -i \curl
& 0\end{bmatrix} : \mathcal{D}_1 \longrightarrow L^2(\Omega)^6.
\]
Let 
\[
      \mathcal{P}=\begin{bmatrix}\epsilon^{1/2} I_{3\times 3}& 0\\
0 & \mu^{1/2} I_{3\times 3} \end{bmatrix}.
\]
Condition \eqref{bdd_away_from_0} ensures that $\mathcal{P}:L^2(\Omega)^6\longrightarrow L^2(\Omega)^6$ is bounded and invertible. Moreover, \[\left(\omega,\begin{bmatrix}\bE  \\ \bH\end{bmatrix} \right)\in \mathbb{R}\times \mathcal{D}_1\] is a solution of \eqref{maxwell}, if and only if
\[\begin{bmatrix}\tilde{\bE} \\ \tilde{\bH}\end{bmatrix}=\mathcal{P}\begin{bmatrix}\bE \\ \bH\end{bmatrix}\] is  a solution of
\[
     \mathcal{P}^{-1}\opmax_1 \mathcal{P}^{-1}\begin{bmatrix} \tilde\bE \\ \tilde\bH \end{bmatrix} =\omega \begin{bmatrix} \tilde\bE \\ \tilde\bH \end{bmatrix}.
\] 
Therefore $\opmax=\mathcal{P}^{-1}\opmax_1 \mathcal{P}^{-1}$ on $\dom(\opmax)=\mathcal{P}\mathcal{D}_1$ is the self-adjoint operator associated to \eqref{maxwell}. 

As $\opmax$ anticommutes with complex conjugation, the spectrum 
is symmetric with respect to $0$. Moreover, $\ker(\opmax)$ is infinite dimensional, because it always contains the gradient space, see~\cite{1990Birman}.

\subsection{Isotropic cylindrical symmetries} \label{cylindrical}

If $\Omega=\tilde{\Omega}\times (0,\pi)$ for $\tilde{\Omega}\subset \mathbb{R}^2$ an open simply connected polygon, then \eqref{maxwell} decouples by separating the variables for  $\epsilon=\mu=1$. In turns, a non-zero $\omega$ is an eigenvalue of $\opmax_1$, if and only if either $\omega^2=\lambda^2$  where $\lambda^2$ is a Dirichlet eigenvalue of the Laplacian in $\tilde{\Omega}$, or $\omega^2=\nu^2+\rho^2$ where $\nu^2$ is a non-zero Neumann eigenvalue of the Laplacian in $\tilde{\Omega}$ and $\rho\in \mathbb{N}$. 

The Neumann problem  can be re-written as ($\nu=\omega$)
\begin{equation}  \label{maxwell_2d}      
 \left\{
\begin{aligned}
& \begin{aligned}&\curl \bE= i\omega H \\ & \curl H=
-i\omega\bE  \end{aligned}    & \text{in }\tilde{\Omega}  \\
 & \bE \cdot \btau = \mathbf{0}& \text{on } \partial \tilde{\Omega}\,,
 \end{aligned}\right. 
 \end{equation}
for \[\left(\omega,\begin{bmatrix}\bE \\ H\end{bmatrix}\right)\in \R\times (\tilde{\mathcal{D}}_1\setminus\{0\}).\] 
Here 
\[
 \bE=\begin{bmatrix}E_1 \\ E_2 \end{bmatrix}, \qquad \curl \bE=\partial_x E_2-\partial_y E_1, \qquad 
\curl H=  \begin{bmatrix}\partial_y H \\ -\partial_x H\end{bmatrix},
\]
$\btau$ is the unit tangent to $\partial \tilde{\Omega}$ and 
\[
   \tilde{\mathcal{D}}_1\!=\!
   \left\{\!\bu \in L^2(\Omega)^2 : \curl \bu \in
L^2(\Omega)\text{ and }\bu \cdot \btau = \mathbf{0}\! \right\} \times
\left\{\!u \in L^2(\Omega) : \curl u \in
L^2(\Omega)^2\!\right\}. 
\]
This two-dimensional Maxwell problem exhibits all the complications concerning spectral pollution as its three-dimensional counterpart. 

We denote by $\tilde{\opmax}:\tilde{\mathcal{D}}\longrightarrow L^2(\tilde{\Omega})^3$ the self-adjoint operator associated to \eqref{maxwell_2d}. This operator has often been employed for tests which can then be validated against numerical calculations for the original Neumann Laplacian via the Galerkin method, \cite{2004Dauge}. Note that the latter is a semi-definite operator with a compact resolvent, so it does not exhibit spectral pollution.


\section{Finite element computation of the eigenvalue bounds}  \label{feceb}

The basic setting of the general method proposed in \cite{ZM95} is achieved by deriving eigenvalue bounds directly from  \cite[Theorem~1.1]{ZM95}, as described in \cite[Section~6]{davies-plum} and \cite{theoretical}.  We will see next that, from this setting, a general finite element scheme for computing guaranteed bounds for the eigenvalues of $\opmax$ which are in the vicinity of a given non-zero $t\in \mathbb{R}$ can be established.  

\subsection{Formulation of the weak problem and eigenvalue bounds} \label{weak_prob}
Let $\{\mathcal{T}_h\}_{h>0}$ be a family of shape-regular \cite{EG04} triangulations of $\overline{\Omega}$, where the elements
$K\in {\mathcal{T}}_h$ are simplexes  with diameter $h_K$ and $h=\max_{K\in{\mathcal{T}}_h}h_K$. 
For $r\ge 1$, let 
\begin{align*}
    \mathbf{V}_h^r &=\{\bv_h\in C^0(\overline{\Omega})^3: \bv_h|_K \in
\mathbb{P}_r(K)^3 \ 
    \forall K\in \mathcal{T}_h \} \\
    \mathbf{V}_{h,0}^r &=\{\bv_h\in \mathbf{V}_h^r: \bv_h\x\nn ={\mathbf 0}
\;\textrm{on}\;\bomega \}.
\end{align*}
Then
\begin{equation}\label{fe-space}
 \L\equiv \L_{h}=\mathbf{V}_{h,0}^r\x \mathbf{V}_h^r \subset \mathcal{D}_1.
\end{equation}
For $t\in \mathbb{R}$, let $\mathfrak{m}^p_t:\mathcal{D}_1\times \mathcal{D}_1 \longrightarrow \mathbb{C}$ be given by
\[
    \begin{aligned}
    \mathfrak{m}^1_t\left(  \begin{bmatrix}\bE \\ \bH \end{bmatrix}, \begin{bmatrix}\bF \\ \bG \end{bmatrix} \right)&=
    \int_\Omega \left((\opmax_1-t\mathcal{P}^2) \begin{bmatrix}\bE \\ \bH \end{bmatrix} \right) \cdot \begin{bmatrix}\bF \\ \bG \end{bmatrix} 
    \\
    \mathfrak{m}^2_t\left(  \begin{bmatrix}\bE \\ \bH \end{bmatrix}, \begin{bmatrix}\bF \\ \bG \end{bmatrix} \right)&=
    \int_\Omega \left((\mathcal{P}^{-1}\opmax_1-t\mathcal{P}) \begin{bmatrix}\bE \\ \bH \end{bmatrix}\right) \cdot 
    \left((\mathcal{P}^{-1}\opmax_1-t\mathcal{P}) \begin{bmatrix}\bF \\ \bG \end{bmatrix} \right)
    \end{aligned}
\]
The following weak eigenvalue problem \cite{ZM95,davies-plum,theoretical} plays a central role below: 
\begin{equation}   \label{weak}
\begin{aligned}
&\text{find }
\left(\tau,\begin{bmatrix}\bE  \\ \bH\end{bmatrix} \right)\in \mathbb{R}\times (\L\setminus \{0\}) \text{ such that } \\
    &\mathfrak{m}^1_t\left(  \begin{bmatrix}\bE \\ \bH \end{bmatrix}, \begin{bmatrix}\bF \\ \bG \end{bmatrix} \right) =
    \tau \mathfrak{m}^2_t\left(  \begin{bmatrix}\bE \\ \bH \end{bmatrix}, \begin{bmatrix}\bF \\ \bG \end{bmatrix} \right)
    \qquad \forall \begin{bmatrix}\bF \\ \bG \end{bmatrix}\in \L .
    \end{aligned}
\end{equation}

Let $m^{\pm}(t)\equiv m^\pm(t,h)$ be the number of negative and positive eigenvalues of \eqref{weak}, respectively. Let
$\tau^{\pm}_j(t)\equiv  \tau^{\pm}_j(t,h)$,
\[\tau^-_1(t) \leq \ldots \leq \tau^-(t)_{m^-(t)}\] be the negative eigenvalues of \eqref{weak} and
\[\tau^+_{m^+(t)}(t)\leq \ldots \leq \tau^+_{1}(t)\] be the positive eigenvalues of \eqref{weak}, if they exist at all.
Let
\[
     \rho_j^\pm(t,h)=t+\frac{1}{\tau_j^\pm(t)}.
\]
As we will see next, the latter quantities provide bounds for the spectrum of $\opmax$ in the vicinity of $t$.

By counting multiplicities, let 
\[
     \ldots \leq \nu_2^-(t) \leq \nu_1^-(t) < t < \nu_1^+(t) \leq \nu_2^+(t) \leq \ldots
\]
be the eigenvalues of $\opmax$ which are adjacent to $t$. That is
$\nu_j^-(t)$ is the $j$-th eigenvalue strictly to the left of $t$ and $\nu_j^+(t)$ is the $j$-th eigenvalue strictly 
to the right of $t$.  The following crucial statement is a direct consequence of \cite[Theorem~2.4]{ZM95} or \cite[Corollary~7]{theoretical} (see also
\cite[Theorem~11]{davies-plum}).

\begin{theorem} \label{bounds}
    Let $t\in\R$. Then
    \[
            \rho_j^-(t,h) \leq \nu_j^-(t) \quad \forall j=1,\ldots,m^-(t) \quad \text{and} \quad \nu_j^+(t) \leq \rho_j^+(t,h)
            \quad \forall j=1,\ldots,m^+(t).
    \]
\end{theorem}

\begin{remark}   \label{2d}
In the case of the lower-dimensional Maxwell operator $\tilde{\opmax_1}$,  the finite element spaces on a
corresponding triangulation $\mathcal{T}_h$ of $\tilde{\Omega}$ are chosen as
\[
\L_h=\left\{\begin{bmatrix}\bu_h \\ v_h \end{bmatrix} \in C^0\left(\overline{\tilde{\Omega}}\right)^{2+1}: \left.\begin{bmatrix}\bu_h \\ v_h \end{bmatrix} \right|_K \in
\mathbb{P}_r(K)^{2+1} \  \forall K\in \mathcal{T}_h \text{ and }  \bu_h\cdot \btau =0 
\text{ on } \partial\tilde{\Omega} \right\}. 
\]
The weak problem analogous to \eqref{weak} and a corresponding version of Theorem~\ref{bounds} (and further statements below) are formulated by substituting $\mathfrak{m}^p_t$ with the corresponding lower-dimensional forms.
\end{remark}

\subsection{Convergence of the eigenvalue bounds}    \label{conv_evb}
According to \cite[Theorem~12]{theoretical}, if $\L$ captures an eigenspace of $\opmax$
within a certain order of precision $\mathcal{O}(\varepsilon)$ for small $\varepsilon$, then
the eigenvalue bounds found in Theorem~\ref{bounds} are within $\mathcal{O}(\varepsilon^2)$. 
We now show a consequence of this statement in the present setting. 

Consider  an open bounded segment  $J\subset\R$, such that $0\not\in J$. Denote by $\E_J$ the eigenspace associated to this segment and assume that $t\in J$. Here and elsewhere the relevant set where the indices $j$ move is
\[
    \mathcal{F}_J^{\pm}(t)=\{j\in \mathbb{N}:\nu_j^\pm(t)\in J\}.
\]

\begin{theorem} \label{order_maxwell}
Let $r\in \mathbb{N}$ be fixed. Then 
\[
 \lim_{h\to 0}     \left| \rho_j^\pm (t,h)-\nu_j^\pm(t)\right|= 0 \qquad \forall j\in \mathcal{F}_J^{\pm}(t).
\]
If in addition $\mathcal{P}^{-1}\E_J \subseteq \mathcal{H}^{r+1}(\Omega)^6$,  then there exist $C_t^\pm\equiv C_t^\pm(r) >0$ such that 
\begin{equation}     \label{order_conv_evalue}
      \left| \rho_j^\pm (t,h)-\nu_j^\pm(t) \right|\leq
C_t^\pm h^{2r}  \qquad \forall j\in \mathcal{F}_J^{\pm}(t)
\end{equation}
for $h$ sufficiently small.
\end{theorem}
\begin{proof}
By combining \cite[Theorem~3.26]{Monk2003} with \eqref{closure} and standard interpolation estimates (cf. \cite{EG04}), it follows that \[\text{ for any } \begin{bmatrix}\bF \\ \bG \end{bmatrix}\in \mathcal{D}_1 \text{ there exists } \begin{bmatrix}\bF_h \\ \bG_h\end{bmatrix} \in \L_h\]
such that
\begin{equation}    \label{approx_max_fe_1}
      \lim_{h\to 0}\Big(  \|\bF-\bF_h\|_{\curl,\Omega} +
\|\bG-\bG_h\|_{\curl,\Omega}\Big) = 0\,.
\end{equation}
Since $\mathcal{P}$ is a bounded operator, then for any
\[\begin{bmatrix}\tilde{\bF} \\ \tilde{\bG}\end{bmatrix}={\mathcal P}\begin{bmatrix}\bF \\ \bG\end{bmatrix}\in \dom(\opmax),\] we have \begin{equation}    \label{approx_max_fe_11}
      \lim_{h\to 0}\left(  \left\|\opmax\begin{bmatrix}
\tilde{\bF}-\tilde{\bF}_h\\
\tilde{\bG}-\tilde{\bG}_h\end{bmatrix}\right\|_{0,\Omega}
      + \left\|\begin{bmatrix} \tilde{\bF}-\tilde{\bF}_h\\
\tilde{\bG}-\tilde{\bG}_h\end{bmatrix}\right\|_{0,\Omega} \right) = 0
\end{equation}
where
\[\begin{bmatrix} \tilde{\bF}_h \\ \tilde{\bG}_h \end{bmatrix}= \mathcal{P} \begin{bmatrix} \bF_h \\ \bG_h \end{bmatrix}  \in \dom(\opmax).\]
In turns, this is exactly the hypothesis required in \cite[Theorem~12]{theoretical} which ensures the claimed statement.  

Let $\mathcal{I}_{r,h}$ denote the interpolation operator associated to the finite element spaces $\mathbf{V}^r_h$ (cf. \cite{EG04}) and let
\[
     \begin{bmatrix} {\bF}_h \\ {\bG}_h \end{bmatrix}=\mathcal{I}_{r,h} \begin{bmatrix} {\bF} \\ {\bG} \end{bmatrix}.
\]
If 
\[
   \begin{bmatrix} \bF \\ \bG \end{bmatrix}\in \mathcal{E}_J \subset \mathcal{H}^{r+1}(\Omega)^6,
\]
then
\begin{equation} \label{approx_max_fe_2}
       \left\|\begin{bmatrix} {\bF}-{\bF}_h\\
{\bG}-{\bG}_h\end{bmatrix}\right\|_{\curl,\Omega}\leq c(r)
h^{r}\left\|\begin{bmatrix}\bF \\ \bG\end{bmatrix} \right\|_{r+1,\Omega} ,
\end{equation}
so that
\begin{equation} \label{approx_max_fe_22}
     \left\| \opmax  \begin{bmatrix} \tilde{\bF}-\tilde{\bF}_h\\
\tilde{\bG}-\tilde{\bG}_h\end{bmatrix}\right\|_{0,\Omega}
       + \left\|\begin{bmatrix} \tilde{\bF}-\tilde{\bF}_h\\
\tilde{\bG}-\tilde{\bG}_h\end{bmatrix}\right\|_{0,\Omega}\leq C h^{r}
\end{equation}
where  $C>0$ is a  constant independent of $h$.
This is precisely the condition \cite[(35)]{theoretical}.
Thus, \cite[Theorem~12]{theoretical} ensures the claimed statement.
\end{proof}

 \subsection{Eigenfunctions}
The statement established in \cite[Corollary~13]{theoretical}, provides an insight on how the eigenspace $\mathcal{E}_J$  in the framework of Theorem~\ref{order_maxwell} is also captured by the trial subspaces $\L$ as $h\to 0$. Let 
\[
\dist_{1}[({\bF},{\bG}),\E] =\inf_{\begin{bmatrix}\bX \\ \bY\end{bmatrix} \in \E}
\left\|  \begin{bmatrix} {\bF}- \bX\\  {\bG}- \bY \end{bmatrix} \right\|_{\curl,\Omega}  
\]
be the Hausdorff distance between a given vector 
\[
\begin{bmatrix}\bF \\ \bG\end{bmatrix}\in \mathcal{D}_1 \qquad \text{ and } \qquad \mathcal{E}\subseteq \mathcal{D}_1.
\]
Denote by  
\[
    \begin{bmatrix}\bF^\pm_{j}(t,h) \\ {\bG}^\pm_{j}(t,h)\end{bmatrix}\in {\L}_h
\] 
the eigenvectors of \eqref{weak} associated to $\tau_j^\pm(t)$ respectively and
assume that 
\[
\left\|\begin{bmatrix}\bF^\pm_{j}(t,h) \\ {\bG}^\pm_{j}(t,h)\end{bmatrix}\right\|_{0,\Omega}=1.
\]
Then, the following result concerning approximation of eigenspaces can be stated.

\begin{theorem} \label{order_maxwell_eigenfunction} 
Let $r\in \mathbb{N}$ be fixed. Then, 
 \[
 \lim_{h\to 0} \dist_{1} [({\bF}^\pm_{j}(t,h) ,
{\bG}^\pm_{j}(t,h)),\E_{J}] = 0.
  \]  
If in addition $\mathcal{P}^{-1}\E_J \subseteq \mathcal{H}^{r+1}(\Omega)^6$,  then there exist 
$C_t^\pm(r) >0$ such that 
\[
 \dist_{1} [({\bF}^\pm_{j}(t,h) ,
{\bG}^\pm_{j}(t,h)),\E_{J}] \leq C^{\pm}_t(r) h^r
\]
for $h$ sufficiently small.
\end{theorem}
\begin{proof}
Proceed as in the proof of Theorem~\ref{order_maxwell} in order to verify the hypotheses of \cite[Corollary~13]{theoretical}.
\end{proof}

Convergence of the eigenvalue bounds in Theorem~\ref{bounds} are therefore ensured, in spite of the fact that $\L$ are spaces of nodal finite elements with no particular mesh structure.
Note that this is guaranteed, even in the case where $\epsilon$ and $\mu$ are rough, however, since $\mathcal{P}^{-1}\E_J \not\subseteq \mathcal{H}^{r+1}(\Omega)^6$ unless these coefficients are smooth themselves, an estimate on the convergence rate in this situation is beyond the scope of Theorem~\ref{order_maxwell_eigenfunction}. For a highly heterogeneous medium, a deterioration of the convergence speed is to be expected.

We remark  that the above analysis relies on the regularity of the eigenspaces associated to the interval $J$ only. Thus, for non-convex $\Omega$, 
this allows the possibility of approximating eigenvalues associated to regular eigenfunctions with high
accuracy, if some a priori information about their location is at hand.


\section{A certified numerical strategy}   \label{numstratnut}

Let us now describe a procedure which, in an asymptotic regime,  
renders small intervals which are guaranteed to contain spectral points.  Convergence will be derived from Theorem~\ref{order_maxwell}.

Denote by $0<t_{\up}<t_{\low}$  the corresponding parameters $t$ 
in the weak problem \eqref{weak}, which are set for
computing $\rho^-_j(t_{\low},h)$ (lower bounds) and $\rho^+_j(t_{\up},h)$ (upper bounds) in the segment $(t_{\up},t_{\low})$.  The  scheme described next aims at finding intervals of enclosure for the eigenvalues of $\opmax$ which lie in this segment, for a prescribed tolerance set by the parameter $\delta>0$. According to Lemma~\ref{lem_certified} below, these intervals will be certified in the regime $\delta\to 0$.

\begin{algorithm} \label{alg3}
 \
 
 \begin{itemize}
 \item[] \underline{Input}. 
 \begin{itemize}
 \item Initial $t_{\up}>0$.
 \item Initial $t_{\low}>t_{\up}$ such that $t_{\low}-t_{\up}$ is fairly large.  
 \item A sub-family $\mathcal{F}$ of finite element spaces ${\mathcal{L}}_h$ as in
\eqref{fe-space}, dense as $h\to 0$. 
 \item A tolerance $\delta>0$ fairly small compared with $t_{\low}-t_{\up}$.
 \end{itemize}
\item[] \underline{Output}. 
\begin{itemize}
\item A prediction $\tilde{m}(\delta)\in \mathbb{N}$  of $\tr \1_{(t_{\up},t_{\low})}(\opmax)$.
 \item Predictions $\omega_{j,\delta}^\pm$ of the endpoints of enclosures for the eigenvalues in  $\spec(\opmax)\cap (t_{\up},t_{\low})$, such that
  $0<\omega_{j,\delta}^+- \omega_{j,\delta}^-< \delta$ for $j=1,\ldots,\tilde{m}(\delta)$.
\end{itemize}
\item[] \underline{Steps}. 
\begin{enumerate}
\item Set initial ${\mathcal{L}}_h\in \mathcal{F}$.
\item \label{alg3a} While 
\[
\rho^+_{j,h} -\rho^-_{j,h} \geq \delta  \text{ or } \rho_{j,h}^->\rho_{j,h}^+ \text{ for some } j=1,\ldots,\tilde{m},
\]
 do \ref{alg3b} - \ref{alg3d}.  
\item \label{alg3b} Compute
\[
      \rho_{j,h}^+=\rho_j^+(t_{\up},h) \qquad \text{for} \qquad
j=1,\ldots,\tilde{m}_{\up}
\]
where $\tilde{m}_{\up}$ is such that  $\rho_{\tilde{m}_{\up},h}^+<t_{\low}$ and
\[\rho_{\tilde{m}_{\up}+1}^+(t_{\up},h)\geq t_{\low}.\] 
\item \label{alg3c}  Compute
\[
      \rho_{\tilde{m}_{\low}-k+1,h}^-=\rho_{k}^-(t_{\low},h) \qquad \text{for}
\qquad k=1,\ldots,\tilde{m}_{\low}
\]
where $\tilde{m}_{\low}$ is such that $\rho_{\tilde{m}_{\low},h}^->t_{\up}$ and
\[
\rho_{\tilde{m}_{\low}+1}^-(t_{\low},h)\leq t_{\up}.
\] 
\item \label{alg3d} If $\tilde{m}_{\low}\not=\tilde{m}_{\up}$,  decrease $h$, set new $\mathcal{L}_h\in \mathcal{F}$ and
go back to \ref{alg3b}. Otherwise set $\tilde{m}=\tilde{m}_{\low}=\tilde{m}_{\up}$, decrease $h$, set new $\mathcal{L}_h\in \mathcal{F}$ and continue from \ref{alg3a}.
\item Exit with $\tilde{m}(\delta)=\tilde{m}$ and  $\omega_{j,\delta}^\pm=\rho_{j,h}^\pm$
for $j=1,\ldots,\tilde{m}$.
\end{enumerate}
\end{itemize}
\end{algorithm}

Assume that 
\[
     (t_{\up},t_{\low})\cap \spec(\opmax)=\{\omega_{k+1},\ldots,\omega_{k+m}\}
\]
where 
\[
     m=\tr \1_{(t_{\up},t_{\low})}(\opmax)>0 \qquad \text{  and } \qquad k\geq 0.
\]
\textit{A priori}, an interval $(\omega_{j,\delta}^-,\omega_{j,\delta}^+)$ obtained as the output of 
Procedure~\ref{alg3} is not guaranteed to have a non-empty intersection with the spectrum of $\opmax$ or in fact include precisely the eigenvalue $\omega_{k+j}$.   However, as it is established by the following lemma,  the latter is certainly true for $\delta$ small enough.

\begin{lemma}   \label{lem_certified}
There exist $t^0>0$ and $\delta_0>0$, 
ensuring 
all the next items for all $t_{\low}\geq t^0$ and $\delta<\delta_0$.
\begin{enumerate}
\item The conditional loop in Procedure~\ref{alg3} always exits in the regime $h\to 0$.
\item $m(\delta)=m$.
\item $\omega^{-}_{j,\delta}\leq \omega_{k+j} \leq \omega_{j,\delta}^+$ for all $j=1,\ldots, m$.
\end{enumerate}
\end{lemma}
\begin{proof}
Set $t^0>\omega_1^+(t_{\up})$ sufficiently large, ensuring
$
 m \neq 0
$
for any $t_{\low}\geq t^0$. Since $\nu_j^+(t_{\up})=\omega_{k+j}=\nu^-_{m-j+1}(t_{\low})$ for all 
$j=1,\ldots,m$, Theorem~\ref{order_maxwell} alongside with the assumption on 
$\mathcal{F}$, ensures the existence of  $\rho^{\pm}_{j,h}$  in 
Procedure~\ref{alg3}-\ref{alg3b} and \ref{alg3c}, for all $j=1,\ldots, m$ 
whenever $h$ is small enough. Moreover
\[
        \rho_{j,h}^+\downarrow \omega_{k+j} \qquad \text{and} \qquad \rho^-_{j,h}\uparrow \omega_{k+j}
        \qquad \text{as } h\to 0
\]
as needed.
\end{proof}

If the eigenfunctions of $\opmax$ lie in $\mathcal{H}^{r+1}(\Omega)^6$,  then 
 \[
     \rho_{j,h}^+ - \rho_{j,h}^- =O(h^{2r}).
\]
This means that the exit rate of the  conditional loop in Procedure~\ref{alg3} is
also $O(h^{2r})$ as $h\to 0$. 

Observe that in the above procedure, a good choice of $t_{\up}$ and $t_{\low}$ has a noticeable impact in performance. See Section~\ref{non-lip}. The results of the recent manuscript \cite{BoHo2013}, suggest\footnote{An upper bound is provided in \cite[Corollary~11]{theoretical}
and the value $-1$ seems to be the right exponent.} that the constants involved in the 
estimates of Theorem~\ref{order_maxwell} are 
of order $|t-\nu^\pm_1(t)|^{-1}$.  Table~\ref{tab_slit_square} strongly 
suggest that the accuracy improves significantly, as $t_{\up}\downarrow 
\nu_1^-(t_{\up})$ and $t_{\low}\uparrow \nu_1^+(t_{\low})$.

In the subsequent sections we proceed to illustrate the practical applicability 
of the ideas discussed above by means of several examples. 
Two canonical references for benchmarks on the Maxwell 
eigenvalue problem are \cite{2004Dauge} and \cite{BFGP1999}. We validate some of our
numerical bounds against these benchmarks. Everywhere below we will write $\omega_j^\pm\equiv \omega_{j,\delta}^\pm$ (see  Procedure \ref{alg3}) where $\delta$ might or might not be specified. In the latter case, we have taken its value small enough to ensure the reported accuracy.
We consider constant $\epsilon=\mu=1$ in sections~\ref{convex-domains} and \ref{non-con}, and
 $\epsilon\not=1$ with jumps in Section~\ref{transmission}.


\section{Convex domains}\label{convex-domains}

The eigenfunctions of \eqref{maxwell} and \eqref{maxwell_2d} are regular in the 
interior of a convex domain, see~\cite{Monk2003,ABDG98}. In this, the best possible case scenario, the method of sections~\ref{feceb} and \ref{numstratnut} achieves an optimal order of convergence for finite elements. 

 Without further mention, the following convention will be in place here and everywhere below. The index $k$ for eigenvalues and eigenvalue bounds will be used, whenever multiplicities are not counted. Otherwise the index $j$ (as in previous sections) will be used.

\subsection{The square}
\label{accuracy}

\begin{figure}[t]
\centerline{
\includegraphics[height=8cm, angle=0]{\dir 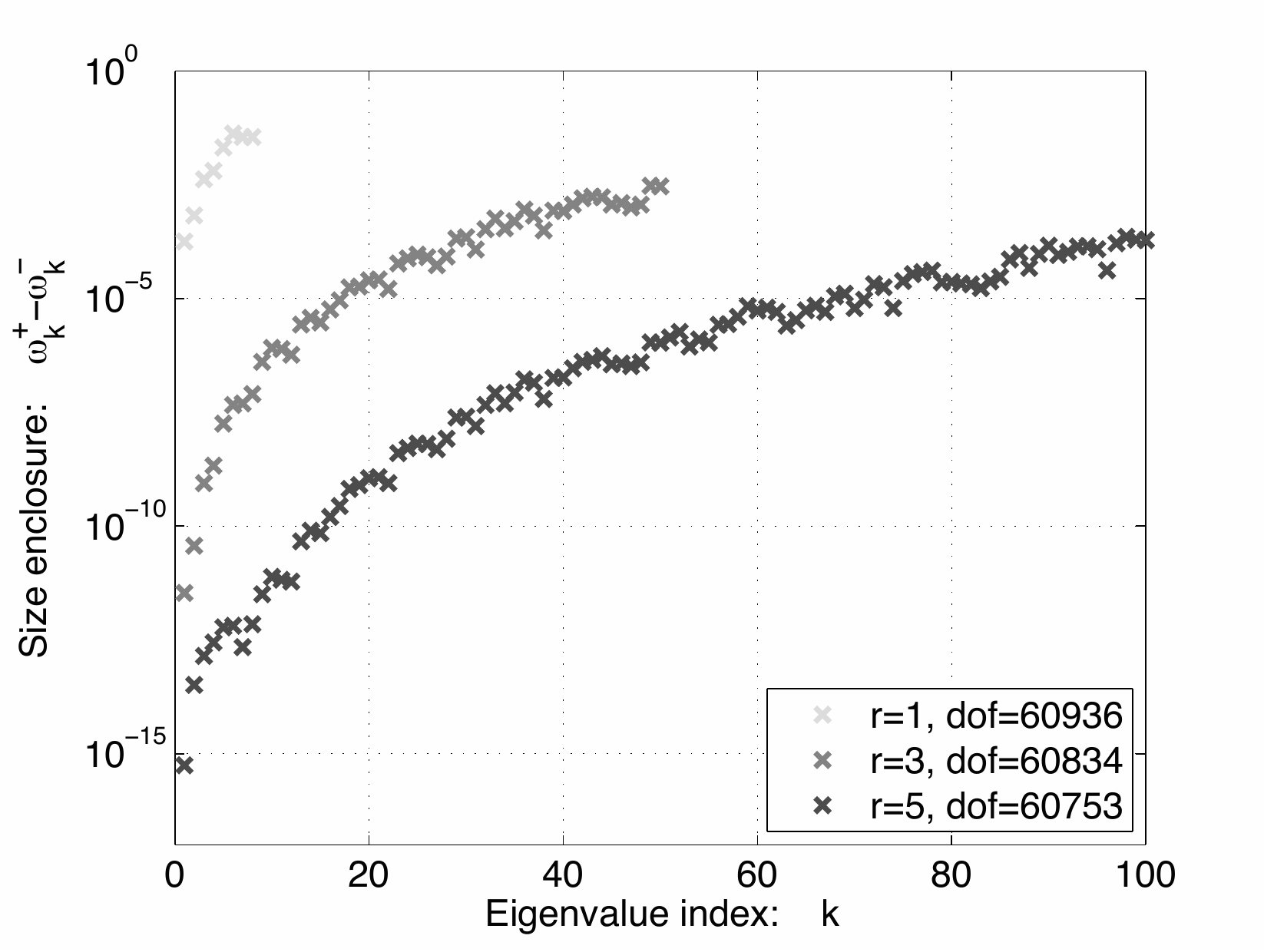}}
\caption{Semi-log graph associated to $\tilde{\Omega}_\sqr$. Vertical axis: 
$\omega^+_k -\omega^-_k$. Horizontal axis: eigenvalue index $k$ (not counting
multiplicity). Here we consider elements of order $r=1,3,5$ on unstructured uniform
meshes rendering roughly the same number of degrees of freedom. For each $r$, 
we have used exactly the same trial subspace $\mathcal{L}_h$ for all the eigenvalues. 
\label{many_eigenvalues}
}
\end{figure}

Let
$
 \tilde{\Omega}\equiv \tilde{\Omega}_\sqr=(0,\pi)^2\subset \RR^2.
$
The eigenvalues
of $\tilde{\opmax}$ are $\omega=\pm \sqrt{l^2+m^2}$ for $l,m\in \N\cup\{0\}$.  Pick 
\[
    t_{\up}=\frac{1}{4}\omega_{k-1} + \frac{3}{4} \omega_k \qquad \text{and} 
    \qquad t_{\low}=\frac{3}{4}\omega_{k} + \frac{1}{4} \omega_{k+1}
\]
to machine precision. In our first experiment we have computed enclosure widths
$\omega_k^+-\omega_k^-$ for $k=1,\ldots,100$
and $r=1,3,5$.  We have chosen $h=h(r)$ such that the corresponding trial subspaces have
roughly the same dimension $\approx61$K. We have then found all the eigenvalue bounds for a fixed $r$, from exactly the same
trial subspace. Figure~\ref{many_eigenvalues} shows the outcomes of this
experiment. In the graph, we have excluded enclosures with size above $10^{-1}$.

As it is natural to expect, for a fixed $\L_h$, the accuracy
deteriorates as the eigenvalue counting number increases: high energy
eigenfunctions have more oscillations, so their approximation requires a higher number of degrees of freedom.  The accuracy increases with the polynomial order.  
The first 100 eigenvalues are approximated fairly accurately (note
that $\omega_{(k=100)}=\sqrt{261}$ with polynomial order $r=5$).

\subsection{The slashed cube}
Let
$
   \Omega\equiv \Omega_{\sla}=(0,\pi)^3 \setminus T \subset \mathbb{R}^3,
$
where $T$ is the closed tetrahedron with vertices $(0,0,0),(\pi/2,0,0),(0,\pi/2,0)$ and $(0,0,\pi/2)$. This domain does not have
symmetries allowing a reduction into two-dimensions.

\begin{figure}[t]
\begin{minipage}{5.5cm}
\includegraphics[height=5.5cm, angle=0]{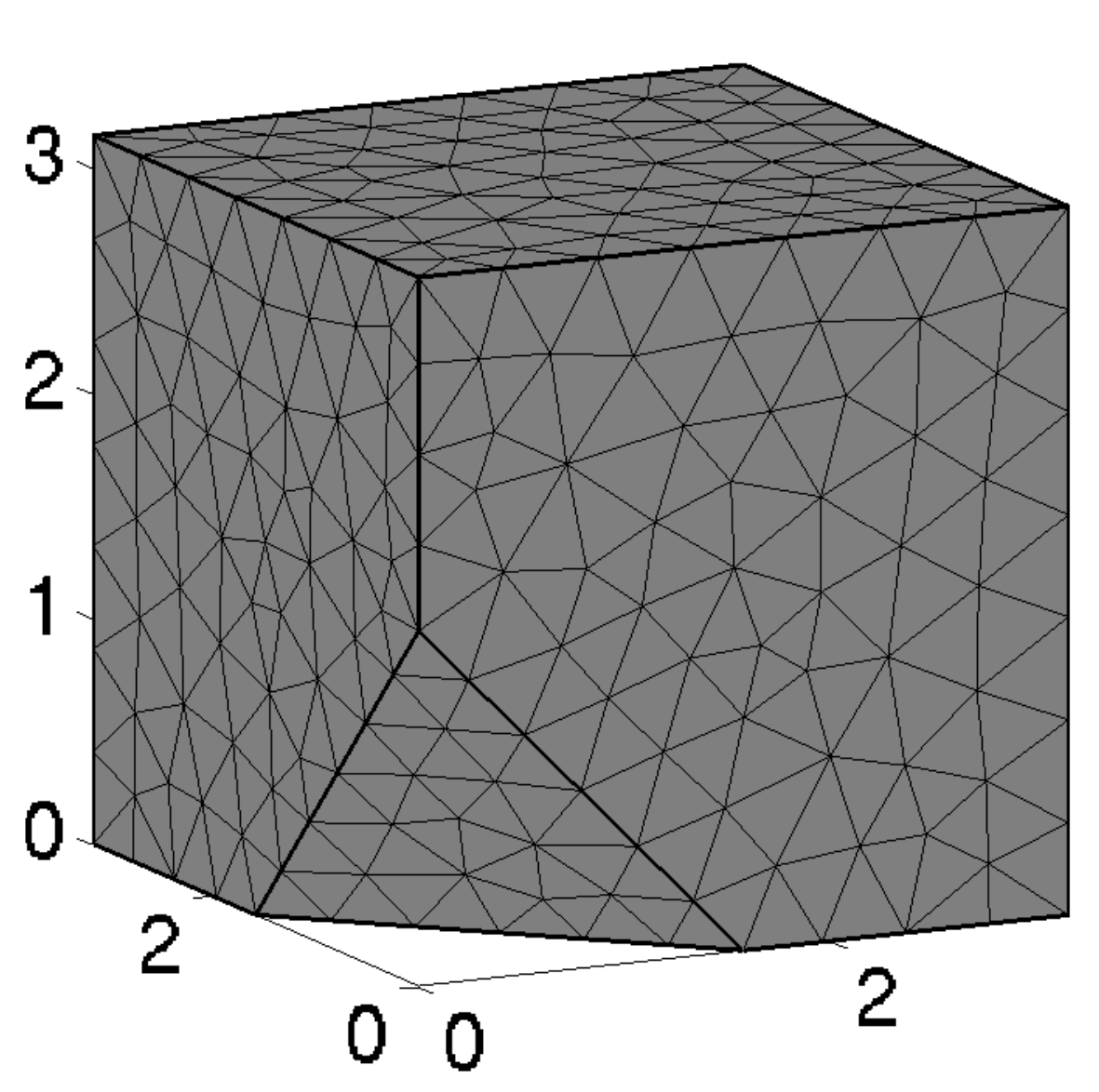}
\end{minipage} \hspace{1cm}
\begin{tabular}{c|c|c|c}
$j$ &  $\omega_j\ _-^+$   & $t_{\up}$ ($l$) & $t_{\low}$ ($l$) \\  
\hline 
$1$ & $1.412^{236}_{000}$ & $0.5$ ($1$) & $1.6$ ($3$) \\
$2$ & $1.430^{672}_{560}$ & $0.5$ ($2$) & $1.6$ ($2$) \\
$3$ & $1.430^{673}_{577}$ & $0.5$ ($3$) & $1.6$ ($1$) \\
$4$ & $1.755^{308}_{043}$ & $1.5$ ($1$) & $2.1$ ($2$) \\
$5$ & $1.755^{329}_{063}$ & $1.5$ ($2$) & $2.1$ ($1$) \\
$6$ & $2.22^{200}_{053}$ & $1.8$ ($1$)& $2.6$ ($5$) \\
$7$ & $2.237^{667}_{434}$ & $1.8$ ($2$) & $2.6$ ($4$) \\
$8$ & $2.237^{684}_{459}$ & $1.8$ ($3$) & $2.6$ ($3$) \\
$9$ & $2.239^{533}_{387}$ & $1.8$ ($4$) & $2.6$ ($2$) \\
$10$ & $2.270^{778}_{558}$ & $1.8$ ($5$) & $2.6$ ($1$) \\
\hline
\end{tabular} \hspace{1cm}
\caption{Benchmark spectral approximation for $\Omega_{\sla}$. In the table we
compute interval of enclosure for the first 10 eigenvalues, by means of an implementation of Procedure~\ref{alg3}. The
trial spaces are made of Lagrange elements of order $r=3$. The final mesh is the
one shown on the right side. Total number of DOF=117102. \label{cut_cube_table}}
\end{figure}

In our first experiment on this region, we determine benchmark eigenvalue
enclosures for \eqref{maxwell}.  The table in Figure~\ref{cut_cube_table} shows the outcomes of implementing a numerical scheme
based on Procedure~\ref{alg3}.  We have iterated our algorithm for three fixed
choices of $t_{\up}$ and $t_{\low}$ (third and fourth columns), with
$\delta=10^{-2}$.  We have picked the family of meshes so that no more than five
iterations were required to achieve the needed accuracy.  We have chosen trial spaces made out of Lagrange elements of
order $r=3$. All the final eigenvalue enclosures have a length of at most
$2\times 10^{-3}$.  The mesh used in the last iteration  is depicted on the
left of Figure~\ref{cut_cube_table}.   The parameter $l$ in
this table counts the number of eigenvalues to the right of $t_\up$ or to the
left of $t_\low$, respectively.

From the table, it is natural to conjecture that there is a cluster of eigenvalues at the
bottom of the positive spectrum near $\sqrt{2}$. The latter is the first
positive eigenvalue for $\Omega\equiv \Omega_{\cbe}=(0,\pi)^3$, which is of multiplicity 3 for that region. See \cite[Section~5.1]{theoretical}. As we deform $\Omega_{\cbe}$ into $\Omega_{\sla}$, it appears that
this eigenvalue splits into a single eigenvalue at the bottom of the spectrum
and a seemingly double eigenvalue slightly above it.  
Another cluster occurs at $\omega_4$ and $\omega_5$ with strong indication that
this is a double eigenvalue. This pair is near $\sqrt{3}$, the second eigenvalue
for $\Omega_{\cbe}$, which is indeed double.  The next eigenvalues for $\Omega_{\cbe}$ are
$2$  and $\sqrt{5}$ with total multiplicity 5. We conjecture that $\omega_j$ for $j=6,\ldots,10$ are indeed perturbations of these eigenvalues.

For our second experiment on the region $\Omega_{\sla}$, we have estimated numerically the
electromagnetic fields corresponding to index up to $j=6$.  
The purpose of the experiment is to set benchmarks
for the eigenfunctions on $\Omega_{\sla}$ and simultaneously illustrate
Theorem~\ref{order_maxwell_eigenfunction}.  
In Figure~\ref{cornerless_1} we depict the density of electric and magnetic
fields, $|\bE|$ and $|\bH|$ both re-scaled to having maximum equal to 1. We also
show arrows pointing towards the direction of these fields on $\partial
\Omega_{\sla}$.  The mesh employed for these calculations is the one shown in
Figure~\ref{cut_cube_table}.  

It is remarkable that for  both experiments on $\Omega_\sla$, 
a reasonable accuracy has been achieved even for the fairly coarse mesh depicted.


\section{Non-convex domains} \label{non-con}

The numerical approximation of the eigenfrequencies and electromagnetic fields
in the resonant cavity is known to be extremely challenging when the domain is not convex. 
The main reason for this is the fact that the
electromagnetic field might have a singularity and a low degree of regularity at
re-entrant corners. See for example the discussion after
\cite[Lemma~3.56]{Monk2003} and references therein. 

In some of the examples of this section we consider a mesh adapted to the geometry of the region.
However, we do not pursue any specialized mesh refinement strategy.  
We show below that, even in the case where there is poor approximation due to
low regularity of the eigenspace, the scheme in Procedure~\ref{alg3} provides
a stable approximation.

\begin{figure} 
\centerline{
\begin{tabular}{c|c|c|cc}
$j$ &$\omega_j$ from \cite{BFGP1999} &  $\omega_j\ _-^+$   & $t_{\up}$ ($l$) &
$t_{\low}$ ($l$) \\  
 &  (from \cite{2004Dauge}) & & & \\ 
\hline 
$1$ &$0.768192684$  & $0.773334_{694}^{991}$ & $0.1$ ($1$) & $2.1$ ($4$) \\
 & ($0.773334985176$) &&& \\
$2$ &$1.196779010$ & $1.1967827557_{026}^{761}$ & $0.1$ ($2$) & $2.1$ ($3$) \\
 & ($1.19678275574$) &&& \\
$3$ &$1.999784988$ & $_{1.99999999933}^{2.00000000064}$ & $1.5$ ($1$) & $2.5$
($4$) \\
 & ($2.00000000000$) &&& \\
$4$ &$1.999784988$ & $_{1.99999999936}^{2.00000000067}$ & $1.5$ ($2$) & $2.5$
($3$) \\
 & ($2.00000000000$) &&& \\
$5$ &$2.148306309$ & $2.14848368_{199}^{365}$ &  $1.5$ ($3$) & $3.1$ ($5$) \\
 & ($2.14848368266$) &&& \\
$6$ &$2.252760528$ & $2.25729_{776}^{896}$ & $1.5$ ($4$) & $3.1$ ($4$) \\
$7$ &$2.828075317$ & $2.8284271_{186}^{354}$ & $1.5$ ($5$) & $3.7$ ($4$) \\
$8$ &$2.938491109$ & $2.94671_{112}^{343}$ & $1.5$ ($6$) & $3.7$ ($3$) \\
$9$ &$3.075901493$ & $3.0758929_{571}^{738}$ & $1.5$ ($7$) & $3.7$ ($2$) \\
$10$ &$3.390427701$ & $3.3980_{676}^{724}$ & $1.5$ ($8$) & $3.7$ ($1$) \\
\hline
\end{tabular}  
\begin{minipage}{7.5cm}
\includegraphics[height=7cm, angle=0]{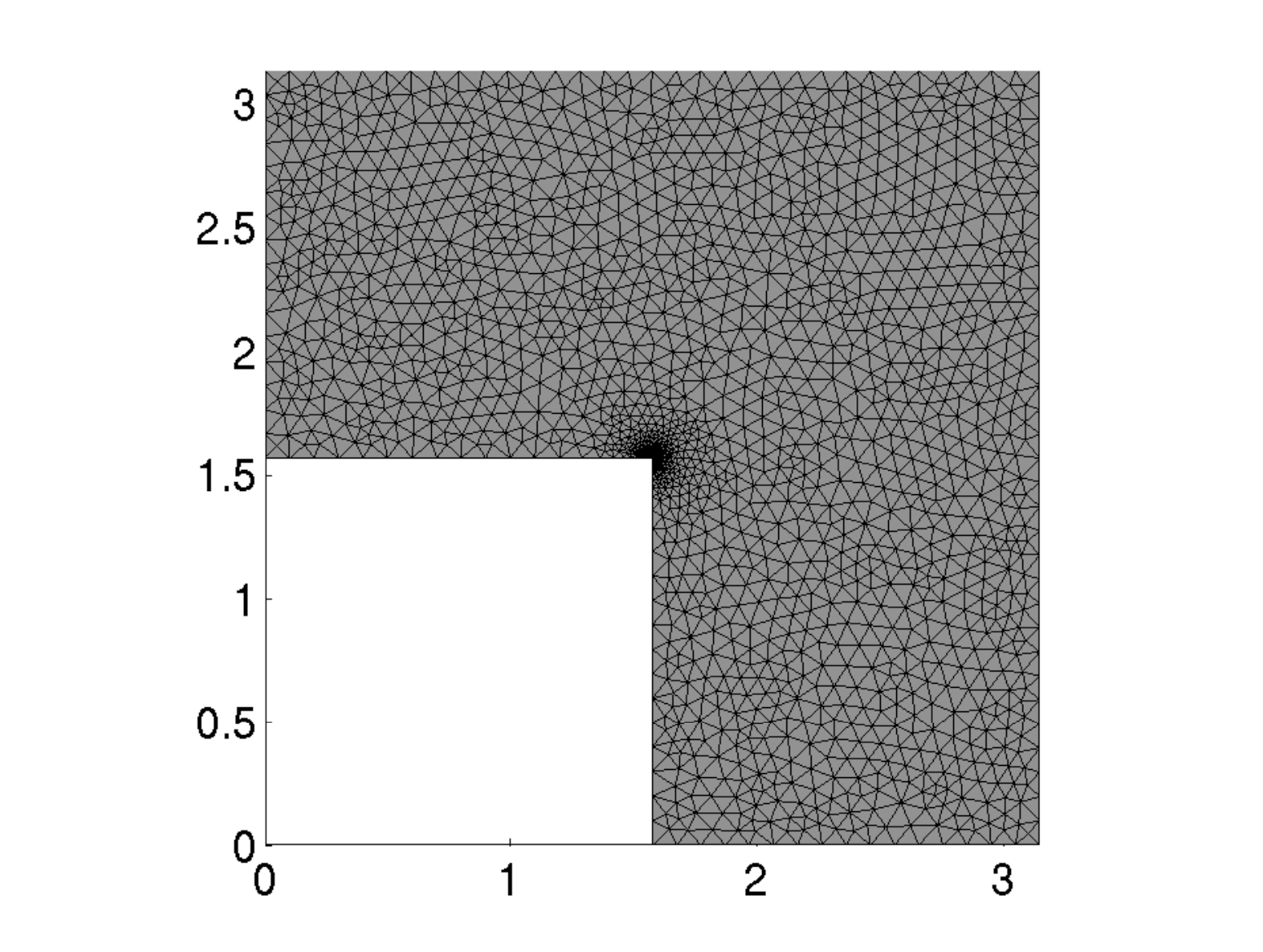}
\end{minipage}
}
\caption{Enclosures for the first 10 positive eigenvalues of $\tilde{\opmax}$ on
$\tilde{\Omega}_{\lsp}$.  The next eigenvalue is above 3.7. Here
Procedure~\ref{alg3} has been implemented on Lagrange elements of order 3. The
final mesh shown on the right has a number of DOF=56055. The mesh has a maximum
element size  $h=0.1$ and has been refined at $(\pi/2,\pi/2)\in \partial \tilde{\Omega}_{\lsp}$. 
For comparison
on the second column we include the eigenvalue estimations
found in \cite{BFGP1999} and \cite{2004Dauge}.\label{table_eigenvalues_lshape}}
\end{figure}

\subsection{A re-entrant corner in two dimensions} \label{Lshape}
The  region \[\tilde{\Omega}\equiv \tilde{\Omega}_{\lsp}=(0,\pi)^2\setminus [0,\pi/2]^2\subset \R^2\]
is a classical benchmark domain both for the Maxwell and the Helmholtz problems,
and it has been extensively examined in the past. Numerical computations for the
eigenvalues of $\tilde{\opmax}$, via an
implementation based on a mixed formulation of \eqref{maxwell_2d} and  edge finite elements, were reported in \cite[Table~5]{BFGP1999}. See also \cite{2004Dauge}. We now show estimation of sharp enclosures for these eigenvalues by means of 
the method described in sections~\ref{convex-domains} and \ref{non-con}.

\begin{figure}[t]
\centerline{\includegraphics[height=8cm, angle=0]{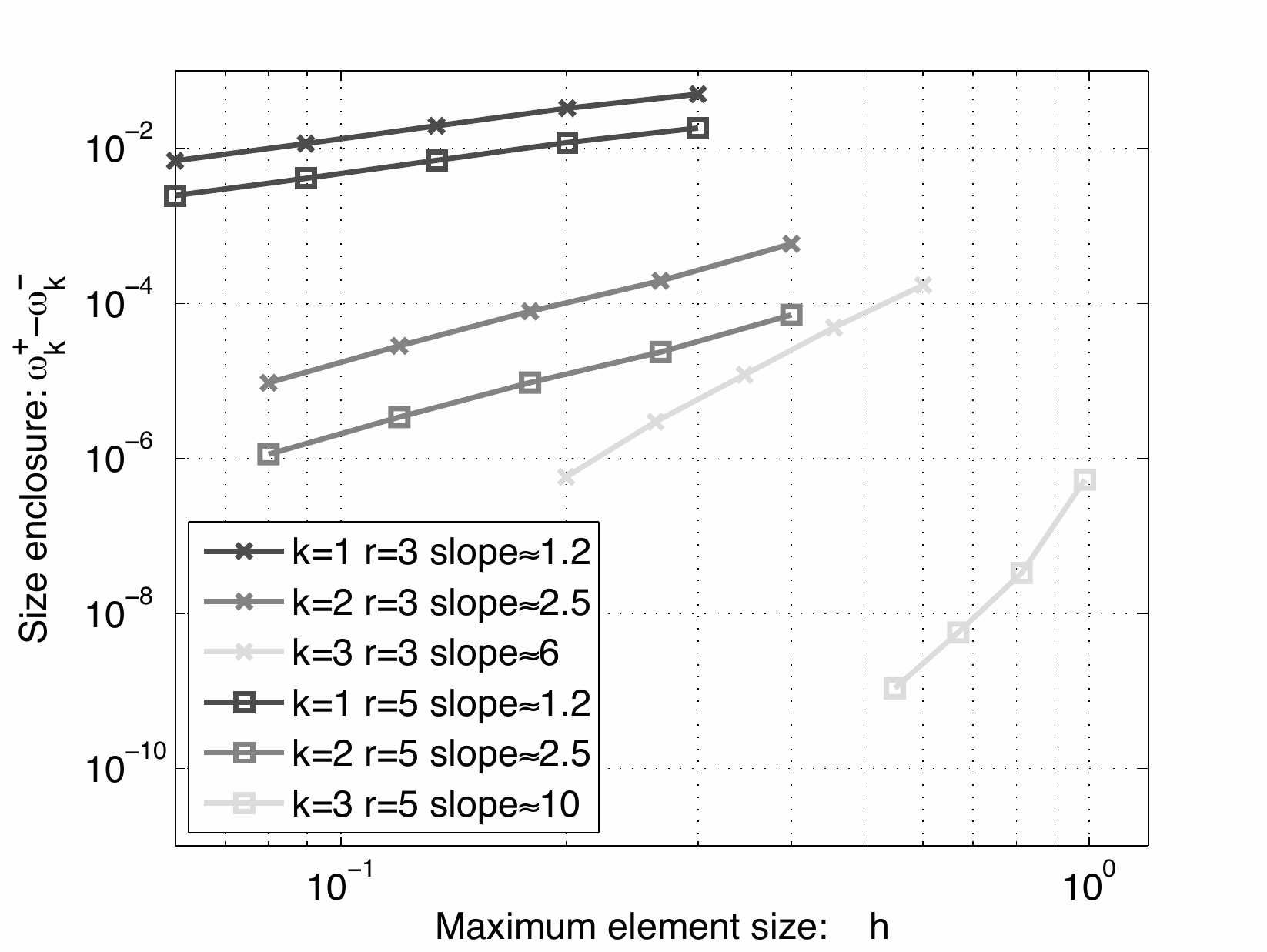}}
\caption{ 
Compared order of approximation for different eigenvalues in the region
$\tilde \Omega_{\lsp}$. The log-log plot shows residual versus maximum element size $h$
for the calculation of enclosures for $\omega_j$ where $j=1,2,3$ and
$\L$ is generated by Lagrange elements of order $r=3$ and $r=5$. Note
 that $(\bE,H)\not \in \mathcal{H}^s(\Omega_{\lsp})^3$ for $j=1$ and $s=1$, and
for $j=2$ and $s=1.5$. On the other hand, for $j=3$ we have $(\bE,H)$
smooth, as the eigenfunction is also solution of \eqref{maxwell} on a square of
side $\pi/2$.
\label{lshape_orders}
}
\end{figure}

For the next set of experiments we consider unstructured triangulations of the
domain, refined around the re-entrant corner $(\pi/2,\pi/2)\in \partial \tilde\Omega_{\lsp}$. The polynomial order is set to
$r=3$.  Figures~\ref{table_eigenvalues_lshape},  \ref{lshape_orders} and  \ref{Lshape_eigenfunctions}
summarize our findings.

We produced the table in Figure~\ref{table_eigenvalues_lshape}  by implementing
Procedure~\ref{alg3} in the same fashion as
for the case of $\Omega_\sla$ described previously.
For comparison, in the second column of this table we have included the benchmark
eigenvalue estimations
found in \cite{BFGP1999} and \cite{2004Dauge}. Note that some of the 
approximations made by means of the mixed formulation  are lower bounds of the true
eigenvalues, and some (see the row for $j=9$ in the table) are upper bounds. This confirms that the latter approach is in general un-hierarchical
as previously suggested in the literature.

From the third column of the table, it is clear that the accuracy depends on the
regularity of the corresponding eigenspaces.
The eigenfunctions  associated to $\omega=2$ and $\omega=\sqrt{8}$
are found by the translation and gluing in an appropriate fashion, of eigenfunctions in the sub-region $\tilde{\Omega}=(0,\pi/2)^2\subset \tilde{\Omega}_{\lsp}$. 
These eigenfunctions are smooth in the interior of $\tilde \Omega_\lsp$ and they achieve a maximum order of convergence. 
The eigenfunctions associated to $\omega_1$ and
$\omega_2$, on the other hand, are singular at the re-entrant corner. Moreover, the electric field component for index $j=1$ is known to be outside $\sobol^1(\Omega_{\lsp})^2$ while that for index $j=2$ is in 
$\sobol^1(\tilde \Omega_{\lsp})^2$. This explains the significant gain in accuracy in the calculation of $\omega_2$ with respect to the one for $\omega_1$.  Here the computation of the eigenvalues with smooth eigenspace ($j=3,4$ or $7$) is less accurate than that for the index $j=2$, because of the mesh chosen.

Figure~\ref{lshape_orders}  depicts
in log-log scale residuals versus maximum element size. We have considered here
Lagrange elements of order $r=3$ and $r=5$. The hierarchy of meshes (not shown)
was chosen unstructured, but with an uniform distribution of nodes.  Since the
eigenfunctions associated to $\omega_1$ and $\omega_2$ have a limited regularity, 
then there is no noticeable improvement on the convergence order as $r$ changes from $3$ to $5$.
Since the third eigenfunction is smooth, it does obey the estimate \eqref{order_conv_evalue}.

Benchmark approximated eigenfunctions are depicted in
Figure~\ref{Lshape_eigenfunctions}. The mesh employed to produce 
these graphs is the one shown on the right of Figure~\ref{table_eigenvalues_lshape}.  As some of the electric fields have a
singularity at $(\pi/2,\pi/2)\in\partial \tilde{\Omega}_{\lsp}$ we have re-scaled each individual plot to a
range in the interval $[0,1]$.

\begin{figure}
\begin{minipage}{7cm}
\includegraphics[height=7cm, angle=0]{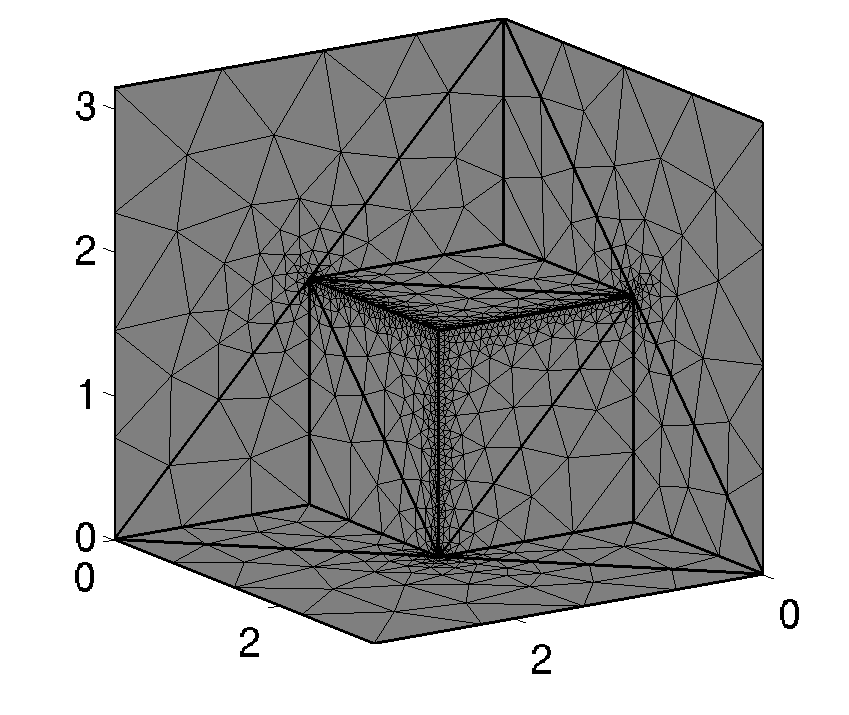}
\end{minipage} \hspace{1cm}
\begin{tabular}{c|c}
$j$  & $\omega_j\ _-^+$ \\  
\hline 
$1$  & $1.1^{441}_{298}$   \\ \hline
$2$ & $1.54^{391}_{228}$ \\
$3$  & $1.54^{392}_{233}$ \\   \hline
$4$  &  $2.0^{815}_{690}$\\
$5$  &  $2.0^{820}_{796}$  \\
$6$ & $2.0^{820}_{796}$\\   \hline
$7$  & $2.2^{348}_{180}$ \\
$8$  & $2.2^{348}_{187}$ \\   \hline
$9$ &  $2.32^{679}_{547}$\\   \hline
$10$  & $2.33^{250}_{044}$ \\
$11$  &  $2.33^{255}_{066}$\\   \hline
$12$  & $2.^{400}_{375}$\\   \hline
$13$  &  $2.60^{366}_{023}$\\
$14$  &  $2.605_{62}^{98}$\\
$15$  &  $2.605_{65}^{97}$\\
\hline
\end{tabular} \hspace{1cm}
\caption{Spectral enclosures for the Fichera domain $\Omega_{\lstd}$. Here we have fixed $t_{\up}=0.1$ and
$t_{\low}=2.8$.  The final mesh in the iteration is shown  on the left side. Its number of DOF=347460. \label{fichera_table}}
\end{figure}

\subsection{The Fichera domain}\label{Section:Fishera}
In this next experiment we consider the region
\[\Omega\equiv \Omega_{\lstd}=(0,\pi)^3 \setminus [0,\pi/2]^3 \subset \R^3.\]
See also \cite[\S5.2]{theoretical} for related results.

The table on right of Figure~\ref{fichera_table} shows numerical
estimation of the first 15 positive eigenvalues. Here we have fixed
$t_{\up}=0.1$ and $t_{\low}=2.8$. We have considered meshes refined along the re-entrant edges. The final mesh is shown on
the left side of Figure~\ref{fichera_table}.  We have
stopped the algorithm when the tolerance $\delta=0.03$ has been achieved.
However, note that  the accuracy is much higher for the indices
$j=2,3,9,10,11,13,14,15$. 

Figure~\ref{fichera} includes the corresponding approximated eigenfunctions. The
mesh employed for this calculation is the same as that of
Figure~\ref{fichera_table}.

\subsection{A non-Lipschitz domain} \label{non-lip}

As mentioned earlier, for a single trial space
$\mathcal{L}$, the accuracy of the eigenvalue bounds established in Theorem~\ref{bounds}
depends on the position of $t$ relative to adjacent components of the spectrum. In this
experiment we demonstrate that this dependence might vary significantly
with $t$. The numerical evidence below suggests that a good choice of $t_{\up}$ and
$t_{\low}$ plays a major role in the design of  efficient algorithms for eigenvalue
calculation based on this method.

\begin{figure}      
\begin{tabular}{c|c||c|c||c|c}
RF &  DOF   &  $t_{\low}=1.95$ & $t_{\low}=2.05$   &
$t_{\up}=1.05$   & $t_{\up}=0.7$    \\  
 &   &  ($l=1$  $\omega_3^-)$&   ($l=3$ $\omega_3^-)$ &
 ($l=1$ $\omega_3^+)$ &  ($l=3$  $\omega_3^+)$ \\  
\hline 
1 & 4143 &  1.24764 & 1.26640 & 1.50395 & 1.3436  \\
0.1 & 9648 & 1.25029 & 1.26830 & 1.49282 & 1.3336 \\
0.01 & 74226 & 1.25063 & 1.26846 & 1.48899 & 1.3274 \\ \hline
\end{tabular}
\caption{Dependence of the accuracy of the bounds from Theorem~\ref{bounds} on the choice of $t$ for the region $\tilde{\Omega}_{\mathrm{cut}}$. It is preferable to
pick $t_{\up}$ and $t_{\low}$ as far as possible from $\omega$, than to increase
the dimension of the trial subspace. \label{tab_slit_square}}
\end{figure}

Let $\tilde{\Omega}\equiv \tilde{\Omega}_{\mathrm{cut}}= (0,\pi)^2\setminus S$ for $S=[\pi/2,\pi]\times \{\pi/2\}$.
Benchmarks \cite{2004Dauge} on the eigenvalues of \eqref{maxwell_2d} are found by means of
solving numerically the corresponding Neumann Laplacian problem.

The first seven positive eigenvalues are
\begin{gather*}
\omega_1\approx 0.647375015,\, \omega_2=1,\, \omega_3\approx 1.280686161,   \\
\omega_4=\omega_5=2,\, \omega_6\approx 2.096486081 \quad\text{and} \quad
\omega_7\approx 2.229523505.  
\end{gather*}
The eigenfunctions associated to $\omega_2$, $\omega_4$ and $\omega_5$ are
smooth, as they are also eigenfunctions on $\tilde \Omega_{\sqr}$. On the other hand,
$\omega_1$ and $\omega_3$ correspond to singular eigenfunctions.
Standard nodal elements are completely unsuitable for the computation of these
eigenvalues, even with a significant refinement of the mesh on $S$.

The table in Figure~\ref{tab_slit_square} shows computation of $\omega_3^\pm$ on a mesh that
is increasingly refined at $S$ with a factor RF for two pairs of choices of $t_{\up}$ and $t_{\low}$. Here $h=0.1$ and we consider Lagrange elements of order $r=1$. The choice of $t_{\up}$ and $t_{\low}$ further from $\omega_3$,  even with the very coarse mesh, provides a sharper estimate of
$\omega_3^\pm$ than the other choices even with a finer mesh. 


\section{The transmission problem}\label{transmission}

In this final example, we consider a non-constant electric permittivity. Let
\begin{gather*}
\Omega_{\sqr,1} =\left(0,\frac\pi 2\right)\times \left(0,\frac\pi 2\right)  \qquad \Omega_{\rm  sqr,2}=\left(\frac\pi 2,\pi\right)\times \left(\frac\pi 2,\pi\right)
\\ \Omega_{\rm sqr,3}=\left(\frac \pi 2,\pi\right)\times\left(0,\frac \pi 2\right) \qquad
\text{and} \qquad \Omega_{\rm  sqr,4}=\left(0,\frac\pi 2\right)\times\left(\frac\pi 2,\pi\right).
\end{gather*}
so that \[\overline{\tilde \Omega_{\sqr}}=\overline{\bigcup_{l=1}^4\tilde\Omega_{\sqr,l}}.\]
Set $\mu=1$ and
\begin{equation*}
\epsilon(x)= \left\{\begin{array}{ll}
    1 & x\in \Omega_{\sqr,1} \cup \Omega_{\sqr,2} \\
    \frac12 & x\in \Omega_{\sqr,3} \cup \Omega_{\sqr,4}.
 \end{array} \right.
\end{equation*}
Numerical estimations of the eigenvalues of $\tilde\opmax$ on $\tilde{\Omega}\equiv \tilde{\Omega}_\sqr$ for this data were found in \cite{2004Dauge}. 

\begin{figure}\label{table_eigenvalues_transmission}
\centerline{
\begin{tabular}{c|c|c|lcc}
$j$ &$\omega_j$ from \cite{2004Dauge} &  $\omega_j\ _-^+$   & $l$ & up  &
low \\  
\hline
$1$ & $1.15954813181$ & $ 1.159^{555}_{456}$& & $1$ & $85$ \\
$2$ &  $1.16804100636$ & $1.16^{807}_{770}$ && $2$ & $84$ \\
$3$ &  $1.5834295853$ & $1.5834^{453}_{229}$ && $3$ &  $83$ \\
$4$ &  $2.3757369919$ & $2.375^{788}_{452}$ & &$4$ &  $82$  \\
$5$ & $2.4724291674$ & $2.472^{479}_{212}$ & &$5$ &   $81$ \\
$6$ & $ 2.5288205712$ & $2.528^{884}_{634}$ & &$6$ &  $80$ \\
$7$ &  $2.7487894882$ & $2.748^{868}_{693}$ && $7$ &  $79$ \\
$8$ &  $3.2334726763$ & $3.23^{362}_{280}$ & &$8$ &  $78$ \\
$9$ &  $3.47832176265$ & $3.47^{8478}_{775}$ &&  $9$ & $77$\\
$10$ & $ 3.51802898831$ & $3.51^{822}_{718}$ && $10$ & $76$ \\
\hline
\end{tabular}  
}
\caption{Enclosures for the first 10 positive eigenvalues of 
$\tilde\opmax$ for the transmission problem (Section~\ref{transmission}).
For comparison, on the second column we include the upper bounds
found in \cite{2004Dauge}. Here the trial subspace is made out of Lagrange elements of order 1,
$t_{\up}=10^{-9}$ and $t_{\low}=11.74$. 
The mesh employed was constructed in an unstructured fashion in the
four sub-domains $\tilde\Omega_{\sqr,l}$. The maximum element size is set to 
$h=.01$ and the total number of DOF=399720.  }
\end{figure}

We have set the experiment reported in Figure~\ref{table_eigenvalues_transmission}, on a family of meshes (not shown), which is unstructured but of equal maximum element sizes in each one of the subdomains $\tilde\Omega_{\sqr,l}$. We implemented Procedure~\ref{alg3} as discussed previously, with fix $t_{\up}=10^{-9}$ and $t_{\low}=11.74$. For comparison, in the second column of the table we have included the benchmark upper bounds from \cite{2004Dauge}. 

As we pointed out in sections~\ref{convex-domains} and \ref{non-con}, accuracy depends  on the regularity of the corresponding eigenspace. Moreover, finding conclusive lower bounds for the ninth and tenth eigenvalues turns out to be extremely expensive, if $t_{\low}\approx 3.5$. Observe that, from the reproduced values in the second column of the table, these 
two eigenvalues form a cluster of multiplicity 2. It seems that in fact
they are part of a larger cluster. The resulting narrow gap from this cluster seems to be the cause of
the dramatic deterioration in accuracy. Recall the observations made in Section~\ref{non-lip}.

The data has a natural symmetry with respect to the diagonals of $\tilde{\Omega}_{\sqr}$. 
Four types of eigenvectors arise from these symmetries, and the analytical problem reduces 
to four different eigen-problems which give rise to degenerate eigenspaces. As we are not considering a mesh 
that completely respects these symmetries, the multiplicities arising from them are not shown completely in the numerics.

In order to find reasonable bounds for $\omega_9$ and $\omega_{10}$, we had to resource to exploiting 
 the fact that $\rho^{-}_{j}(t,h)$ is locally non-increasing in $t$, and it respects ordering in $j$. An analytical proof of this property is achieved by extending to the indefinite case the results of \cite[\S3]{BoHo2013}, but in the present context we have examined them only from a numerical perspective. Note that, when $t_\low$ is near to cross an eigenvalue,  $\rho^{-}_{j}(t_{\low},h)$ jumps. These jumps appear to be small (respecting the order of the $j$) as long as the subspace captures well the eigenvectors. This effect will disappear eventually as we increase $t_\low$ further, due to the fact that $\L$ is finite-dimensional. In our experiments, we have determined that $t=t_{\low}\approx 11.74$ is near to optimal for the trial subspaces employed. Note that $t_{\low}=11.74$ gives $85$ eigenvalues in the segment $(10^{-9},11.74)$ for these trial subspaces.

%
%
%
%
%


\section*{Acknowledgements} 
We kindly thank Universit\'e de Franche-Comt{\'e}, University College London and the Isaac Newton Institute for Mathematical Sciences, for their hospitality. Funding was provided by the British-French project PHC Alliance  (22817YA), the British Engineering and Physical Sciences Research Council  (EP/I00761X/1 and \linebreak EP/G036136/1) and the French Ministry of Research (ANR-10-BLAN-0101).

\bibliographystyle{siam}

\def\cprime{$'$}

\begin{figure}
\centerline{
\includegraphics[height=5cm, angle=0]{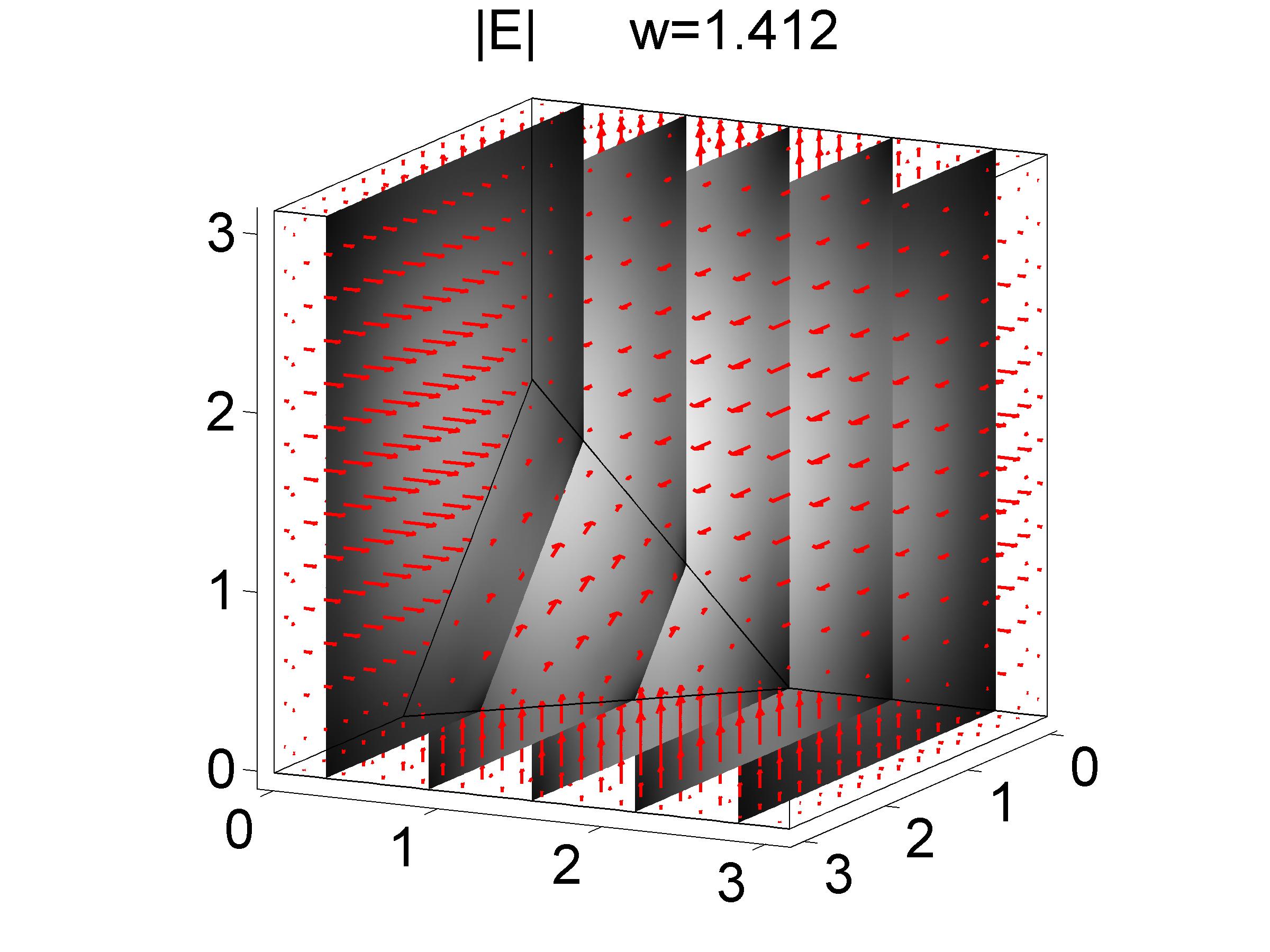} \hspace{-5mm}
\includegraphics[height=5cm, angle=0]{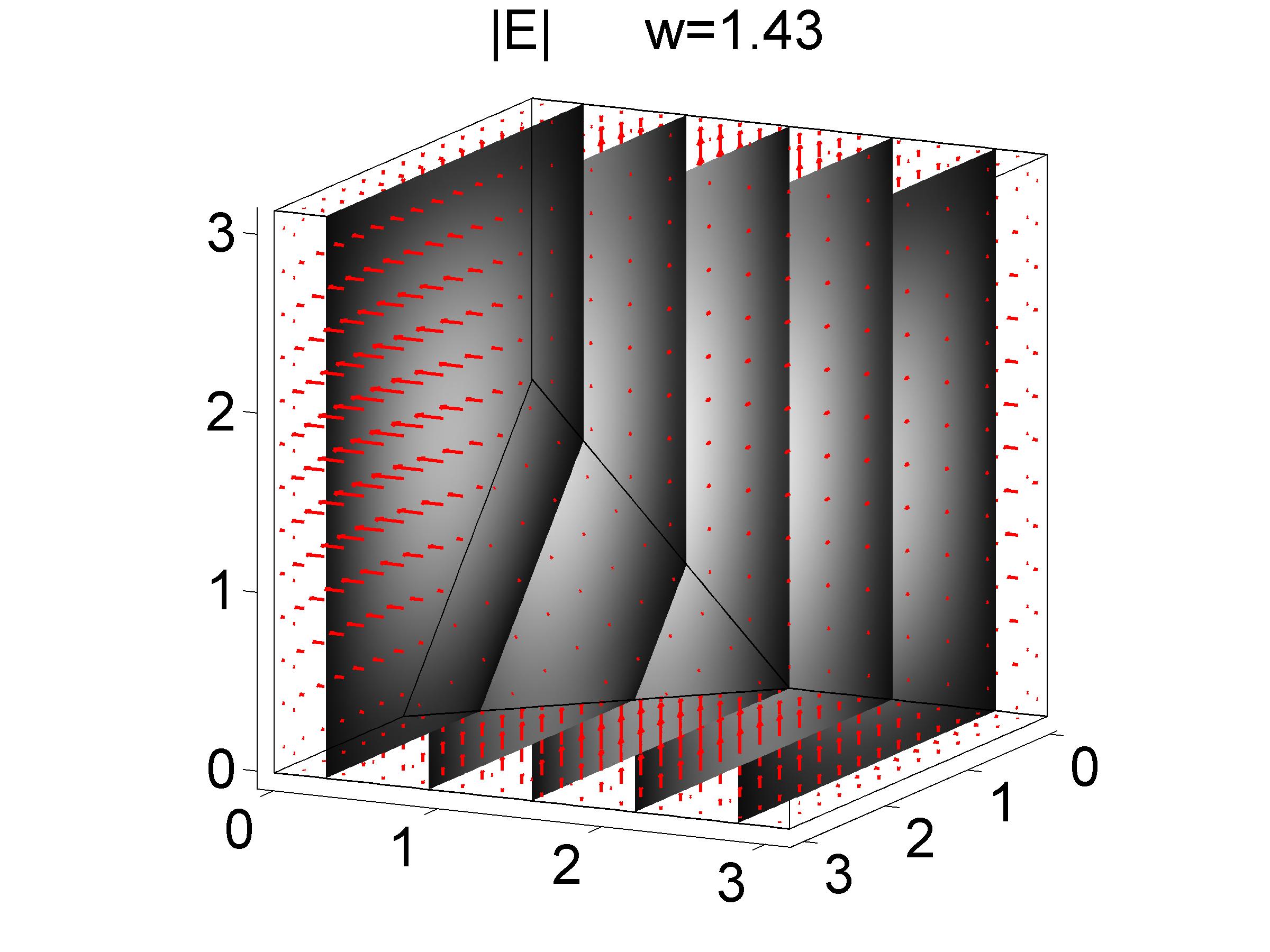} \hspace{-5mm}
\includegraphics[height=5cm, angle=0]{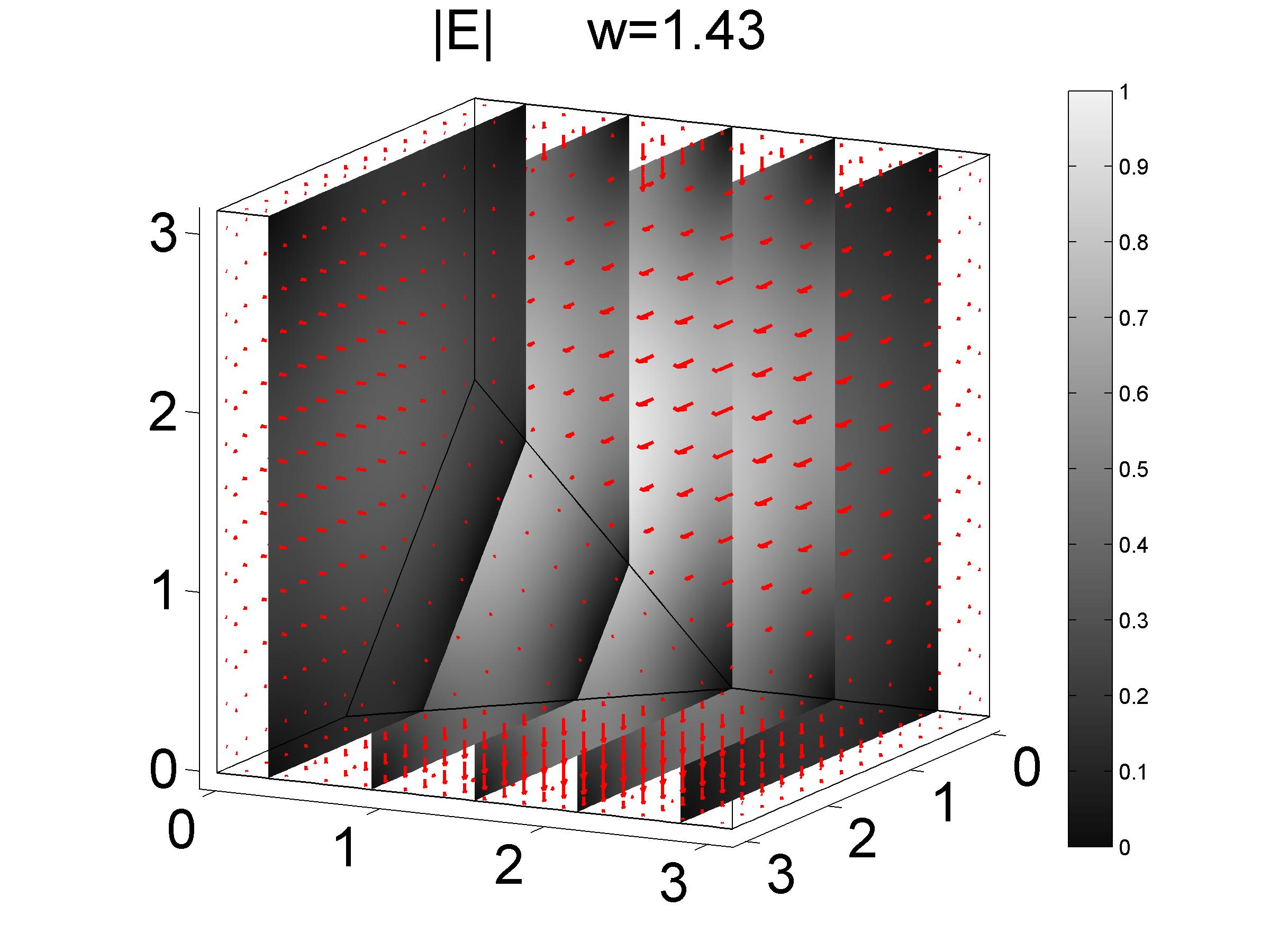} 
}
\centerline{
\includegraphics[height=5cm, angle=0]{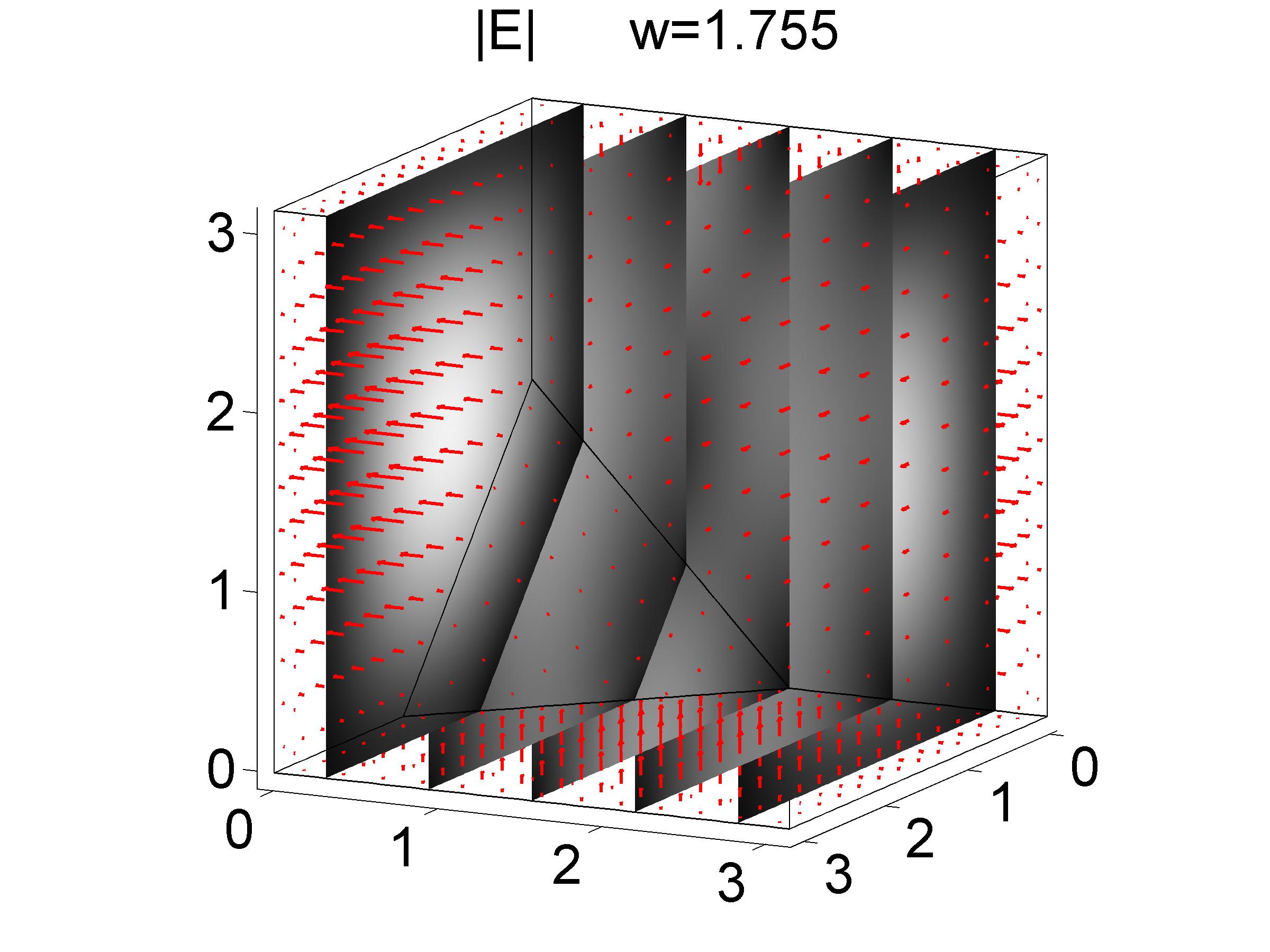} \hspace{-5mm}
\includegraphics[height=5cm, angle=0]{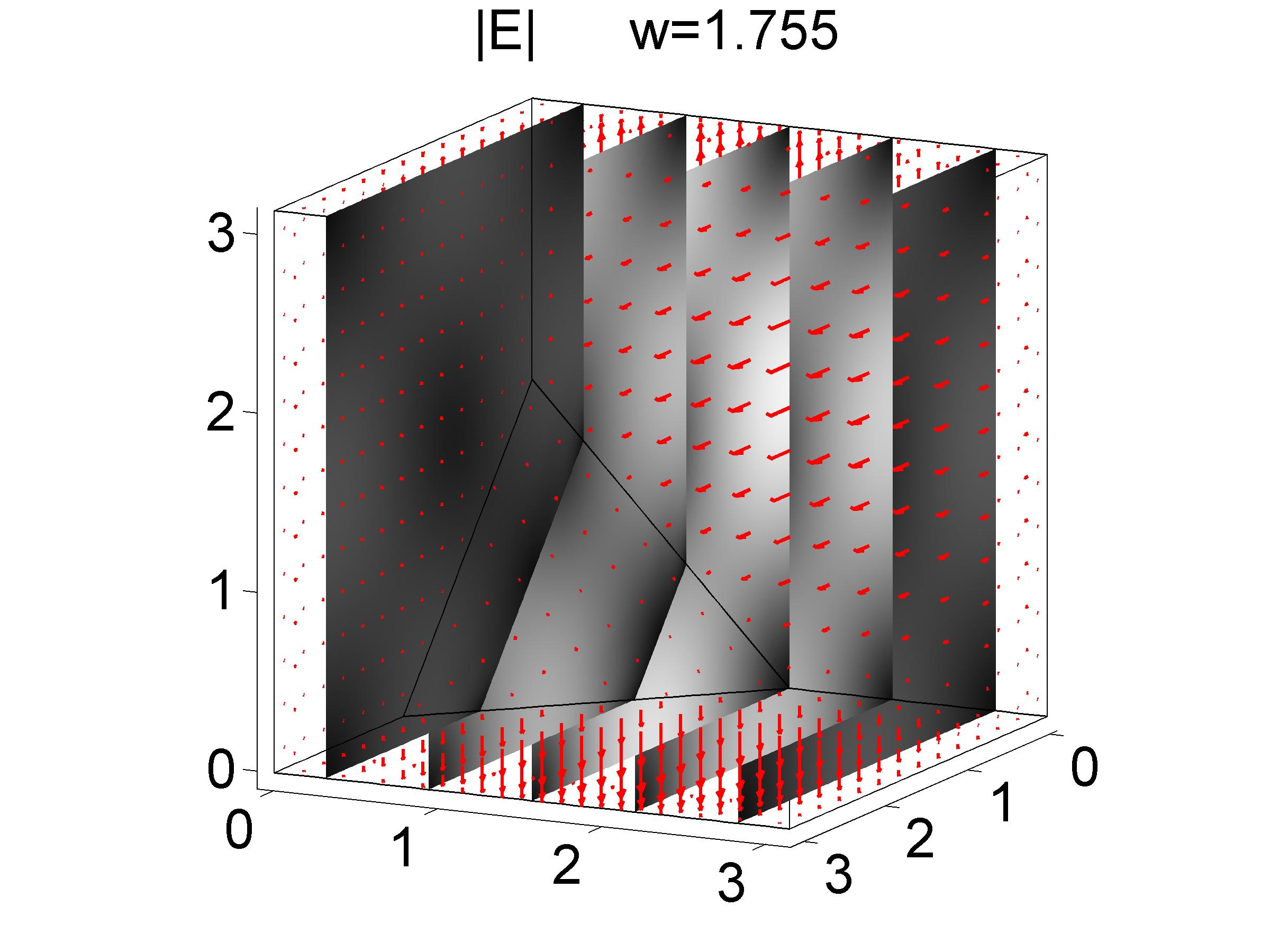} \hspace{-5mm}
\includegraphics[height=5cm, angle=0]{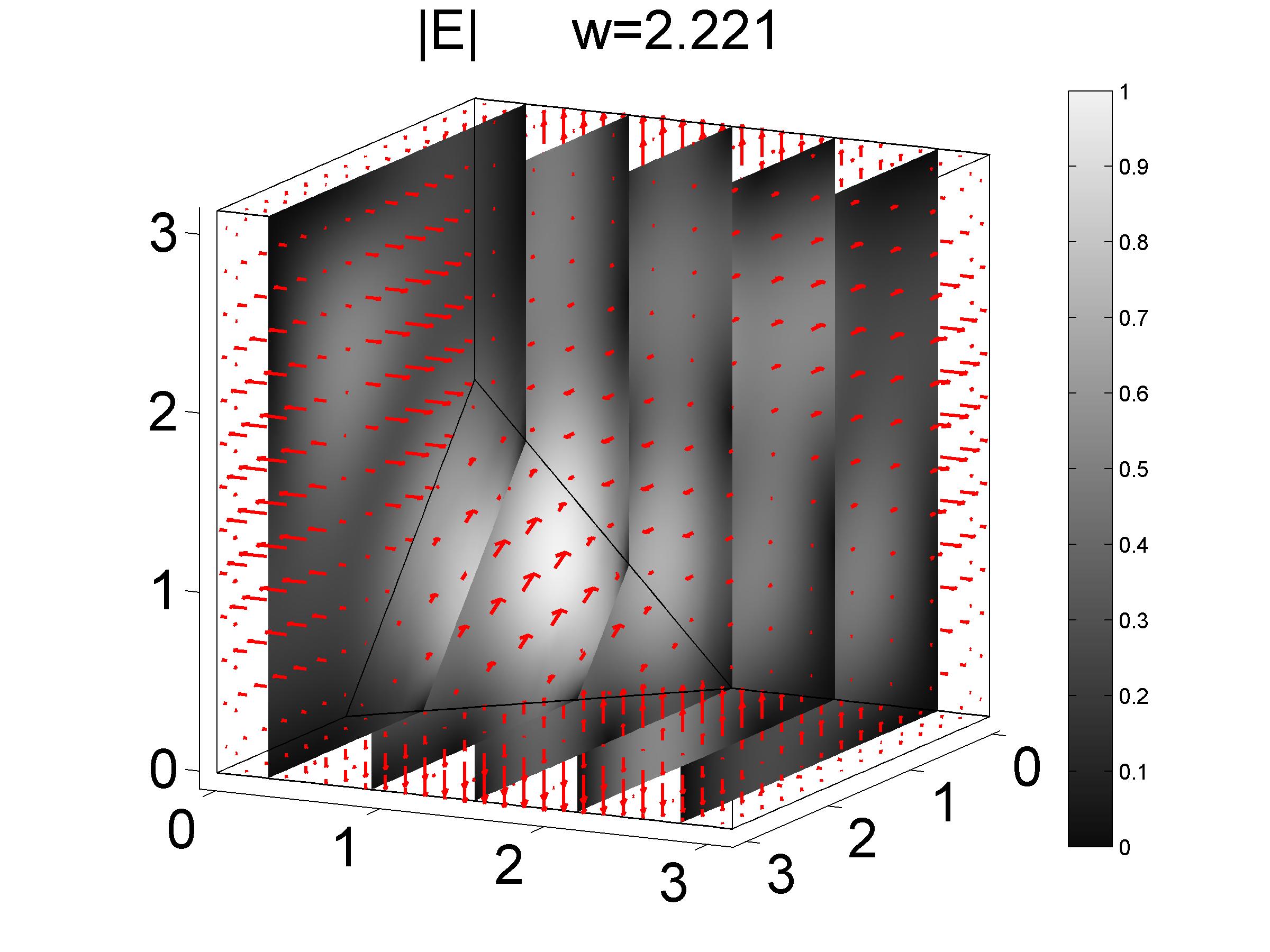} 
}
\rule[-0.1cm]{5cm}{0.01cm}
\centerline{
\includegraphics[height=5cm, angle=0]{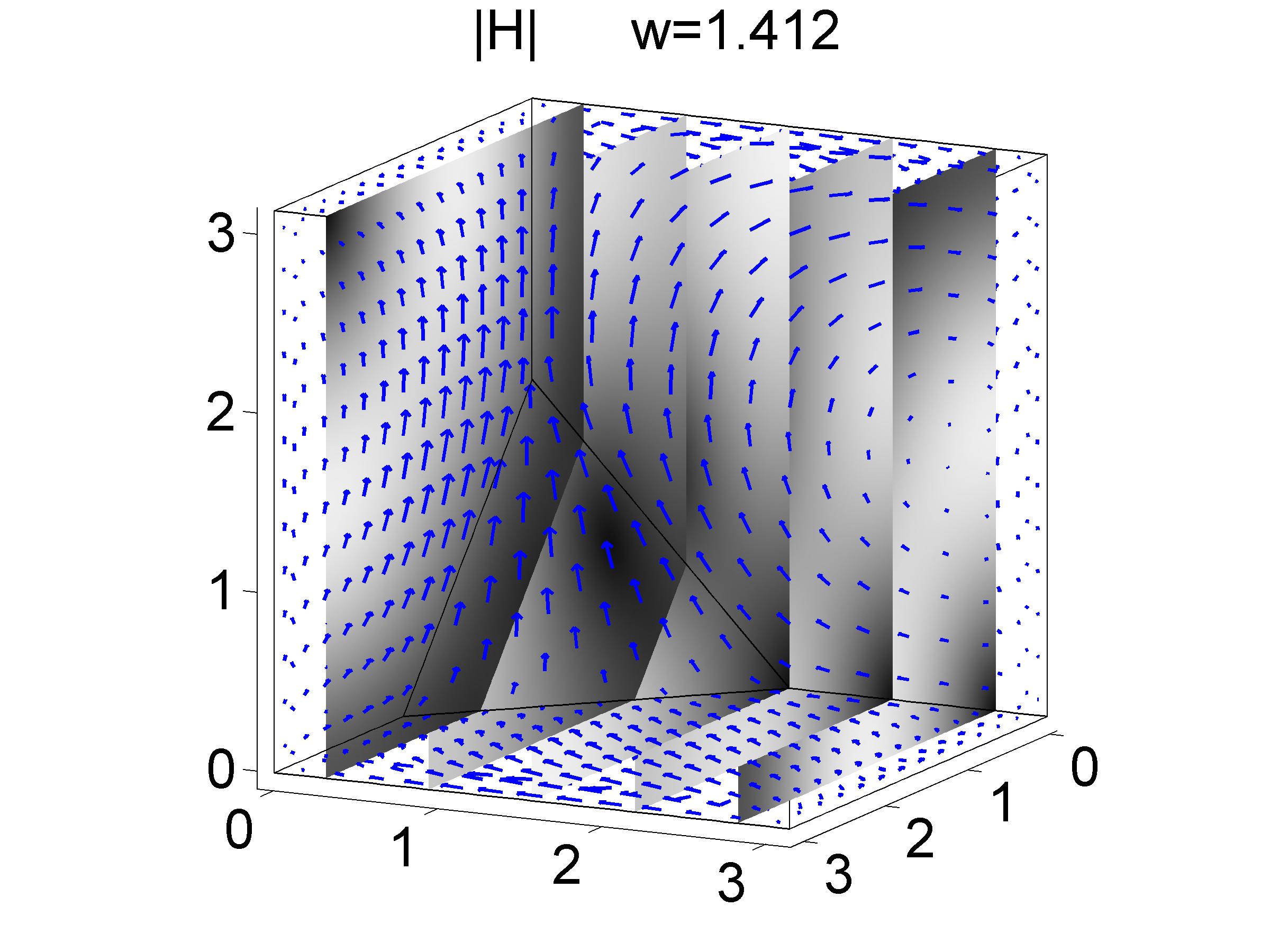} \hspace{-5mm}
\includegraphics[height=5cm, angle=0]{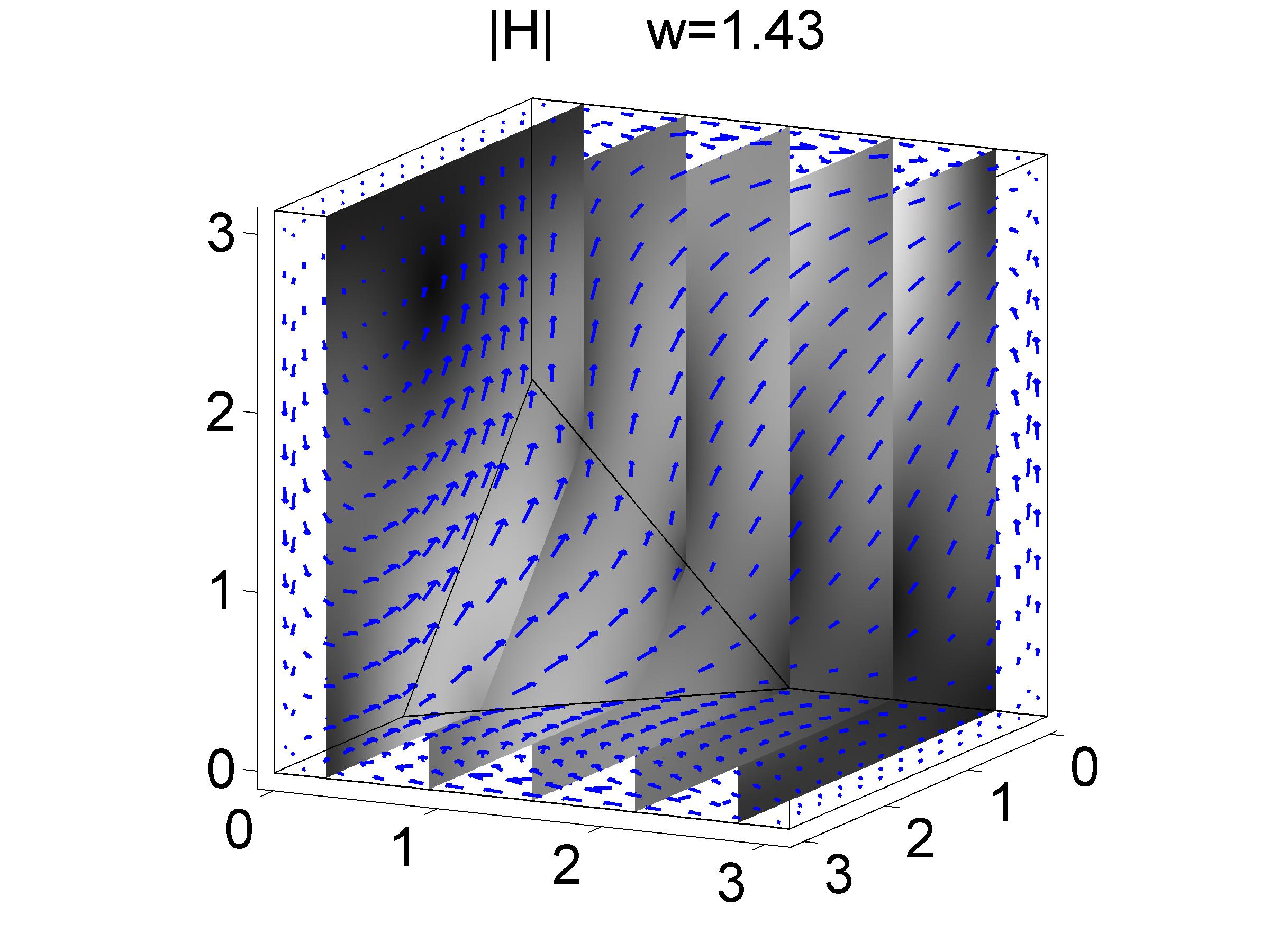} \hspace{-5mm}
\includegraphics[height=5cm, angle=0]{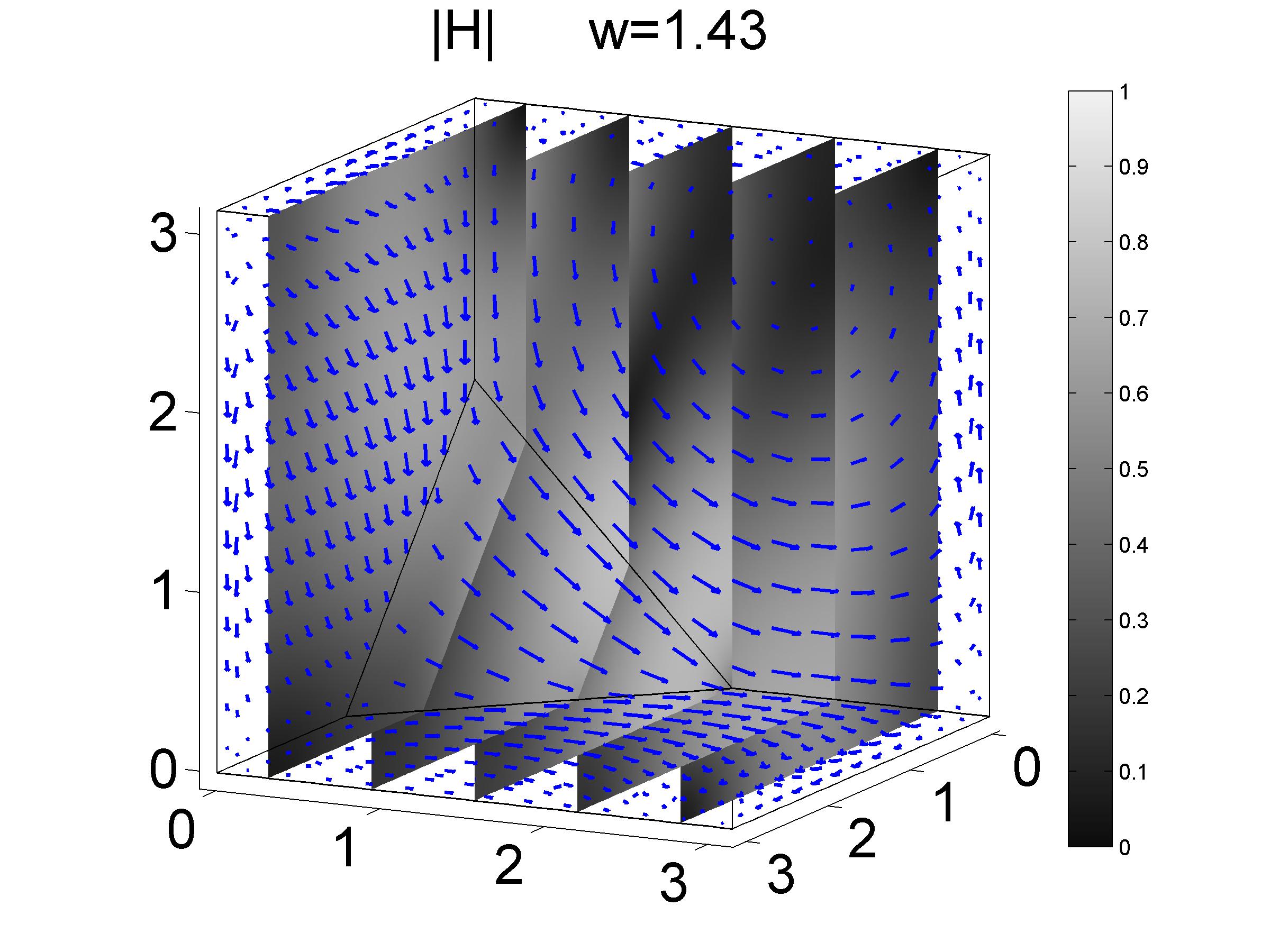} 
}
\centerline{
\includegraphics[height=5cm, angle=0]{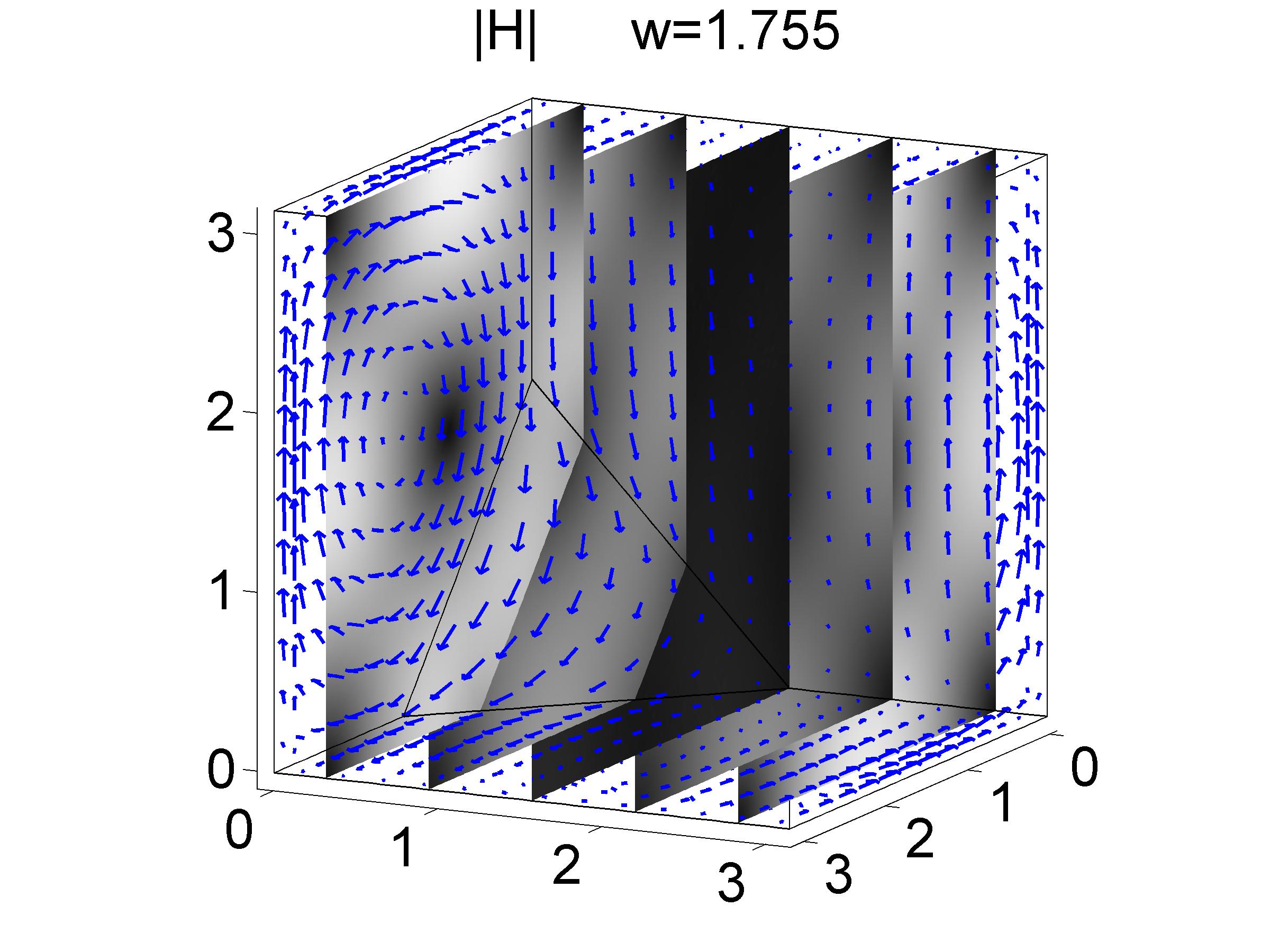} \hspace{-5mm}
\includegraphics[height=5cm, angle=0]{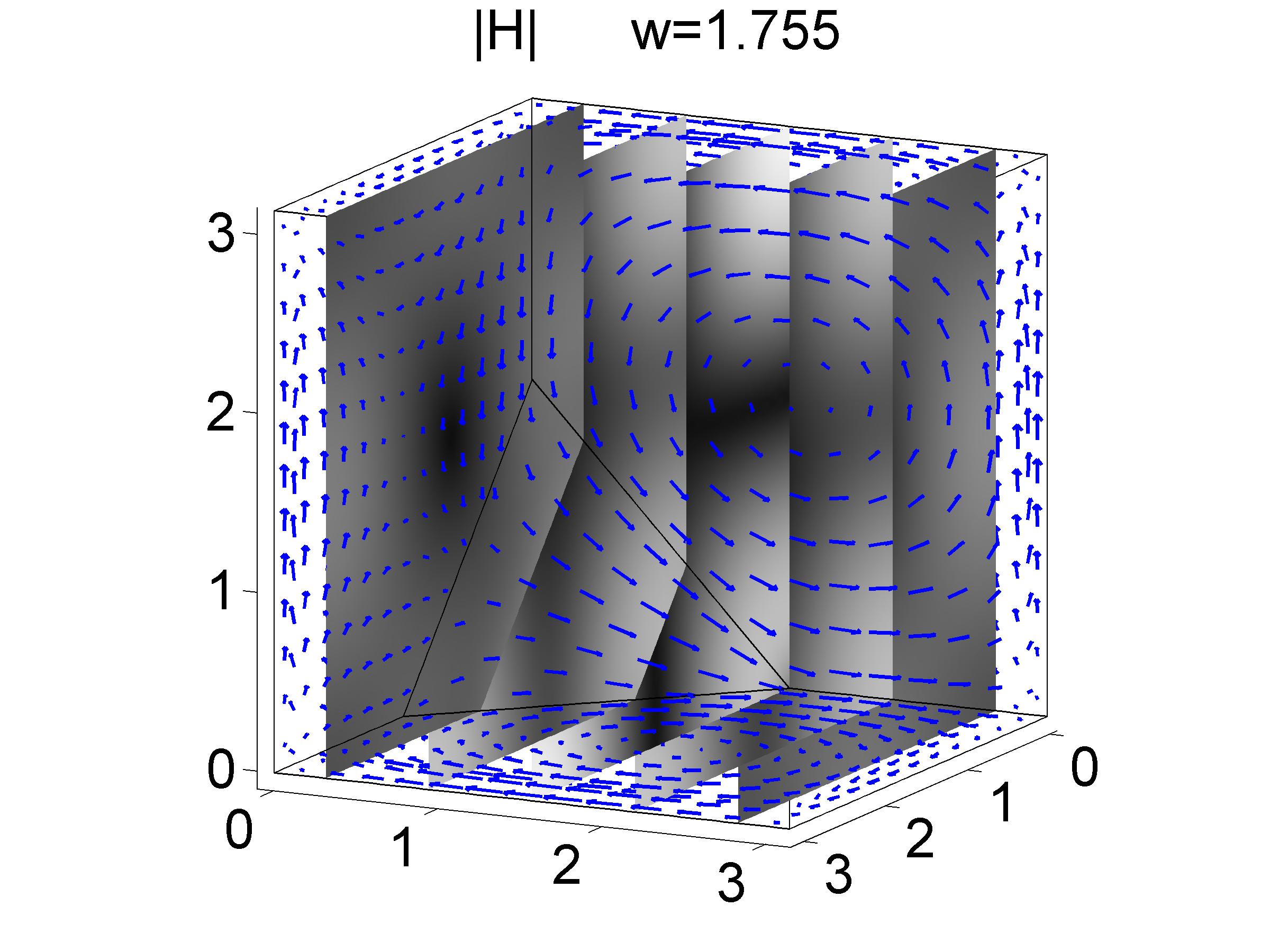} \hspace{-5mm}
\includegraphics[height=5cm, angle=0]{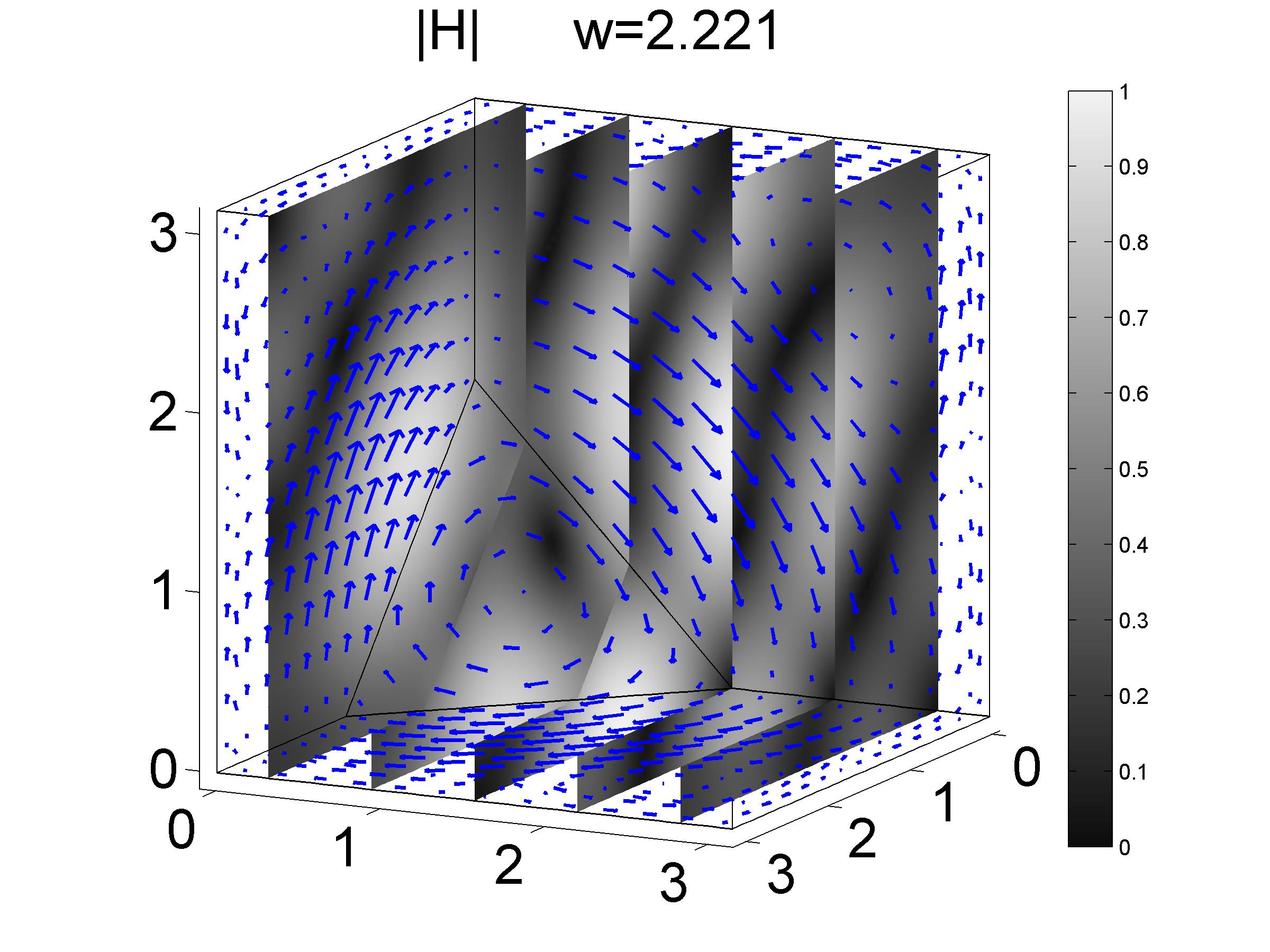} 
}
\caption{ 
The first six eigenfunctions on $\Omega_{\sla}$ for the first six positive
eigenvalues. Densities $|\bE|$ (top) and $|\bH|$ (bottom). Corresponding arrow
fields $\bE$ (red) and $\bH$ (blue) on $\partial \Omega_{\sla}$.  
 \label{cornerless_1}
}
\end{figure}

\begin{figure}
\centerline{
\includegraphics[height=4cm, angle=0]{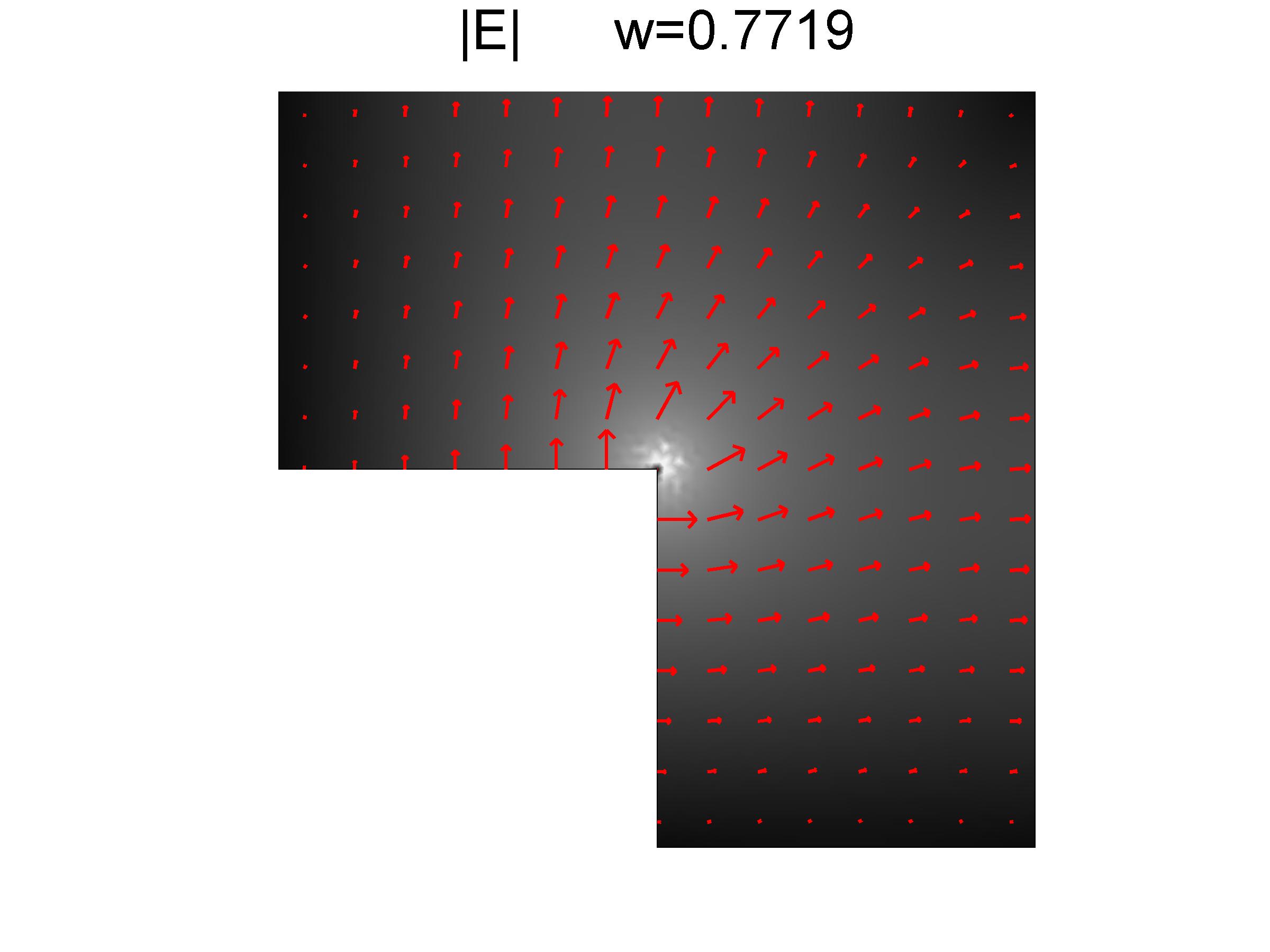} \hspace{-7mm}
\includegraphics[height=4cm, angle=0]{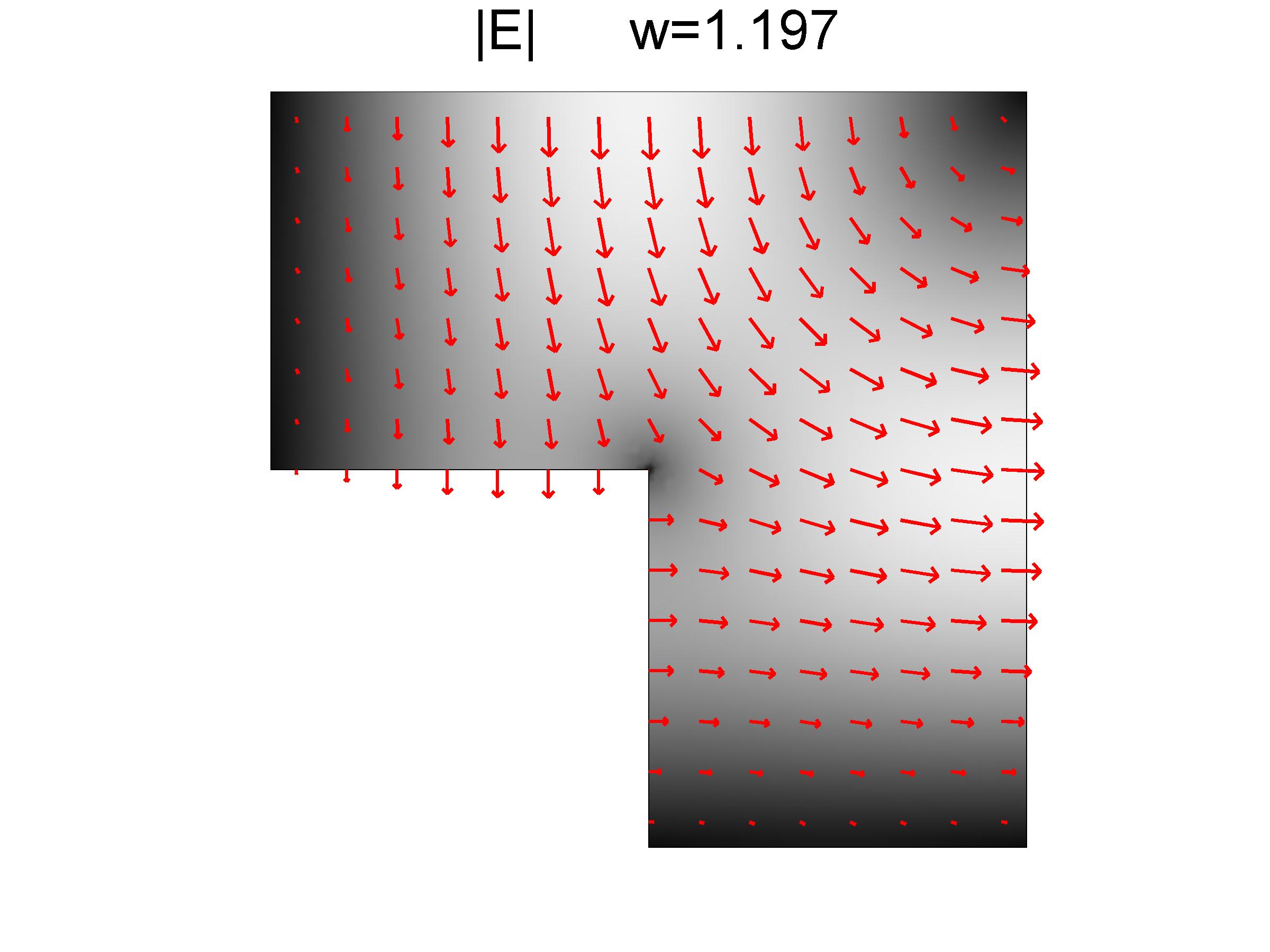} \hspace{-7mm}
\includegraphics[height=4cm, angle=0]{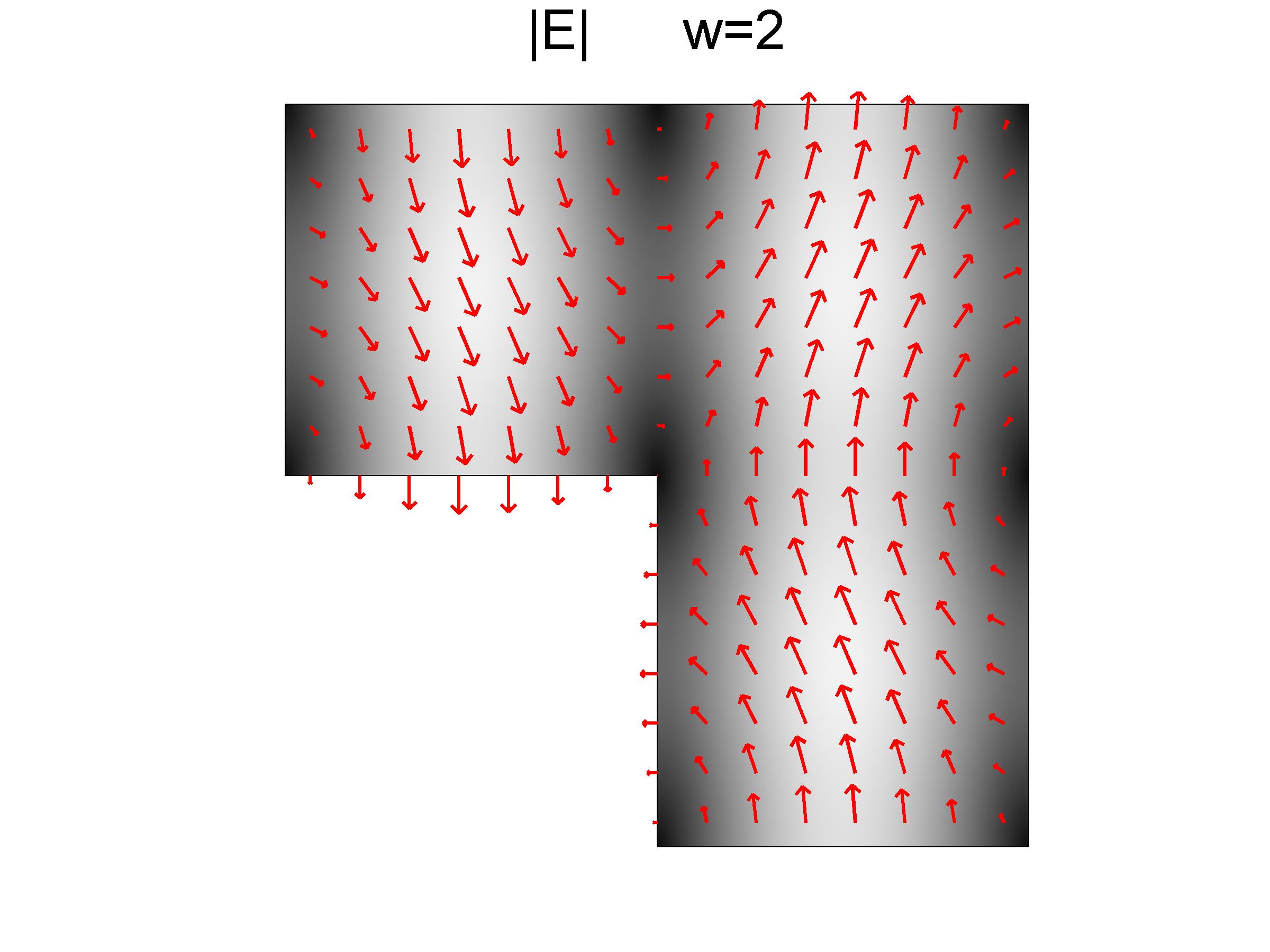} \hspace{-7mm}
\includegraphics[height=4cm, angle=0]{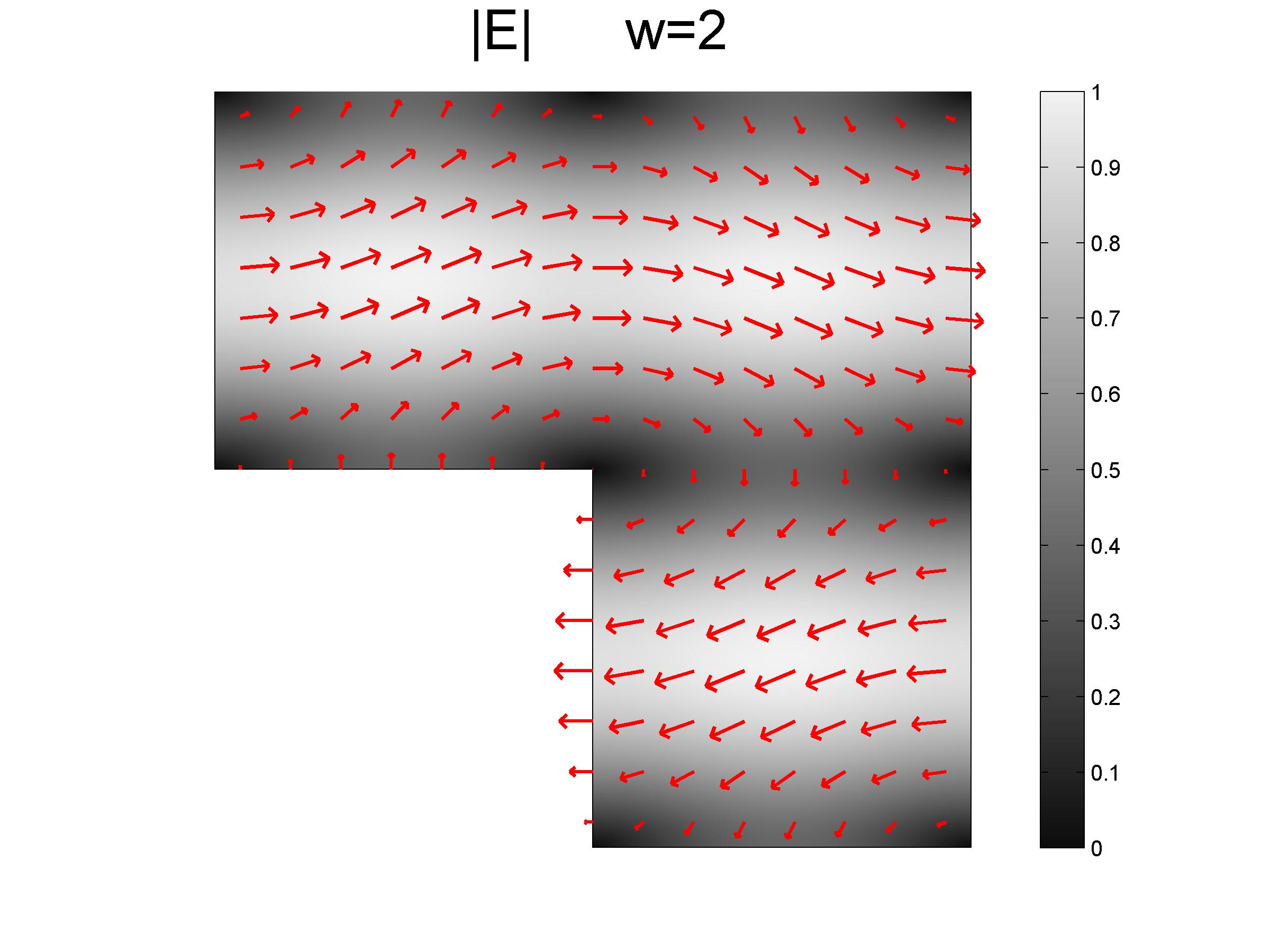} 
}
\centerline{
\includegraphics[height=4cm, angle=0]{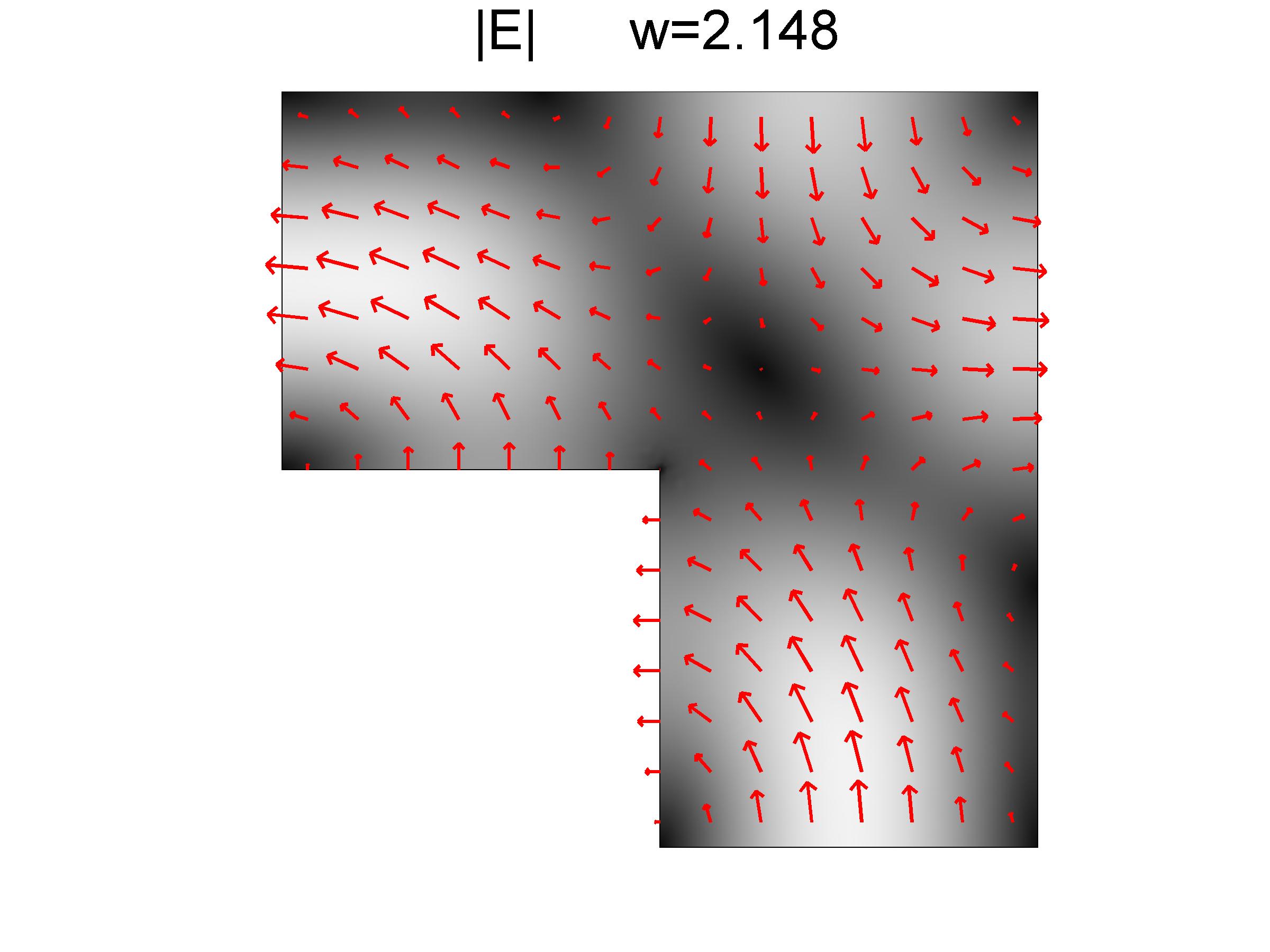} \hspace{-7mm}
\includegraphics[height=4cm, angle=0]{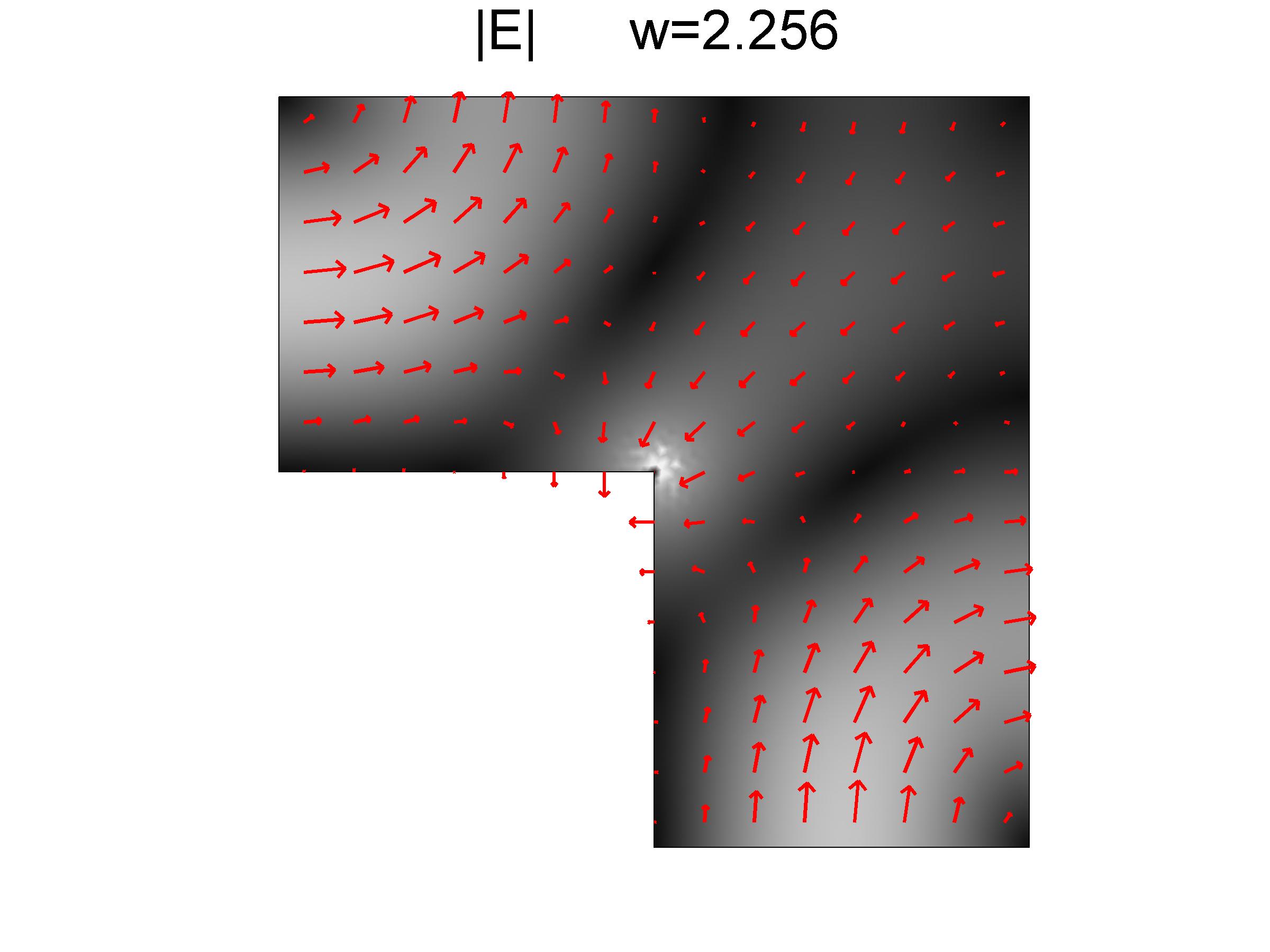} \hspace{-7mm}
\includegraphics[height=4cm, angle=0]{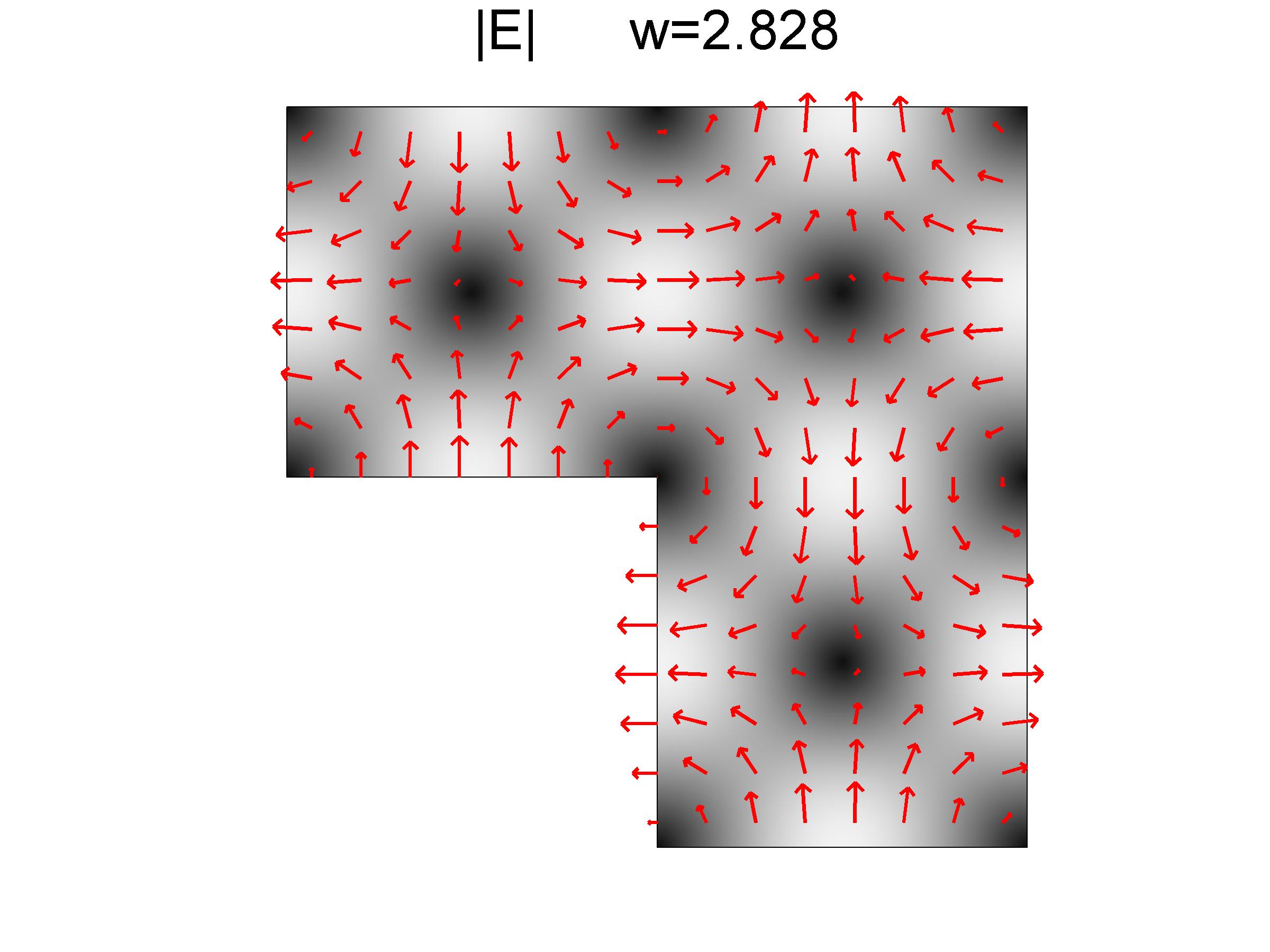} \hspace{-7mm}
\includegraphics[height=4cm, angle=0]{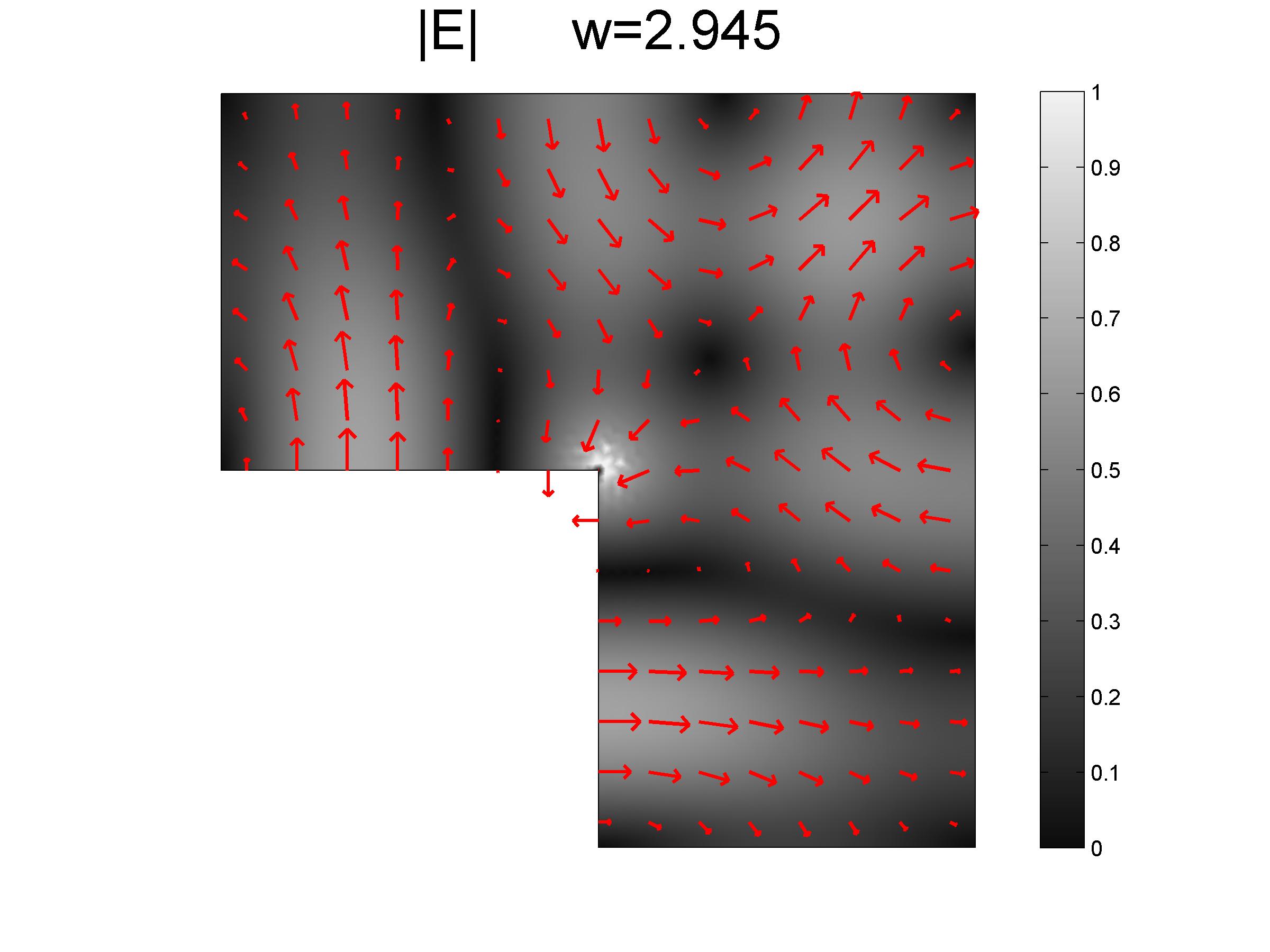}}
\rule[-0.1cm]{5cm}{0.01cm}
\vspace{1cm} \
\centerline{
\includegraphics[height=4cm, angle=0]{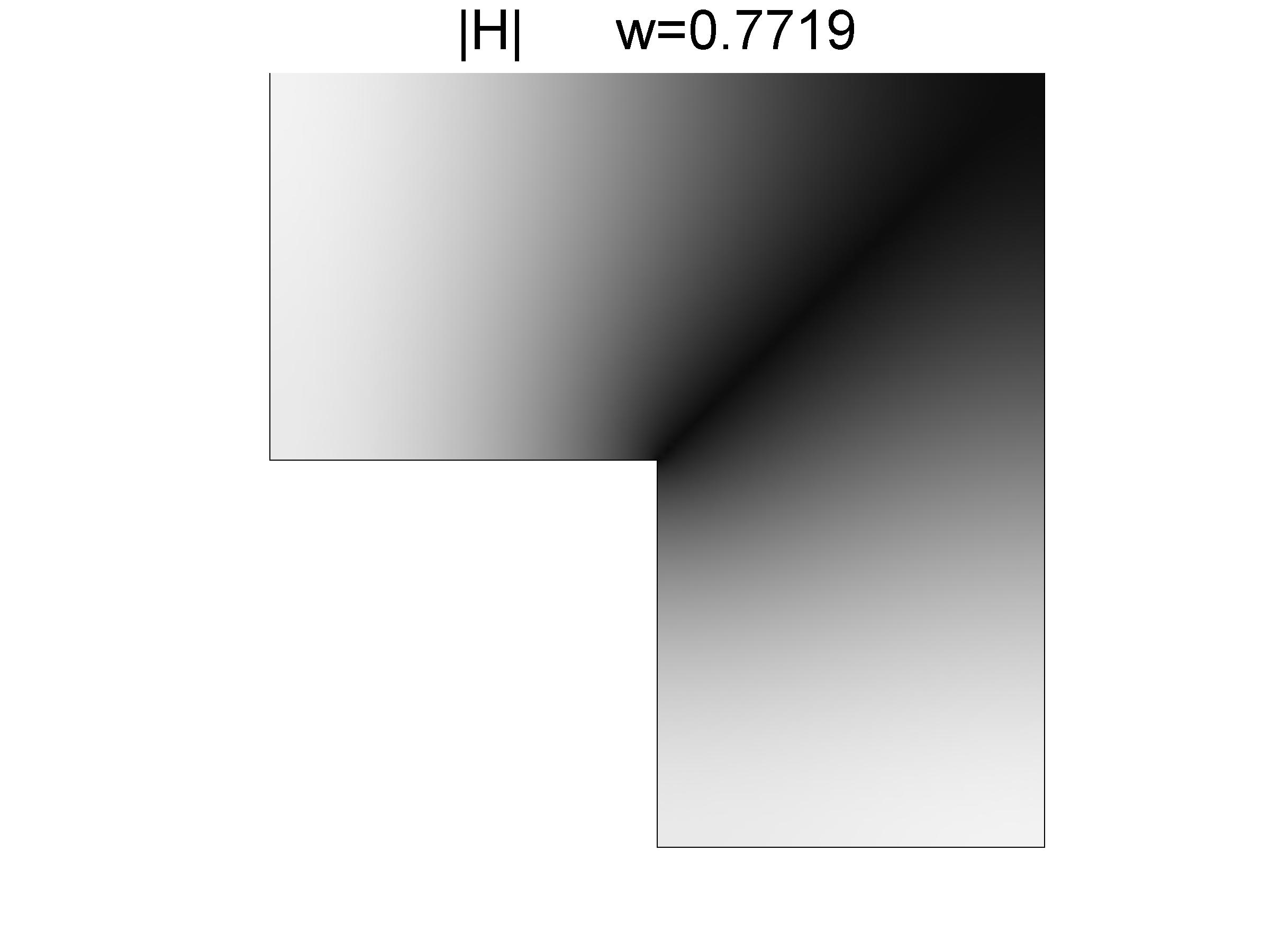} \hspace{-7mm}
\includegraphics[height=4cm, angle=0]{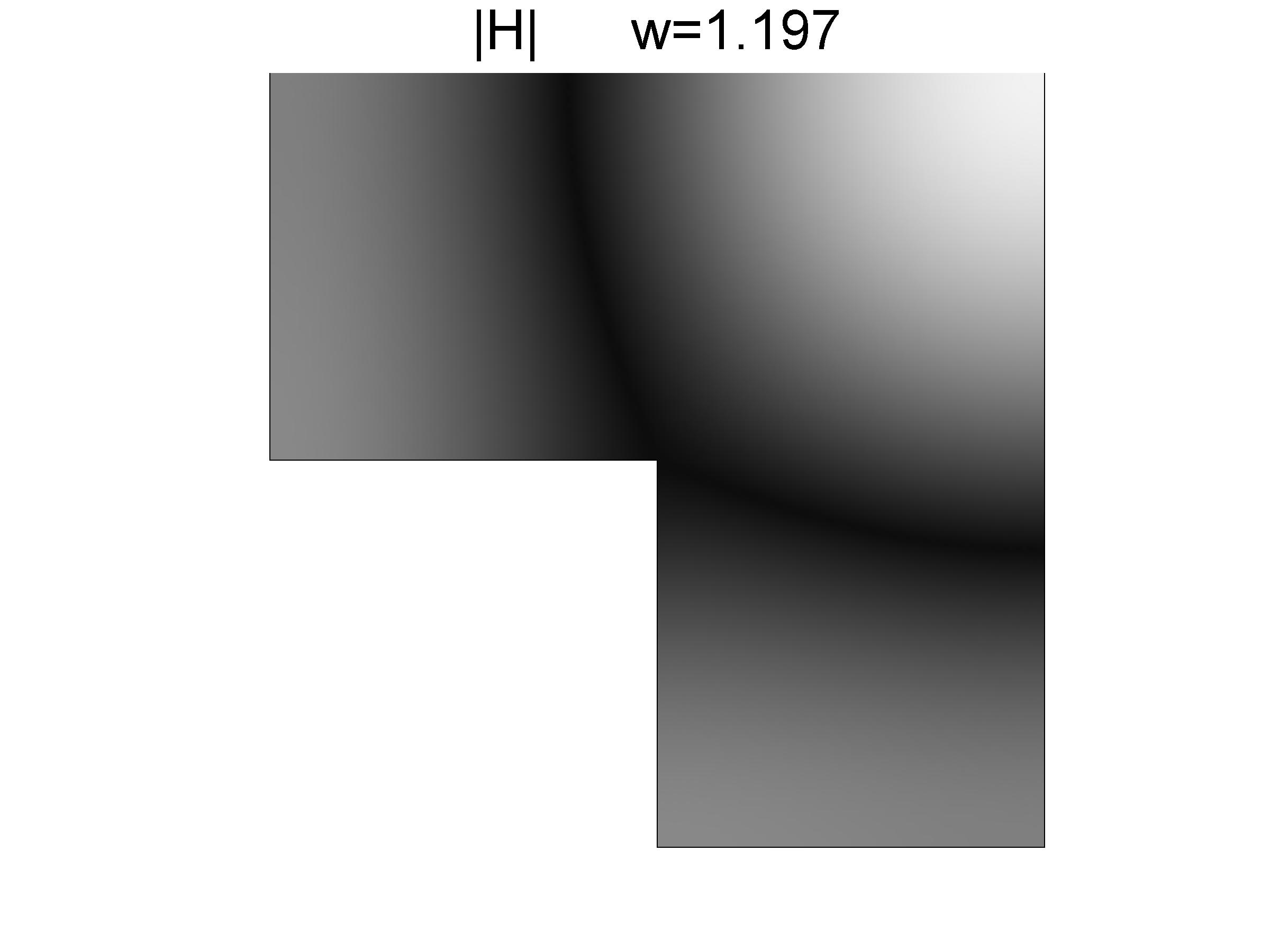} \hspace{-7mm}
\includegraphics[height=4cm, angle=0]{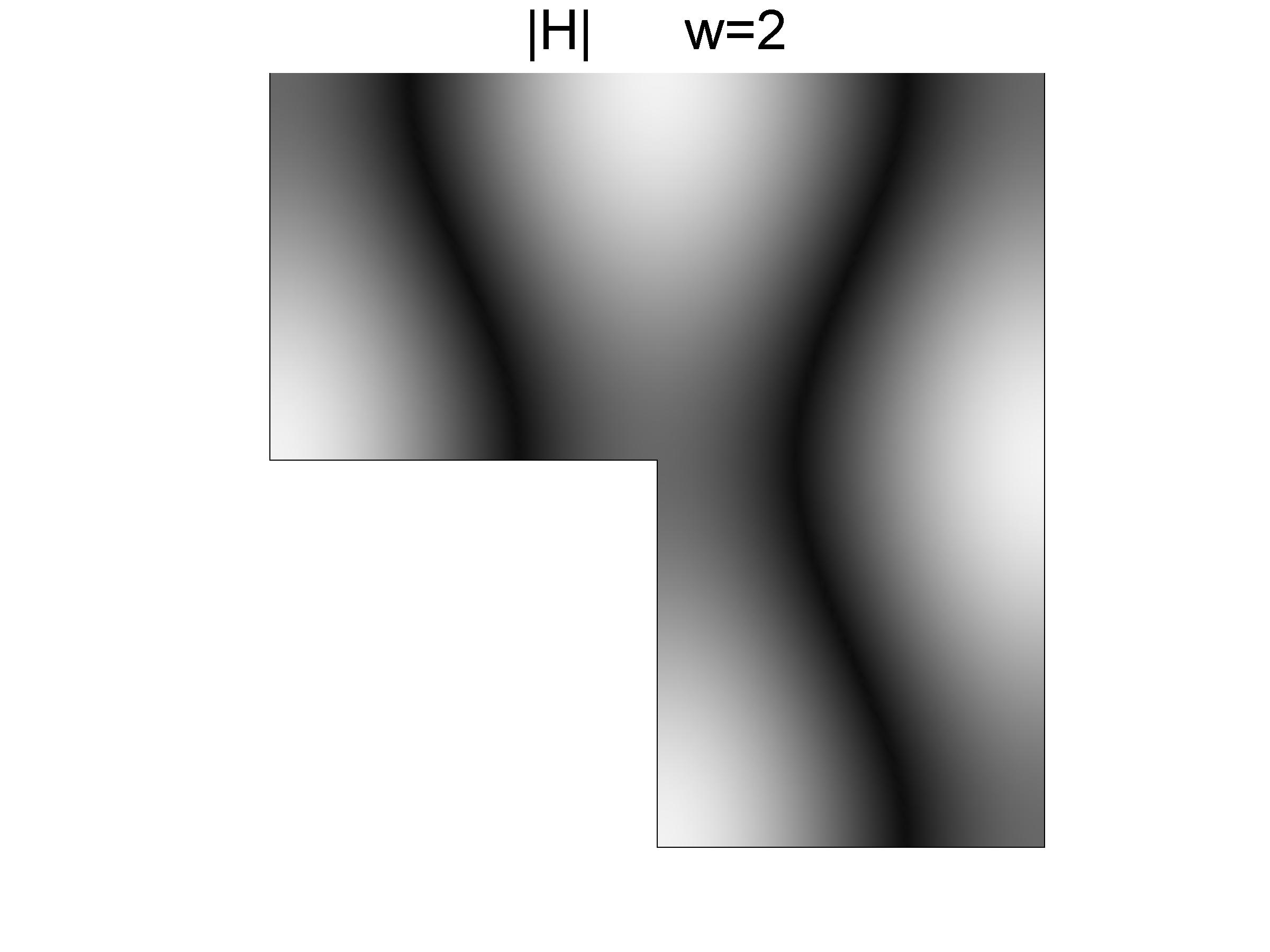} \hspace{-7mm}
\includegraphics[height=4cm, angle=0]{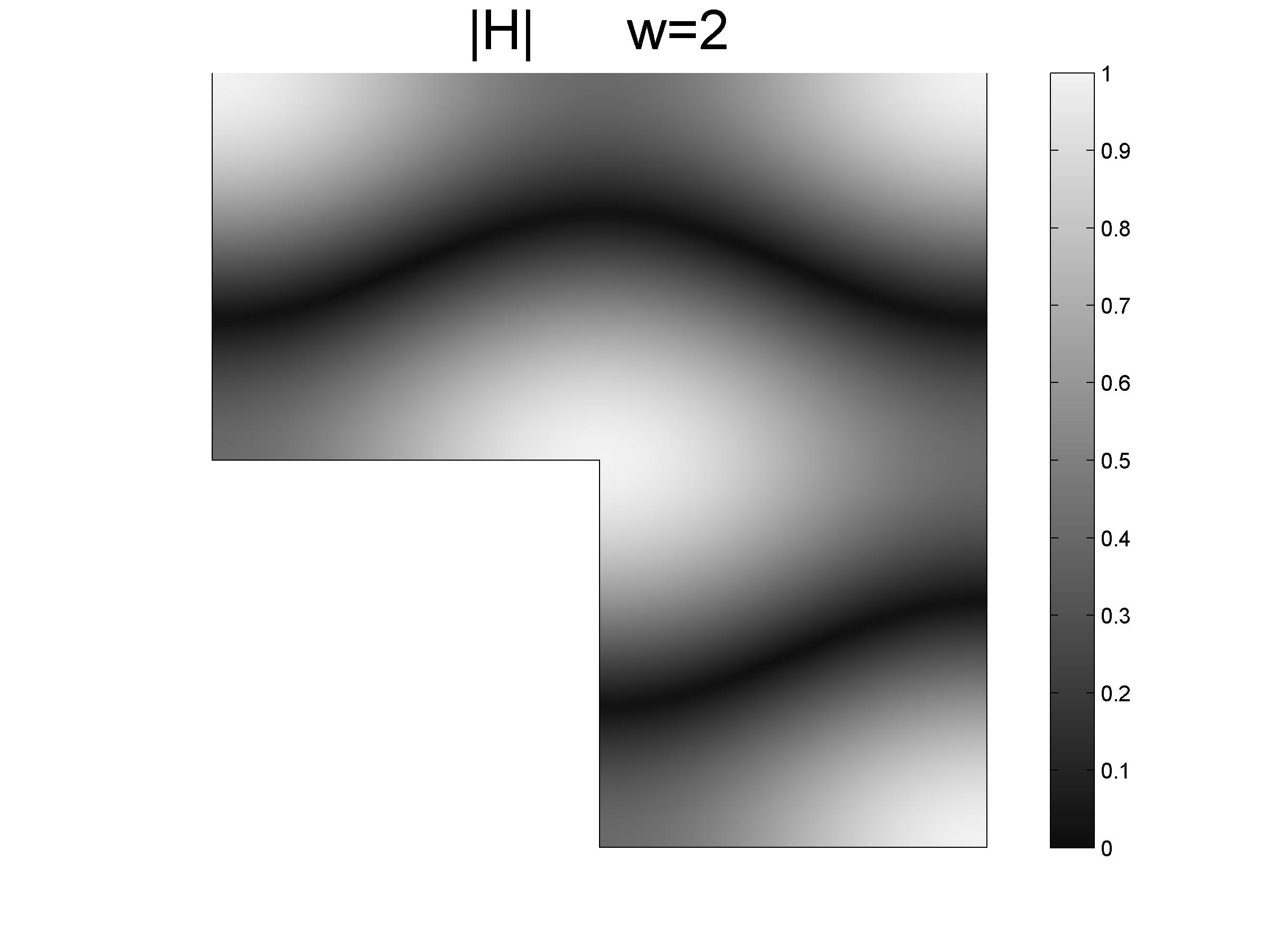} 
}
\centerline{
\includegraphics[height=4cm, angle=0]{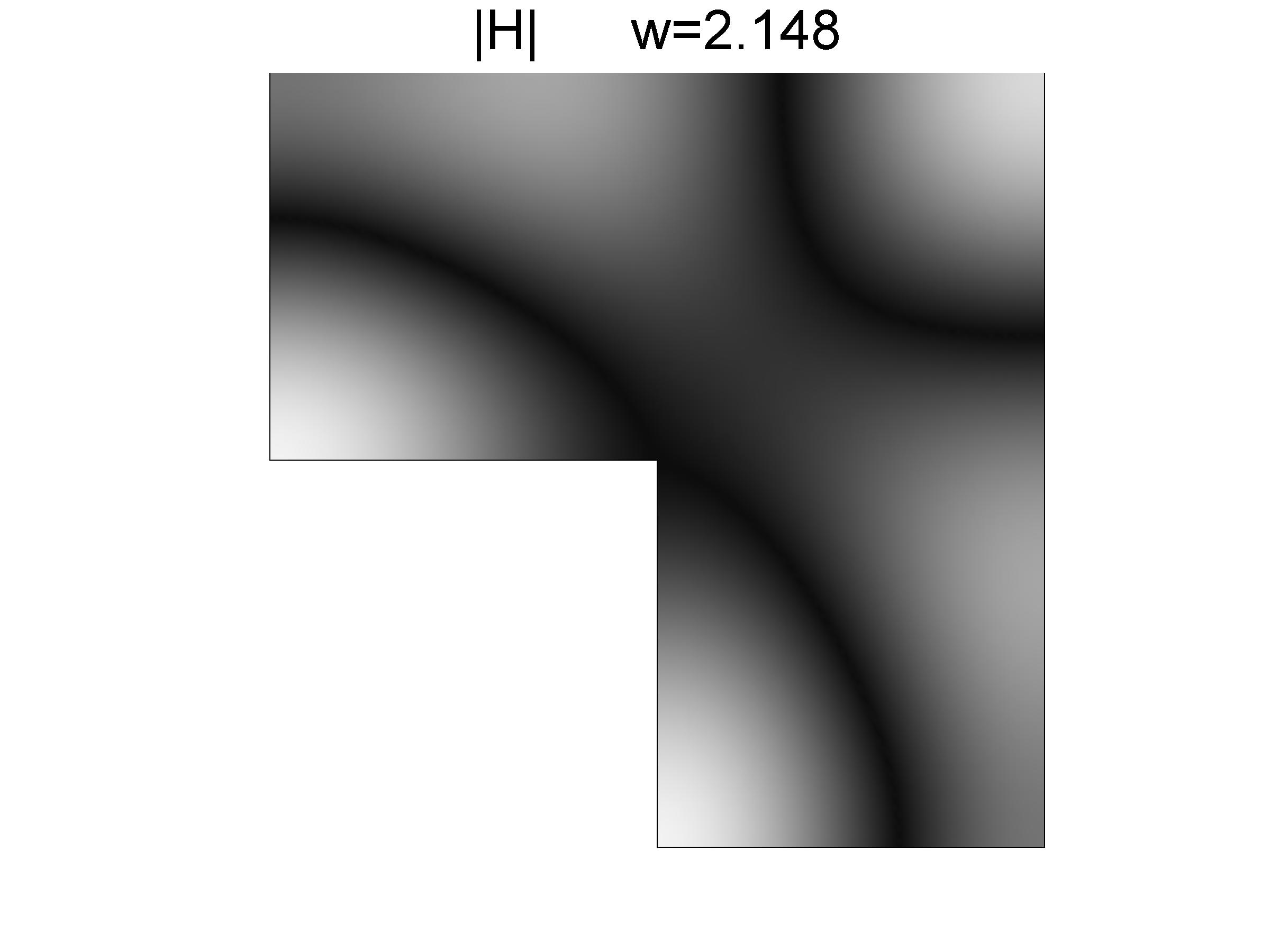} \hspace{-7mm}
\includegraphics[height=4cm, angle=0]{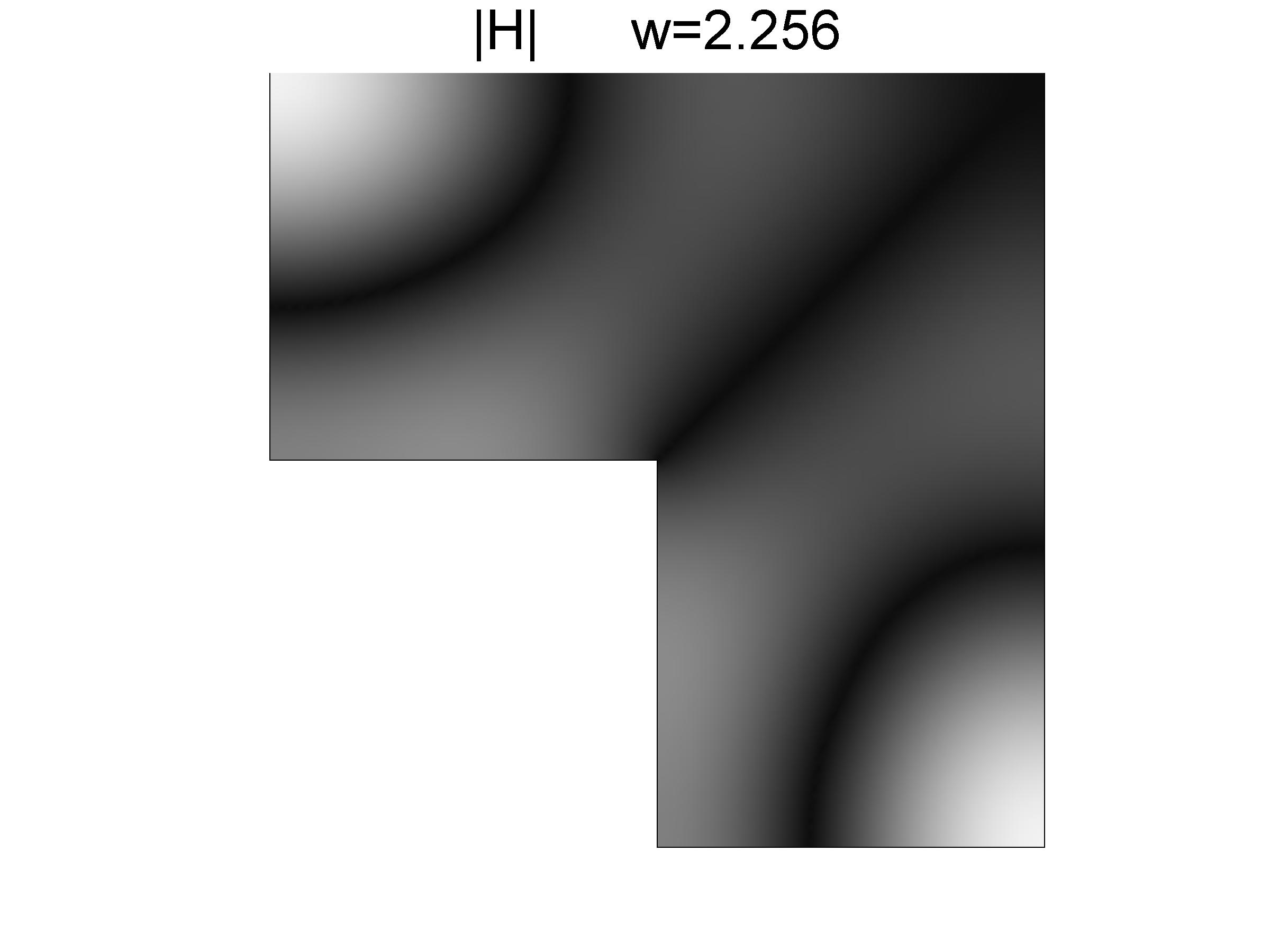} \hspace{-7mm}
\includegraphics[height=4cm, angle=0]{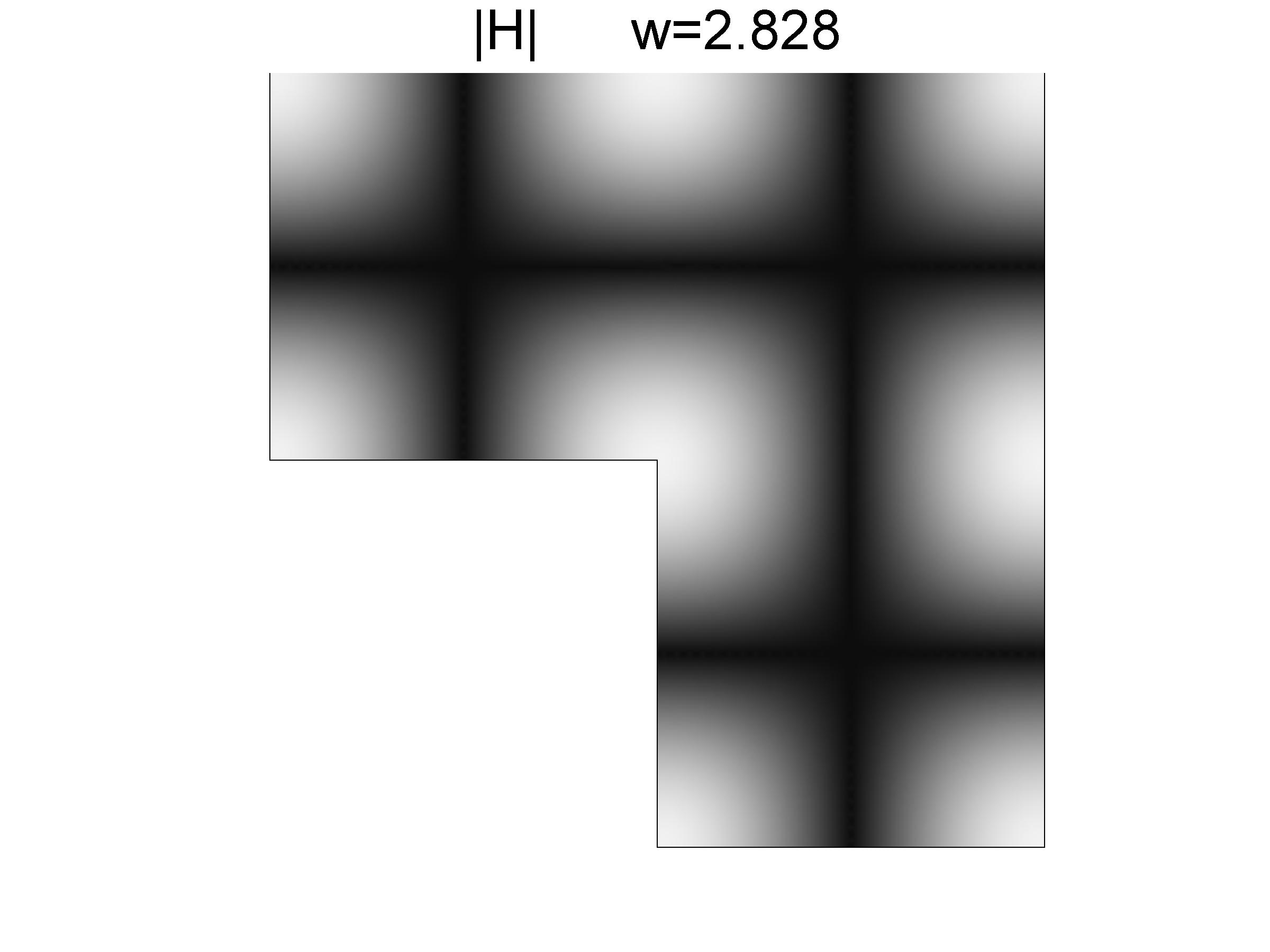} \hspace{-7mm}
\includegraphics[height=4cm, angle=0]{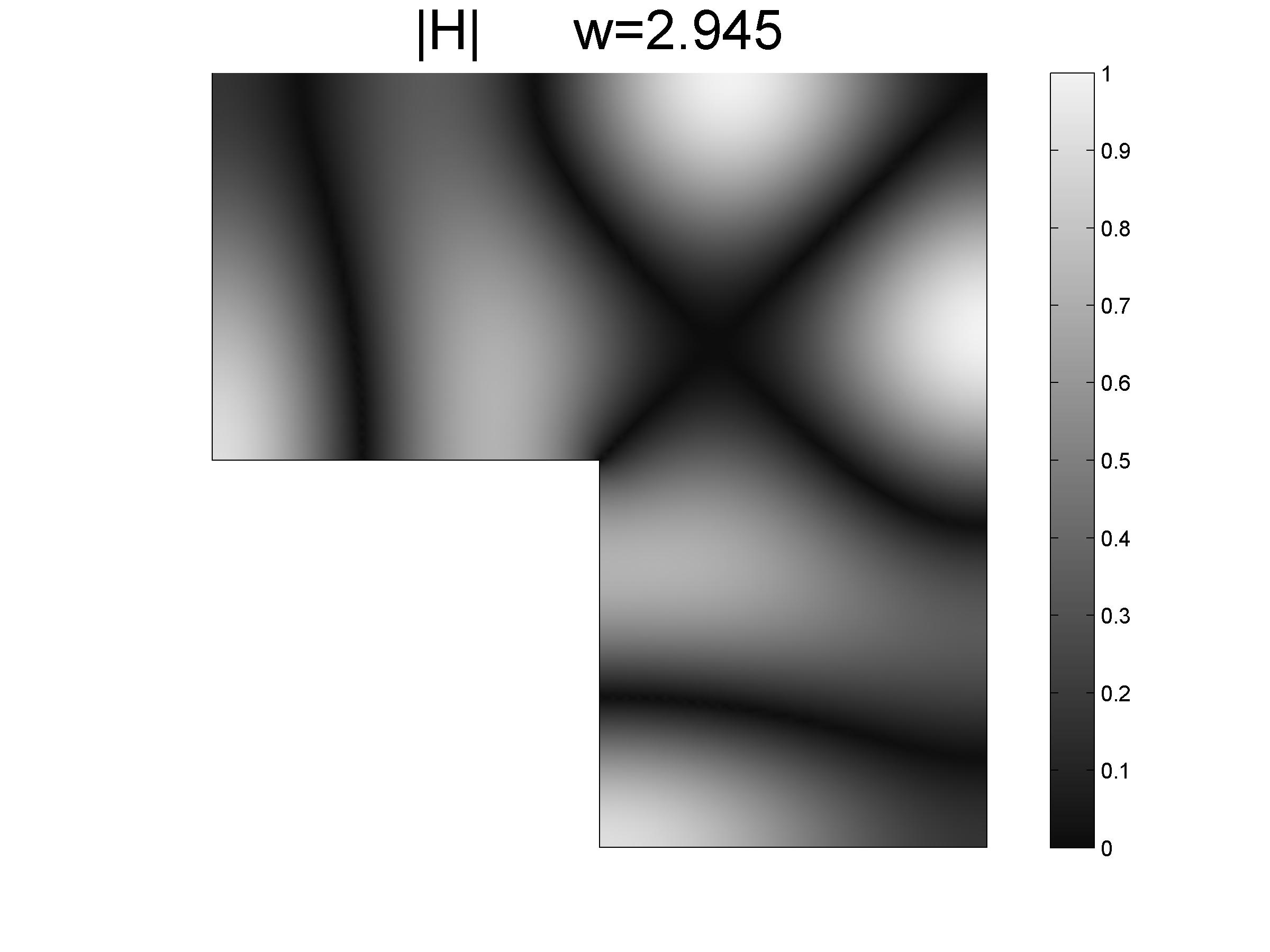}}
\caption{Eigenfunctions on $\tilde{\Omega}_{\lsp}$ associated to the first eight
positive eigenvalues. Densities $|\bE|$ (top) and $|H|$ (bottom). Corresponding
arrow fields $\bE$. We have re-scaled each individual density to have as
maximum the value 1.
 \label{Lshape_eigenfunctions}
}
\end{figure}

\begin{figure}
\centerline{
\includegraphics[height=5cm, angle=0]{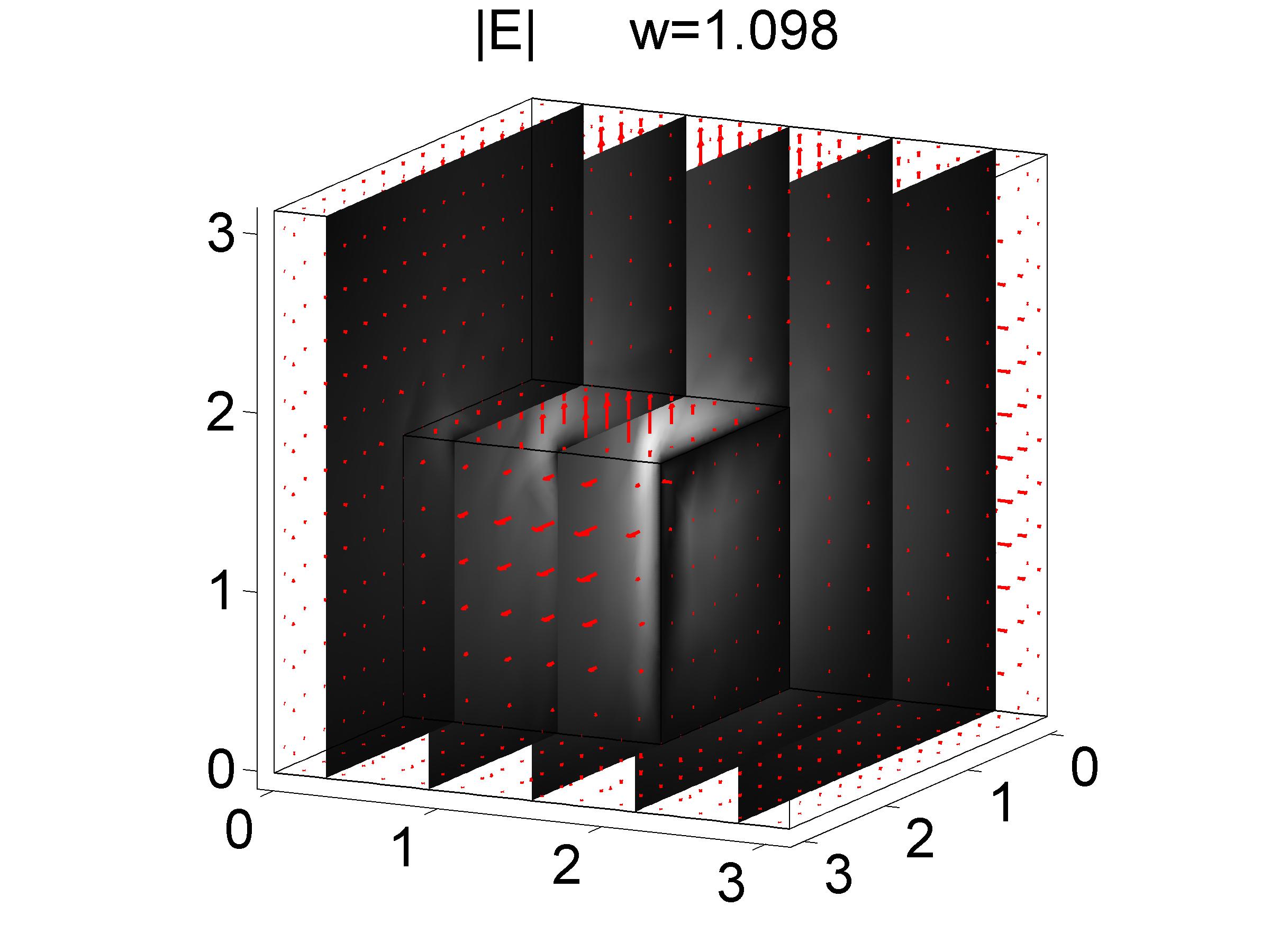} \hspace{-5mm}
\includegraphics[height=5cm, angle=0]{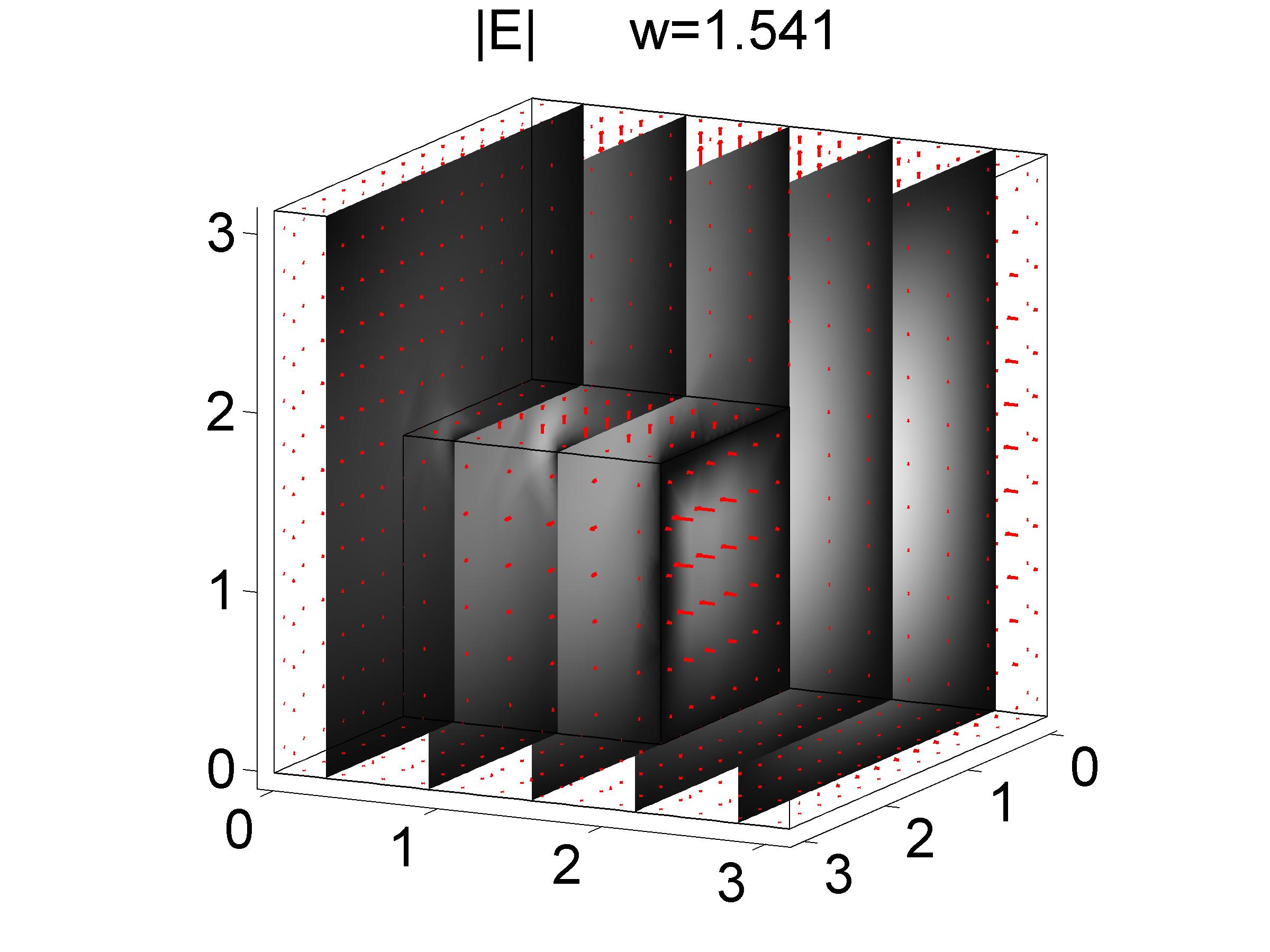} \hspace{-5mm}
\includegraphics[height=5cm, angle=0]{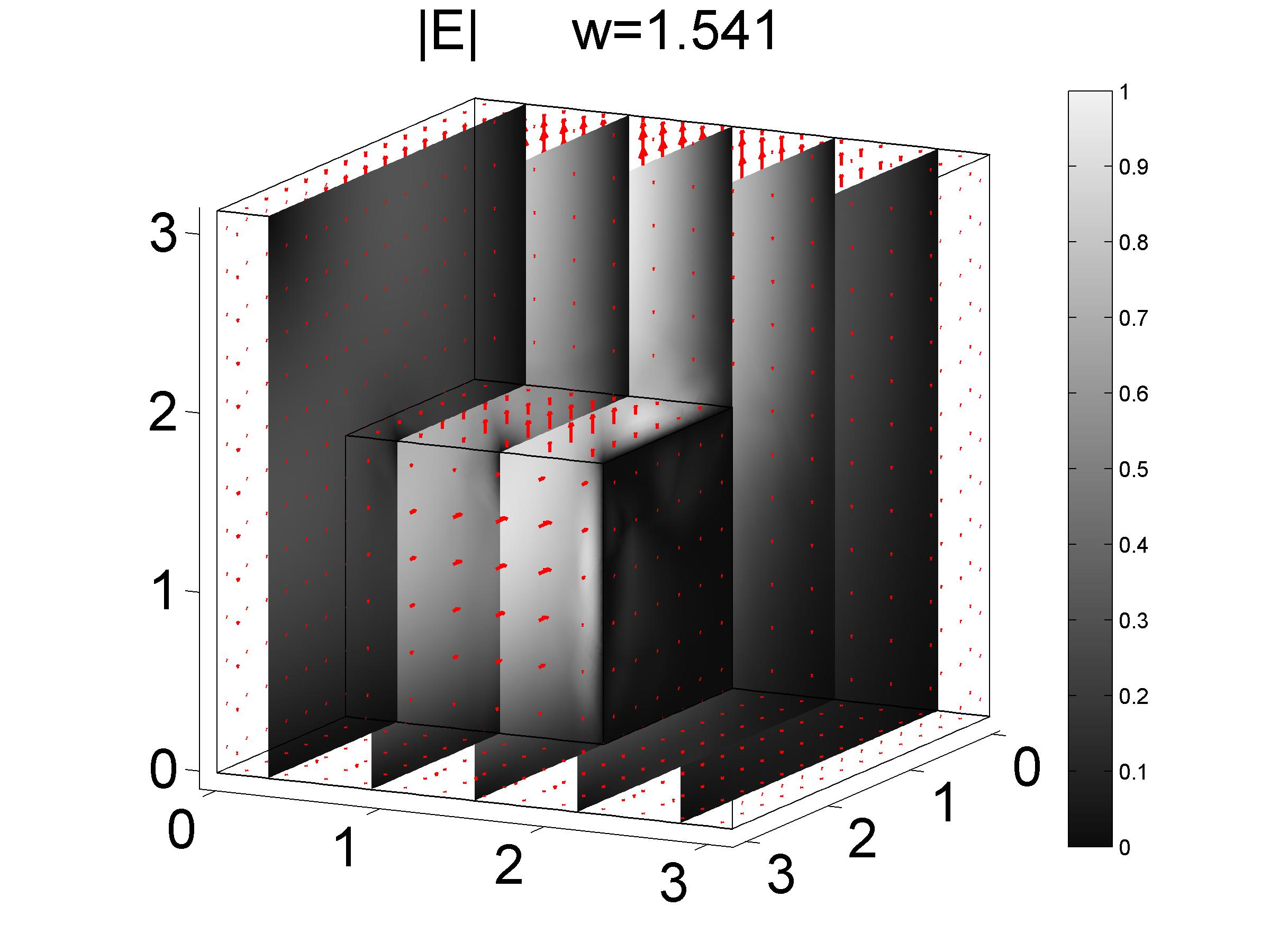} 
}
\centerline{
\includegraphics[height=5cm, angle=0]{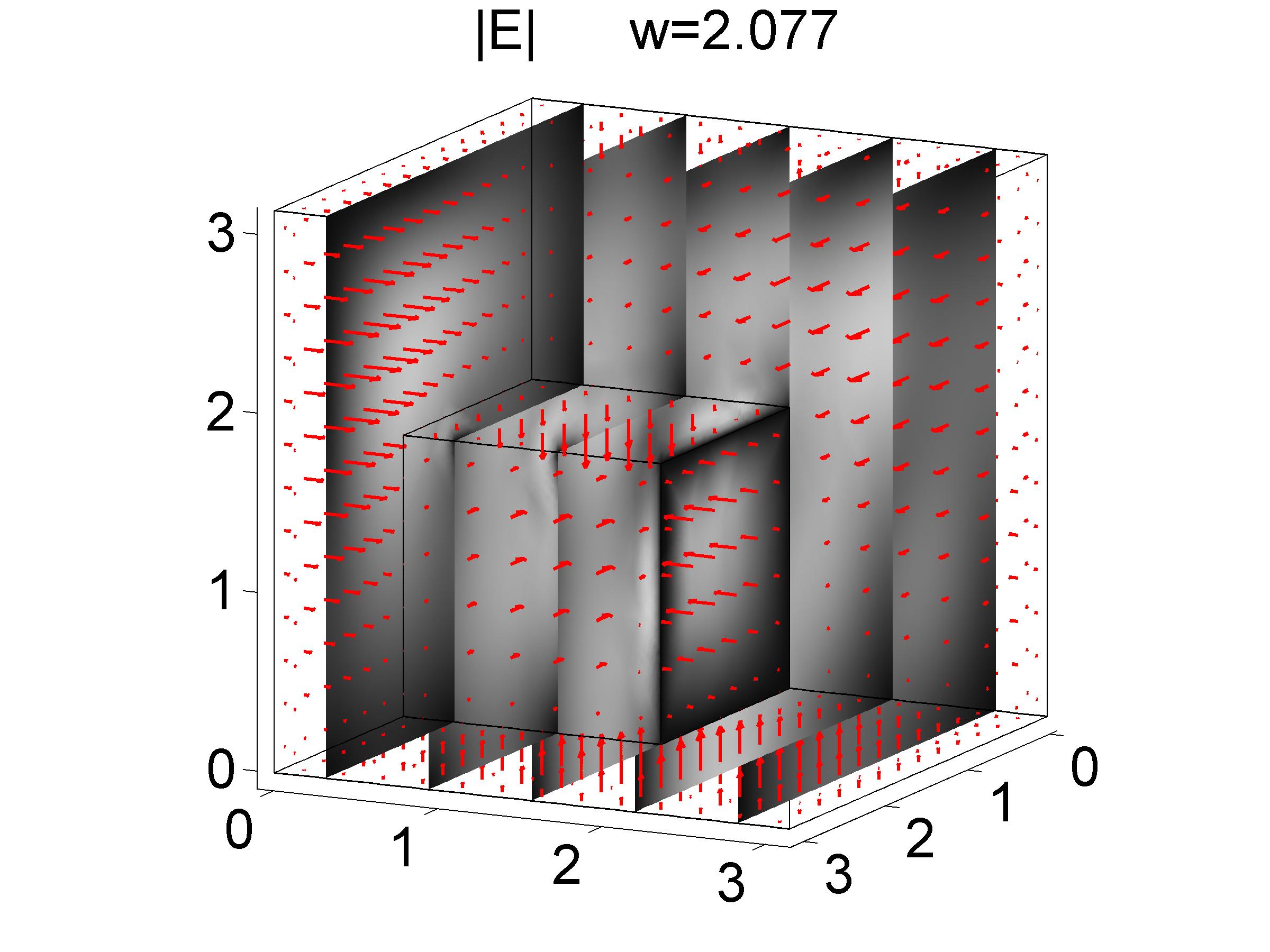} \hspace{-5mm}
\includegraphics[height=5cm, angle=0]{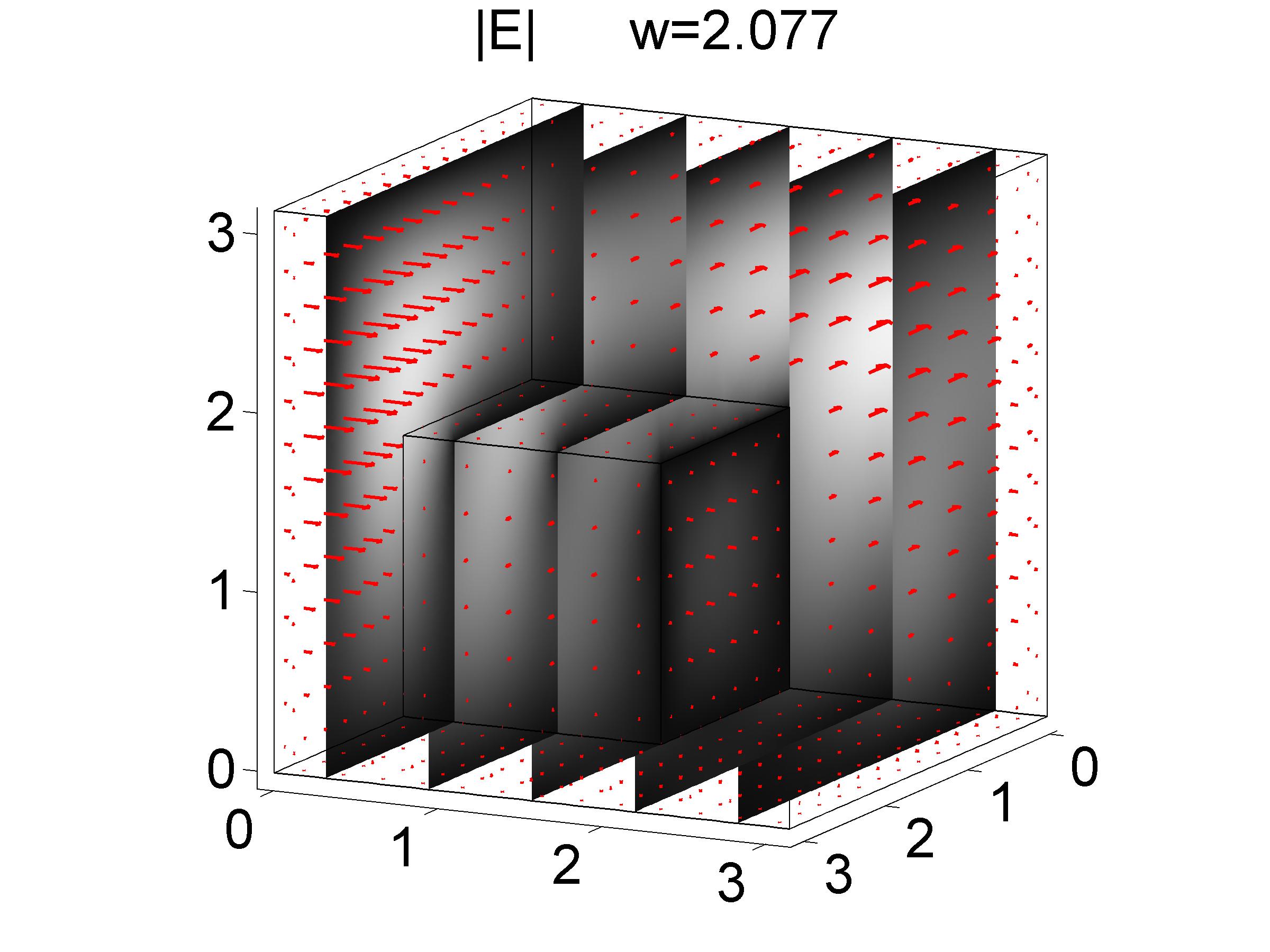} \hspace{-5mm}
\includegraphics[height=5cm, angle=0]{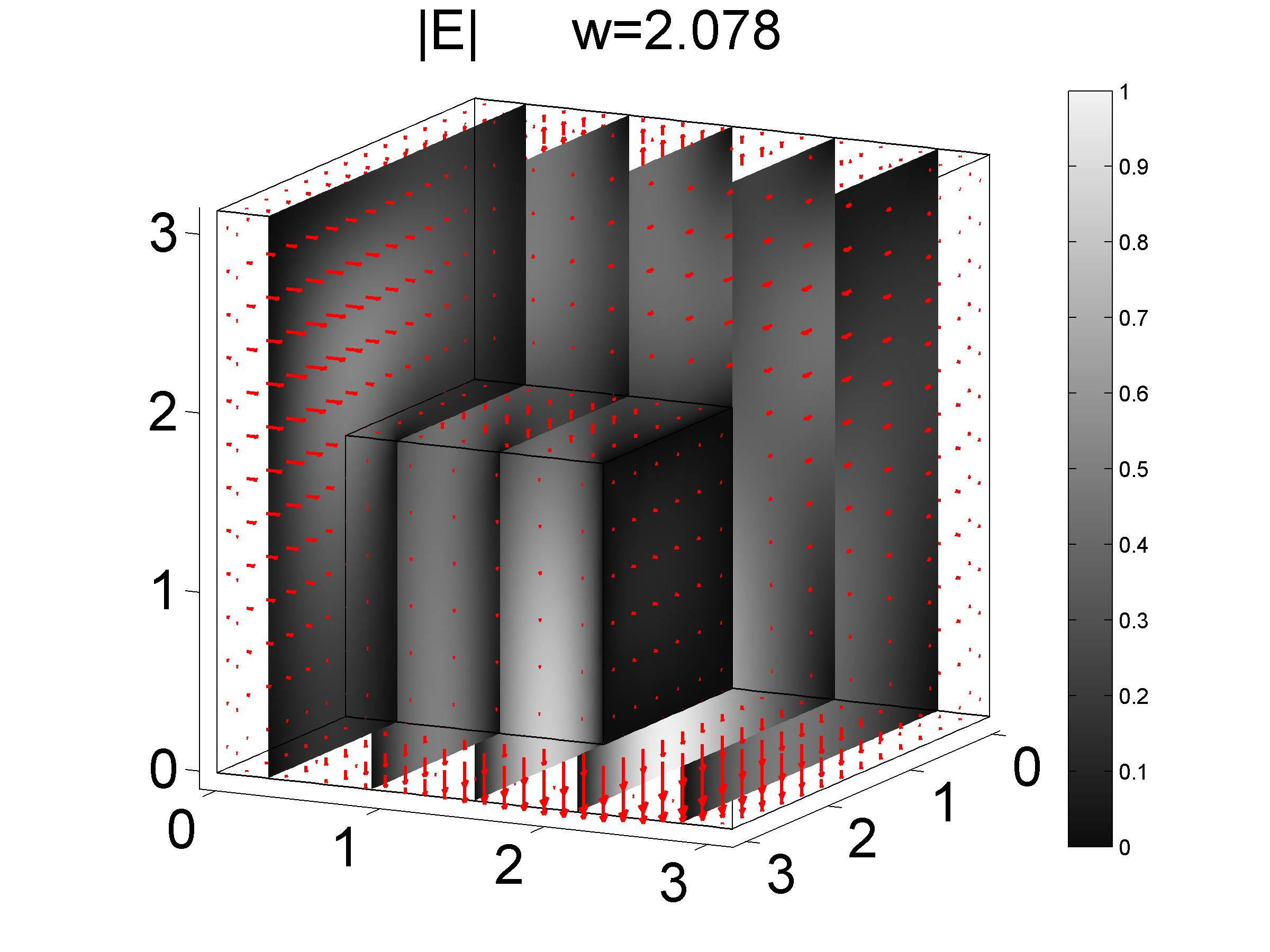} 
}
\rule[-0.1cm]{5cm}{0.01cm}
\centerline{
\includegraphics[height=5cm, angle=0]{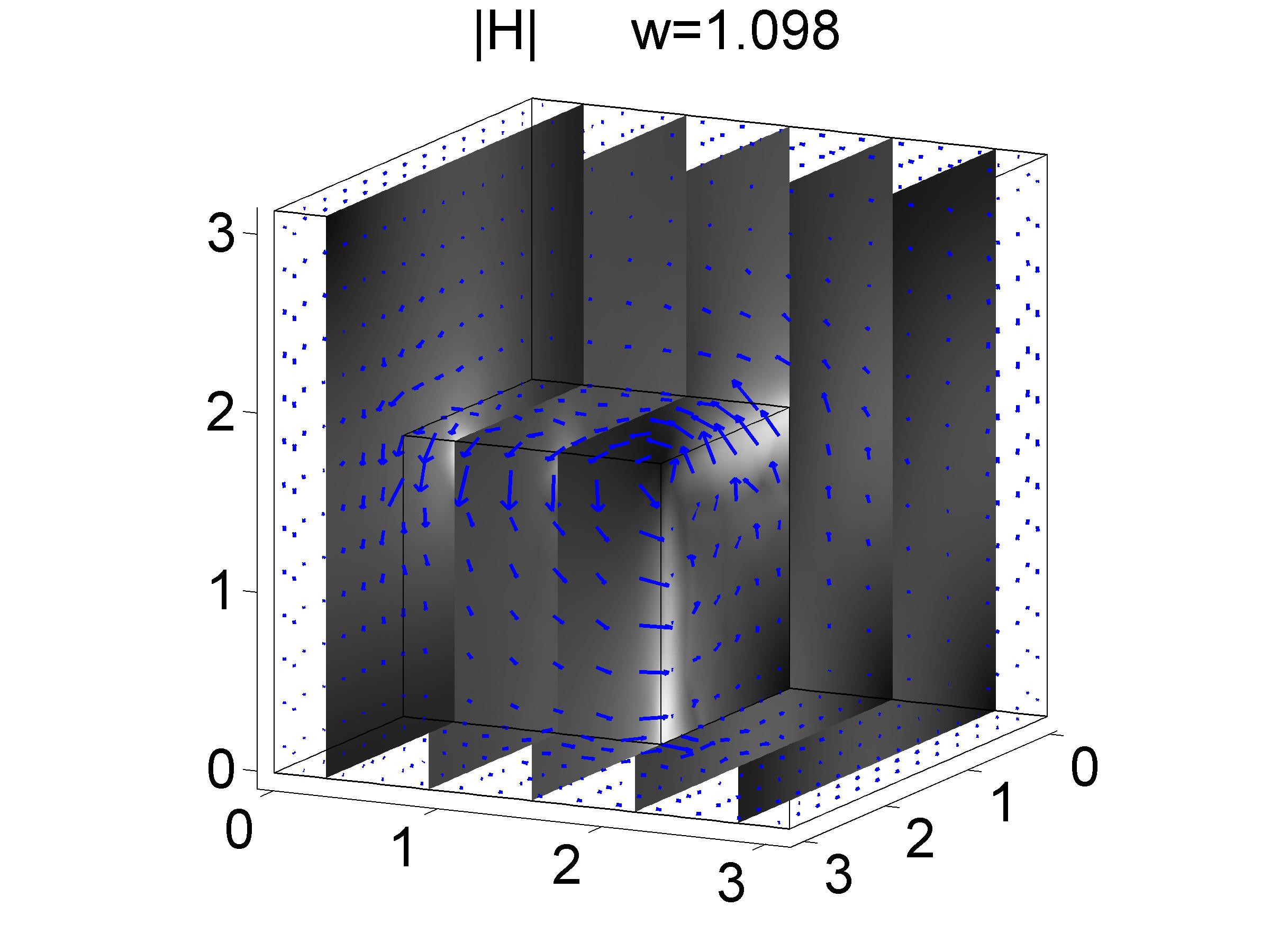} \hspace{-5mm}
\includegraphics[height=5cm, angle=0]{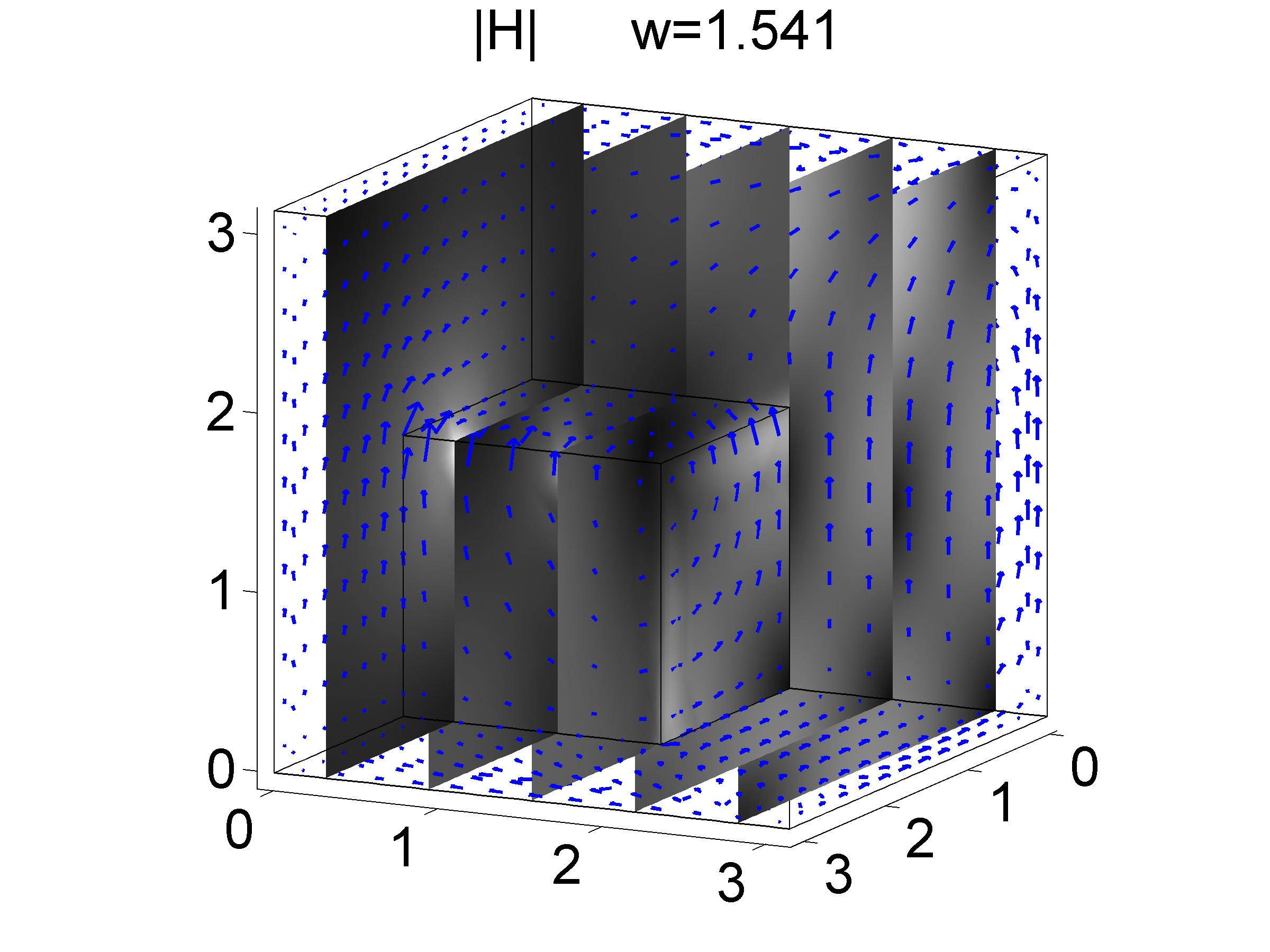} \hspace{-5mm}
\includegraphics[height=5cm, angle=0]{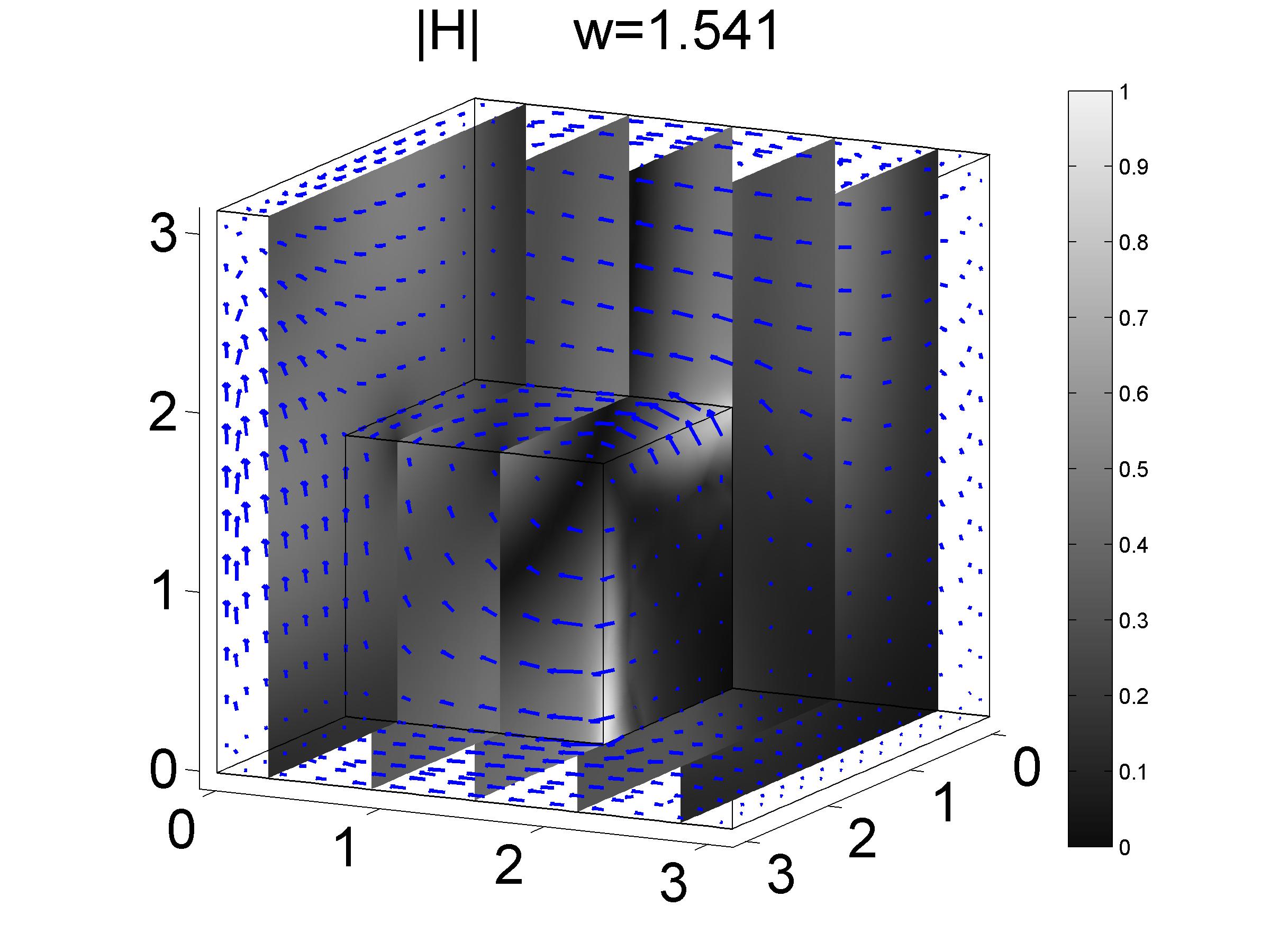} 
}
\centerline{
\includegraphics[height=5cm, angle=0]{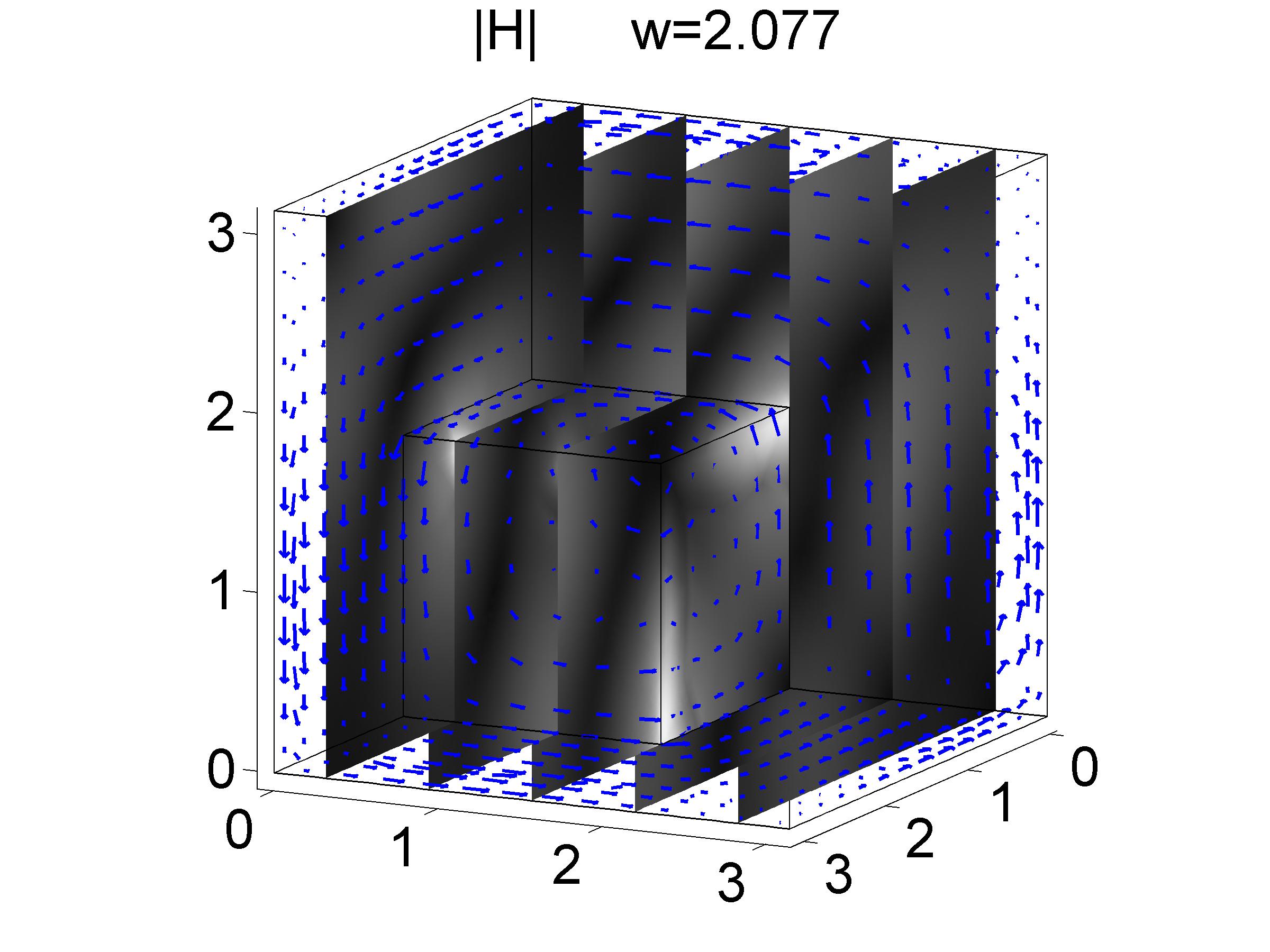} \hspace{-5mm}
\includegraphics[height=5cm, angle=0]{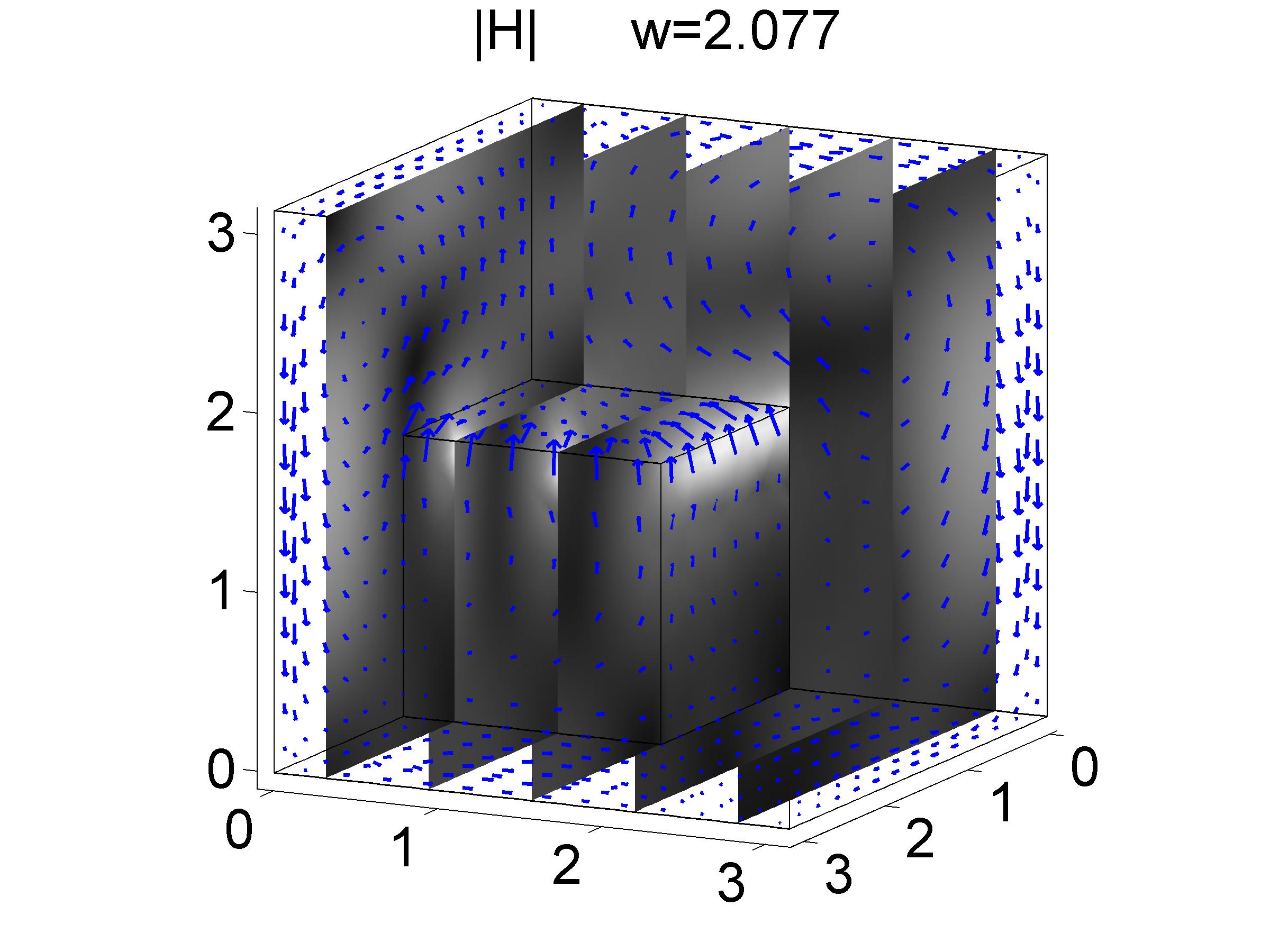} \hspace{-5mm}
\includegraphics[height=5cm, angle=0]{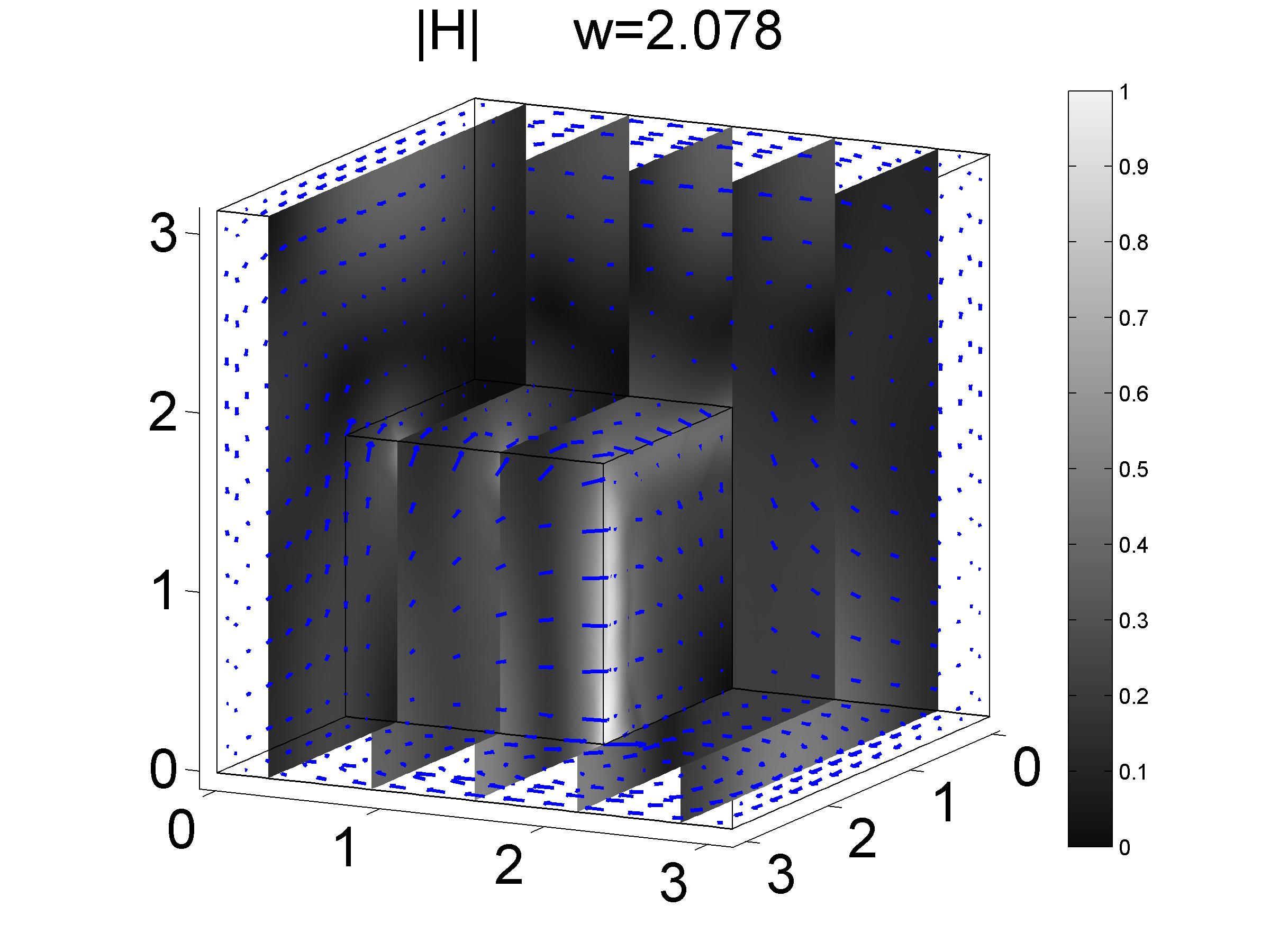} 
}
\caption{ 
The first six eigenfunctions on $\Omega_{\lstd}$ for the first six positive
eigenvalues. Densities $|\bE|$ (top) and $|\bH|$ (bottom). Corresponding arrow
fields $\bE$ (red) and $\bH$ (blue) on $\partial \Omega_{\lstd}$. 
 \label{fichera}
}
\end{figure}

\end{document}